\documentclass{article}
\usepackage[utf8]{inputenc}
\usepackage[utf8]{inputenc}
\usepackage[utf8]{inputenc}
\usepackage{dirtytalk}
\usepackage{tikz}
\usepackage{geometry}
\usepackage{graphicx}
\usepackage{enumitem}
\usepackage{parskip}
\usepackage{amsfonts}
\usepackage{amssymb}
\usepackage{amsmath}
\usepackage{amsthm}
\usepackage{xcolor}
\usepackage{hyperref}
\newlist{myitemize}{itemize}{1}
\setlist[myitemize,1]{leftmargin = 0.5in}
\hypersetup{colorlinks = true, linkcolor = red,  linktocpage}

\theoremstyle{plain}
\newtheorem{thm}{Theorem}[section]
\newtheorem*{thm*}{Theorem}

\newtheorem{cor}[thm]{Corollary}

\theoremstyle{definition}

\newtheorem{conj}[thm]{Conjecture}

\newtheorem{rem}[thm]{Remark}

\title{\textbf{\small{COUNTING THE NUMBER OF $n$-PERIODIC $\mathbb{Z}_{p}$-\&-$\mathbb{F}_{p}[t]$-POINTS OF A DISCRETE DYNAMICAL SYSTEM WITH APPLICATIONS FROM ARITHMETIC STATISTICS, VI}}}
\author{\footnotesize{BRIAN KINTU}}
\date{\small{\textit{April 4, 2026}}}

\begin{document}
\maketitle
\begin{abstract}
\footnotesize{In this follow-up paper, we again inspect a surprising relationship between the set of $n$-periodic points of a polynomial map $\varphi_{d, c}$ defined by $\varphi_{d, c}(z) = z^d + c$ for all  $c, z \in \mathbb{Z}_{p}$ or $\in \mathbb{F}_{p}[t]$ and the coefficient $c$, where $d>2$ is an integer and $n\in \mathbb{Z}_{\geq 2}$ is any fixed (period). As in \cite{BK222} we again wish to study counting problems that are inspired by $n$-torsion point-counting in arithmetic statistics and $n$-periodic point-counting in arithmetic dynamics. In doing so, we then first prove that for any prime $p\geq 3$ and for any fixed $\ell \in \mathbb{Z}_{\geq 1}$ and fixed (period) $n\in \mathbb{Z}_{\geq 2}$, the average number of distinct $n$-periodic $p$-adic integral points of any $\varphi_{p^{\ell}, c}$ modulo $p\mathbb{Z}_{p}$ is unbounded or zero as $c\to \infty$; and then also prove that for any prime $p\geq 5$ and for any fixed $\ell\in \mathbb{Z}_{ \geq 1}$, the average number of distinct $n$-periodic $p$-adic integral points of any $\varphi_{(p-1)^{\ell}, c}$ modulo $p\mathbb{Z}_{p}$ is $1$ or $2$ or $0$ as $c\to \infty$; and so the average behavior here coincide with the average behavior of the number of distinct fixed $p$-adic integral points in \cite{BK3}. Motivated further by periodic $\mathbb{F}_{p}(t)$-point-counting in arithmetic dynamics, we then also prove that for any prime $p\geq 3$ and for any fixed $\ell \in \mathbb{Z}_{\geq 1}$ and fixed (period) $n\in \mathbb{Z}_{\geq 2}$, the average number of distinct $n$-periodic points of any $\varphi_{p^{\ell}, c}$ modulo prime $\pi$ is unbounded or zero as $c$ varies; and then also prove that for any prime $p\geq 5$ and for any fixed $\ell \in \mathbb{Z}_{\geq 1}$, the average number of distinct $n$-periodic points of any $\varphi_{(p-1)^{\ell}, c}$ modulo $\pi$ is $1$ or $2$ or $0$ as $c$ varies; and so the average behavior here also coincide with the average behavior of the number of distinct fixed points in \cite{BK3}. Finally, we then apply here density, polynomial-counting, field-counting, and equidistribution results from arithmetic statistics, and thereby obtaining counting and statistical results on irreducible monic polynomials, Artin-Mazur zeta functions, global fields, (Artin) $L$-functions, and on zeta functions of global fields that arise naturally in our polynomial discrete dynamical settings.}
\end{abstract}

\begin{center}
\tableofcontents
\end{center}

\begin{center}
    \section{Introduction}\label{sec1}
\end{center}
\noindent
Given any morphism $\varphi: {\mathbb{P}^N(K)} \rightarrow {\mathbb{P}^N(K)} $ of degree $d \geq 2$ defined on a projective space ${\mathbb{P}^N(K)}$ of dimension $N$, where $K$ is a number field. Then for any $n\in\mathbb{Z}$ and $\alpha\in\mathbb{P}^N(K)$, we then call $\varphi^n = \underbrace{\varphi \circ \varphi \circ \cdots \circ \varphi}_\text{$n$ times}$ the $n^{th}$ \textit{iterate of $\varphi$} and call $\varphi^n(\alpha)$ the \textit{$n^{th}$ iteration of $\varphi$ on $\alpha$}. By convention, $\varphi^{0}$ acts as the identity map, i.e., $\varphi^{0}(\alpha) = \alpha$ for every point $\alpha\in {\mathbb{P}^N(K)}$. As before, the everyday philosopher may want to know (quoting here Devaney \cite{Dev}): \say{\textit{Where do points $\alpha, \varphi(\alpha), \varphi^2(\alpha), \ \cdots\ ,\varphi^n(\alpha)$ go as $n$ becomes large, and what do they do when they get there?}} Now for any given integer $n\geq 0$ and any given point $\alpha\in {\mathbb{P}^N(K)}$, we then call the set consisting of all the iterates $\varphi^n(\alpha)$ the \textit{(forward) orbit of $\alpha$}; and which in dynamical systems we usually denote it by $\mathcal{O}^{+}(\alpha)$.

As mentioned in \cite{BK222} that one of the main 
goals in arithmetic dynamics (a newly emerging area of mathematics concerned with studying number-theoretic properties of discrete dynamical systems) is to classify all the points $\alpha\in\mathbb{P}^N(K)$ according to the behavior of their forward orbits $\mathcal{O}^{+}(\alpha)$. In this direction, we recall that any point $\alpha\in {\mathbb{P}^N(K)}$ is called a \textit{periodic point of $\varphi$}, whenever $\varphi^n (\alpha) = \alpha$ for some integer $n\in \mathbb{Z}_{\geq 0}$. In this case, we recall that any integer $n\geq 0$ such that the iterate $\varphi^n (\alpha) = \alpha$, is called \textit{period of $\alpha$}; and the smallest such positive integer $n\geq 1$ is called the \textit{exact period of $\alpha$}. We recall Per$(\varphi, {\mathbb{P}^N(K)})$ to denote set of all periodic points of $\varphi$; and also recall that for any given point $\alpha\in$Per$(\varphi, {\mathbb{P}^N(K)})$ the set of all iterates of $\varphi$ on $\alpha$ is called \textit{periodic orbit of $\alpha$}. In their 1994 paper \cite{Russo} and in his 1998 paper \cite{Poonen} respectively, Walde-Russo and Poonen give independently interesting examples of rational periodic points of any $\varphi_{2,c}$ defined over the field $\mathbb{Q}$; and so the interested reader may wish to revisit \cite{Russo, Poonen} to gain familiarity with the notion of periodicity of points. 

Previously in article \cite{BK3} we (inspired by work of Bhargava-Shankar-Tsimerman (BST) and their collaborators in arithmetic statistics (a branch of number theory concerned with counting and distribution of arithmetic objects), and also of Adam-Fares \cite{Ada} in arithmetic dynamics) proved that the number of distinct fixed $p$-adic integral points of any $\varphi_{p^{\ell},c}$ modulo $p\mathbb{Z}_{p}$ (for every $\ell \in \{1, p\})$ is equal to $p$ or zero; from which it then followed that the average number of distinct fixed $p$-adic integral points of any $\varphi_{p^{\ell},c}$ modulo $p\mathbb{Z}_{p}$ (for every $\ell \in \{1, p\})$ is unbounded or zero as $c\to \infty$. Moreover, we then also observed in \cite{BK3} that the expected total number of fixed $p$-adic integral points in the whole family of maps $\varphi_{p^{\ell},c}$ modulo $p\mathbb{Z}_{p}$ (for every $\ell \in \{1, p\})$ is equal to $p+0=p$; which may grow to infinity whenever degree $p^{\ell}\to \infty$. Later in [\cite{BK222}, Corollary 2.4] we (inspired by work of Artin-Mazur \cite{AM} on periodic orbits and of (BST) on $n$-torsion point-counting in arithmetic statistics, along with conjectural work \ref{per} of Morton-Silverman in arithmetic dynamics) proved that the number of distinct $n$-periodic integral points of any $\varphi_{p^{\ell},c}$ modulo $p$ is equal to $p$ or zero; from which it then also followed that the average number of distinct $n$-periodic integral points of any $\varphi_{p^{\ell},c}$ modulo $p$ is unbounded or zero as $c\to \infty$. Moreover, we then also observed in [\cite{BK222}, Remark 2.5] that the expected total number of $n$-periodic integral points in the whole family of maps $\varphi_{p^{\ell},c}$ modulo $p$ is equal to $p+0=p$ (for every fixed period $n\in \mathbb{Z}_{\geq 2}$); which may also grow to infinity when $p^{\ell}\to \infty$. So now, inspired (as in \cite{BK222}) by work of Artin-Mazur \cite{AM} on periodic orbits and also of (BST) on $n$-torsion point-counting in arithmetic statistics, along with work \cite{Ada} of Adam-Fares in arithmetic dynamics, we revisit the settings in \cite{BK3, BK222} and then prove the following main theorem on every map $\varphi_{p,c}$, which we state later more precisely as Theorem \ref{2.2}; and which we then also generalize further as Theorem \ref{2.3}:

\begin{thm}\label{BB2} 
Let $p\geq 3$ be any fixed prime, and $n\geq 2$ be any fixed integer. Let $\varphi_{p, c}$ be defined by $\varphi_{p, c}(z)=z^p + c$ for all $c, z\in\mathbb{Z}_{p}$. The number of distinct $n$-periodic $p$-adic integral points of $\varphi_{p,c}$ modulo $p\mathbb{Z}_{p}$ is $p$ or zero. 
\end{thm}

Recall further in that same \cite{BK3} we (again inspired by work of (BST) in arithmetic statistics, and also of Adam-Fares \cite{Ada} in arithmetic dynamics) proved that the number of distinct fixed $p$-adic integral points of any $\varphi_{(p-1)^{\ell},c}$ modulo $p\mathbb{Z}_{p}$ is equal to $1$ or $2$ or $0$; from which it then followed that the average number of distinct fixed $p$-adic integral points of any $\varphi_{(p-1)^{\ell},c}$ modulo $p\mathbb{Z}_{p}$ is also $1$ or $2$ or $0$ as $c\to \infty$. Moreover, we then also observed in [\cite{BK3}, Remark 4.4] that the expected total number of distinct fixed $p$-adic integral points in the whole family of maps $\varphi_{(p-1)^{\ell},c}$ modulo $p\mathbb{Z}_{p}$ is a constant equal to $1 + 2 + 0 =3$ even when $(p-1)^{\ell}\to \infty$. Later in article \cite{BK222} we (again inspired by work of Artin-Mazur \cite{AM} and of (BST) in arithmetic statistics, along with conjectural work \ref{conjecture 3.2.1} of Hutz and \textit{abc}(\textit{d})-conditional work [\cite{par2}, Theorem 1.7 and 1.8] of Panraksa in arithmetic dynamics) proved that the number of distinct $n$-periodic integral points of any $\varphi_{(p-1)^{\ell},c}$ modulo $p$ is equal to $1$ or $2$ or $0$; from which it then followed that the average number of distinct $n$-periodic integral points of any $\varphi_{(p-1)^{\ell},c}$ modulo $p$ is also $1$ or $2$ or $0$ as $c\to \infty$. Moreover, we then also observed in [\cite{BK222}, Corollary 3.4, Remark 3.5 and 3.6] that the expected total number of distinct $n$-periodic integral points in the whole family of maps $\varphi_{(p-1)^{\ell},c}$ modulo $p$ is also a constant equal to $1 + 2 + 0 =3$ for every fixed odd period $n\in \mathbb{Z}_{\geq 1}$ (or equal to $1+1+2+0 = 4$ for every fixed even period $n\in \mathbb{Z}_{\geq 2}$) even when degree $(p-1)^{\ell}\to \infty$. So now, motivated again by work of Artin-Mazur \cite{AM} and of (BST) in arithmetic statistics and also of Adam-Fares \cite{Ada} in arithmetic dynamics, we revisit the setting in Section \ref{sec2} and then prove in Section \ref{sec3} the following main theorem on any map $\varphi_{p-1,c}$, which we state later more precisely as Theorem \ref{3.2}; and then also generalize more as Theorem \ref{3.3}:

\begin{thm}\label{BB3}
Let $p\geq 5$ be any fixed prime, and $n\geq 2$ be any fixed integer. Let $\varphi_{p-1, c}$ be defined by $\varphi_{p-1, c}(z)$ for all $c, z\in\mathbb{Z}_{p}$. The number of distinct $n$-periodic $p$-adic integral points of $\varphi_{p-1,c}$ modulo $p\mathbb{Z}_{p}$ is $1$ or $2$ or zero.
\end{thm}

\noindent Notice that the count obtained in Theorem \ref{BB3} and more precisely in Theorem \ref{3.2} on the number of distinct $n$-periodic $p$-adic integral points of any $\varphi_{p-1,c}$ modulo $p\mathbb{Z}_{p}$ is independent of $p$ (and so independent of deg$(\varphi_{p-1,c}))$ in each of the possibilities. Moreover, we may also observe that the expected total count (namely, $1 + 2 + 0 =3$ for every fixed odd period $n\in \mathbb{Z}_{\geq 3}$ or $1 + 1 + 2 + 0 =4$ for every fixed even period $n\in \mathbb{Z}_{\geq 2}$) in Theorem \ref{3.2} (and hence in Theorem \ref{BB3}) on the number of distinct $n$-periodic $p$-adic integral points in the whole family of maps $\varphi_{p-1,c}$ modulo $p\mathbb{Z}_{p}$ is not only also independent of $p$ (and so independent of deg$(\varphi_{p-1,c})$), but is also a constant equal to $3$ or $4$ even when $p-1\to \infty$. On the other hand, we may also notice that the count obtained in Theorem \ref{BB2} on the number of distinct $n$-periodic $p$-adic integral points of any $\varphi_{p,c}$ modulo $p\mathbb{Z}_{p}$ may depend on $p$ (and hence on deg$(\varphi_{p,c})$) in one of the  possibilities. Consequently, the expected total count (namely, $p+0 =p$ for every fixed period $n\in \mathbb{Z}_{\geq 2}$) in Theorem \ref{BB2} on number of distinct $n$-periodic $p$-adic integral points in the whole family of maps $\varphi_{p,c}$ modulo $p\mathbb{Z}_{p}$ may not only depend on $p$, but also may grow to infinity when $p\to \infty$.

Previously in again article \cite{BK3} we (greatly motivated by a \say{counting-application} philosophy in arithmetic statistics and function field number theory, and also by periodic $\mathbb{F}_{p}(t)$-point-counting result of Benedetto in arithmetic dynamics restated here in Theorem \ref{main}) proved that the number of distinct fixed $\mathbb{F}_{p}[t]$-points of any $\varphi_{p^{\ell},c}$ modulo prime $\pi$ (for every $\ell \in \{1, p\})$ is equal to $p$ or zero; from which it then followed that the average number of distinct fixed $\mathbb{F}_{p}[t]$-points of any $\varphi_{p^{\ell},c}$ modulo $\pi$ (for every $\ell \in \{1, p\})$ is unbounded or zero as deg$(c)\to \infty$. Moreover, we then also observed in \cite{BK3} that the expected total number of distinct fixed $\mathbb{F}_{p}[t]$-points in the whole family of maps $\varphi_{p^{\ell},c}$ modulo $\pi$ (for every $\ell \in \{1, p\})$ is equal to $p+0=p$; and which may grow to infinity when degree $p^{\ell}\to \infty$. So now, motivated again by that same \say{counting-application} philosophy in arithmetic statistics and function field number theory, and also by that same periodic $\mathbb{F}_{p}(t)$-point-counting result \ref{main} of Benedetto in arithmetic dynamics, we revisit the setting in \cite{BK3} and then prove the following main theorem on any $\varphi_{p,c}$, which we state later more precisely as Theorem \ref{4.2}; and generalize more as Theorem \ref{4.3}:

\begin{thm}\label{BB4} 
Let $p\geq 3$ be any fixed prime integer, $n\geq 2$ be any fixed integer, and let $\pi\in \mathbb{F}_{p}[t]$ be any fixed irreducible monic polynomial of degree $m\geq 1$. Let $\varphi_{p, c}$  be a polynomial map defined by $\varphi_{p, c}(z) = z^p + c$ for all $c, z\in\mathbb{F}_{p}[t]$. Then the number of distinct $n$-periodic points of any polynomial map $\varphi_{p,c}$ modulo $\pi$ is $p$ or zero. 
\end{thm}

Recall further in \cite{BK222} we (again motivated by that same \say{counting-application} philosophy in arithmetic statistics and function field number theory, along with that same periodic $\mathbb{F}_{p}(t)$-point-counting result in Theorem \ref{main} of Benedetto in arithmetic dynamics) proved that the number of distinct fixed $\mathbb{F}_{p}[t]$-points of any $\varphi_{(p-1)^{\ell},c}$ modulo prime $\pi$ is equal to $1$ or $2$ or $0$; from which it then followed that the average number of distinct fixed $\mathbb{F}_{p}[t]$-points of any $\varphi_{(p-1)^{\ell},c}$ modulo $\pi$ is also equal to $1$ or $2$ or $0$ as deg$(c)\to \infty$. Moreover, we then also observed in [\cite{BK3}, Remark 6.4] that the expected total number of distinct fixed $\mathbb{F}_{p}[t]$-points in the whole family of maps $\varphi_{(p-1)^{\ell},c}$ modulo $\pi$ is a constant equal to $1 + 2 + 0 =3$ even when degree $(p-1)^{\ell}\to \infty$. So now, motivated again by that same \say{counting-application} philosophy in arithmetic statistics and function field number theory, along with that same periodic $\mathbb{F}_{p}(t)$-point-counting result in Theorem \ref{main} of Benedetto in arithmetic dynamics, we revisit the setting in Section \ref{sec4} and then also prove in Section \ref{sec5} the following main theorem on any $\varphi_{p-1,c}$, which we state later more precisely as Theorem \ref{5.2}; and which we then also generalize further as Theorem \ref{5.3}:

\begin{thm}\label{BB5}
Let $p\geq 5$ be any fixed prime integer, $n\geq 2$ be any fixed integer, and let $\pi\in \mathbb{F}_{p}[t]$ be any fixed irreducible monic polynomial of degree $m\geq 1$. Let $\varphi_{p-1, c}$ be a polynomial map defined by $\varphi_{p-1, c}(z) = z^{p-1} + c$ for all $c, z\in\mathbb{F}_{p}[t]$. Then the number of distinct $n$-periodic points of any map $\varphi_{p-1,c}$ modulo $\pi$ is $1$ or $2$ or zero.
\end{thm}

\noindent As before, we may also notice that the count obtained in Theorem \ref{BB5} and more precisely in Theorem \ref{5.2} on the number of distinct $n$-periodic points of any $\varphi_{p-1,c}$ modulo $\pi$ is independent of $p$ (and so independent of deg$(\varphi_{p-1,c})$) and $m=$ deg$(\pi)$ in each of the possibilities considered. Moreover, we may also observe that the expected total count (namely, $1 + 2 + 0 =3$ for every fixed odd period $n\in \mathbb{Z}_{\geq 3}$ or $1 + 1 + 2 + 0 =4$ for every fixed even period $n\in \mathbb{Z}_{\geq 2}$) in Theorem \ref{5.2} (and hence in Theorem \ref{BB5}) on the number of distinct $n$-periodic points in the whole family of maps $\varphi_{p-1,c}$ modulo $\pi$ is not only also independent of $p$ (and hence independent of deg$(\varphi_{p-1,c})$) and $m$, but is also a constant $3$ or $4$ even when degree $p-1\to \infty$ or $m\to \infty$. On the other hand, we may notice that the count obtained in Theorem \ref{BB4} on the number of distinct $n$-periodic points of any $\varphi_{p,c}$ modulo $\pi$ may depend on $p$ (and so on deg($\varphi_{p,c}))$ in one of the possibilities. As a result, the expected total count (namely, $p+0 =p$ for every fixed period $n\in \mathbb{Z}_{\geq 2}$) in Theorem \ref{BB4} on the number of distinct $n$-periodic points in the whole family of maps $\varphi_{p,c}$ modulo $\pi$ may not only also depend on $p$, but also may grow to infinity as $p\to \infty$. Mind you, this same phenomena may also occur in $\mathcal{O}_{K}$-setting as noted in \cite{BK222} and here in $\mathbb{Z}_{p}$-setting.

Recall in \cite{BK3} that in addition to the notion of a periodic point and periodic orbit, a point $\alpha\in {\mathbb{P}^N(K)}$ is called a \textit{preperiodic point of $\varphi$}, whenever the iterate $\varphi^{m+n}(\alpha) = \varphi^{m}(\alpha)$ for some integers $m\geq 0$ and $n\geq 1$. In this case, we recall that the smallest integers $m\geq 0$ and $n\geq 1$ such that $\varphi^{m+n}(\alpha) = \varphi^{m}(\alpha)$ happens, are called the \textit{preperiod} and \textit{eventual period of $\alpha$}, respectively. Again, we denote the set of preperiodic points of $\varphi$ by PrePer$(\varphi, {\mathbb{P}^N(K)})$. For any given preperiodic point $\alpha$ of $\varphi$, we then call the set of all iterates of $\varphi$ on $\alpha$, \textit{the preperiodic orbit of $\alpha$}.
Now observe for $m=0$, we have $\varphi^{n}(\alpha) = \alpha $ and so $\alpha$ is a periodic point of period $n$. Thus, the set  Per$(\varphi, {\mathbb{P}^N(K)}) \subseteq$ PrePer$(\varphi, {\mathbb{P}^N(K)})$; however, it need not be PrePer$(\varphi, {\mathbb{P}^N(K)})\subseteq$ Per$(\varphi, {\mathbb{P}^N(K)})$. In their 2014 paper \cite{Doyle}, Doyle-Faber-Krumm give nice examples (which also recover examples in Poonen's paper \cite{Poonen}) of preperiodic points of any quadratic map $\varphi$ (where $\varphi$ is not necessarily the somewhat mostly studied $\varphi_{2,c}$ in arithmetic dynamics) defined over quadratic fields; and so the interested reader may wish to revisit \cite{Poonen, Doyle}.

Inspired further by that same \say{counting-application} philosophy in arithmetic statistics and function field number theory, and also by work of Doyle-Poonen \cite{DoyPo} along with conjectural work \ref{conjecture 3.2.1} of Hutz and \textit{abc}(\textit{d})-conditional work [\cite{par2}, Thm 1.7 and 1.8] of Panraksa on preperiodic point-counting in arithmetic dynamics, we revisit the settings \cite{BK111, BK222} and in this article; and then prove (in follow-up works \cite{BK3rm}) counts and asymptotics on preperiodic points; which are very analogous to counts and asymptotics proved in odd degree, but surprisingly different from counts and asymptotics proved in even degree polynomial settings in \cite{BK111, BK222} and in this article.

In the year 1950, Northcott \cite{North} used the theory of height functions to show that not only is the set PrePer$(\varphi, {\mathbb{P}^N(K)})$ always finite, but also for a given morphism $\varphi$ the set PrePer$(\varphi, {\mathbb{P}^N(K)})$ can be computed effectively. Forty-five years later, in the year 1995, Morton and Silverman conjectured that PrePer$(\varphi, \mathbb{P}^N(K))$ can be bounded in terms of degree $d$ of $\varphi$, degree $D$ of $K$, and dimension $N$ of the space ${\mathbb{P}^N(K)}$. This celebrated conjecture is called the \textit{Uniform Boundedness Conjecture}; which we then restate here as the following conjecture:

\begin{conj} \label{silver-morton}[\cite{Morton}]
Fix integers $D \geq 1$, $N \geq 1$, and $d \geq 2$. There exists a constant $C'= C'(D, N, d)$ such that for all number fields $K/{\mathbb{Q}}$ of degree at most $D$, and all morphisms $\varphi: {\mathbb{P}^N}(K) \rightarrow {\mathbb{P}^N}(K)$ of degree $d$ defined over $K$, the total number of preperiodic points of a morphism $\varphi$ is at most $C'$, i.e., \#PrePer$(\varphi, \mathbb{P}^N(K)) \leq C'$.
\end{conj}
\noindent A special case of Conjecture \ref{silver-morton} is when the degree $D$ of a number field $K$ is $D = 1$, dimension $N$ of a space $\mathbb{P}^N(K)$ is $N = 1$, and degree $d$ of a morphism $\varphi$ is $d = 2$. In this case, if $\varphi$ is a polynomial morphism, then it is a quadratic map defined over the field $\mathbb{Q}$. Moreover, in this very special case, in the year 1995, Flynn and Poonen and Schaefer conjectured that a quadratic map has no points $z\in\mathbb{Q}$ with exact period more than 3. This conjecture of Flynn-Poonen-Schaefer \cite{Flynn} (which has been resolved for cases $n = 4$, $5$ in \cite{mor, Flynn} respectively and conditionally for $n=6$ in \cite{Stoll} is, however, still open for all integers $n\geq 7$ and moreover, which also Hutz-Ingram \cite{Ingram} gave strong computational evidence supporting it) is restated here formally as the following conjecture. Note that in this same special case, rational points of exact period $n\in \{1, 2, 3\}$ were first found in the year 1994 by Russo-Walde \cite{Russo} and also found in the year 1995 by Poonen \cite{Poonen} using a different set of techniques. We now restate the anticipated conjecture of Flynn-Poonen-Schaefer as the following conjecture:
 
\begin{conj} \label{conj:2.4.1}[\cite{Flynn}, Conjecture 2]
If $n \geq 4$, then there is no quadratic polynomial $\varphi_{2,c }(z) = z^2 + c\in \mathbb{Q}[z]$ with a rational point of exact period $n$.
\end{conj}
Now by assuming Conjecture \ref{conj:2.4.1} and also establishing interesting results on preperiodic points, in the year 1998, Poonen \cite{Poonen} then concluded that the total number of rational preperiodic points of any quadratic polynomial $\varphi_{2, c}(z)=z^2 +c$ is at most nine. We restate here formally Poonen's result as the following corollary:
\begin{cor}\label{cor2}[\cite{Poonen}, Corollary 1]
If Conjecture \ref{conj:2.4.1} holds, then $\#$PrePer$(\varphi_{2,c}, \mathbb{Q}) \leq 9$,  for all quadratic maps $\varphi_{2, c}$ defined by $\varphi_{2, c}(z) = z^2 + c$ for all points $c, z\in\mathbb{Q}$.
\end{cor}

On still the same note of exact periods and pre(periodic) points, the next natural question that one could ask is whether the aforementioned phenomenon on exact periods and pre(periodic) points has been investigated in some other cases, namely, when $D\geq 2$, $N\geq 1$ and $d\geq 2$. In the case $D = d = 2$ and $N = 1$, then again if $\varphi$ is a polynomial map, then $\varphi$ is a quadratic map defined over a quadratic field $K = \mathbb{Q}(\sqrt{D'})$. In this case, in the years 1900, 1998 and 2006, Netto \cite{Netto}, Morton-Silverman \cite{Morton} and Erkama \cite{Erkama} resp., found independently a parametrization of a point $c$ in the field $\mathbb{C}$ of all complex points which guarantees $\varphi_{2,c}$ to have periodic points of period $M=4$. And moreover when $c\in \mathbb{Q}$, Panraksa \cite{par1} showed that one gets \textit{all} orbits of length $M = 4$ defined over $\mathbb{Q}(\sqrt{D'})$. For $M=5$, Flynn-Poonen-Schaefer \cite{Flynn} found a parametrization of a point $c\in \mathbb{C}$ that yields points of period 5; however, these periodic points are not in $K$, but rather in some other extension of $\mathbb{Q}$. In the same case $D = d = 2$ and $N = 1$, Hutz-Ingram \cite{Ingram} and Doyle-Faber-Krumm \cite{Doyle} did not find in their computational investigations points $c\in K$ for which $\varphi_{2,c}$ defined over $K$ has $K$-rational points of exact period $M = 5$. Note that to say that the above authors didn't find points $c\in K$ for which $\varphi_{2,c}$ has $K$-rational points of exact period $M=5$, is not the same as saying that such points do not exist; since it's possible that the techniques which the authors employed in their computational investigations may have been far from enabling them to decide concretely whether such points exist or not. In fact, as of the present article, we do not know whether $\varphi_{2,c}$ has $K$-rational points of exact period $5$ or not, but surprisingly from \cite{Flynn, Stoll, Ingram, Doyle} we know that for $c=-\frac{71}{48}$ and $D'=33$ the map $\varphi_{2,c}$ defined over $K = \mathbb{Q}(\sqrt{33})$ has $K$-rational points of exact period $M = 6$; and mind you, this is the only example of $K$-rational points of exact period $M=6$ that is currently known of in the whole literature of arithmetic dynamics. For $M>6$, in 2013, Hutz-Ingram [\cite{Ingram}, Prop. 2 and 3] gave strong computational evidence which showed that for any absolute discriminant $D'$ at most 4000 and any $c\in K$ with a certain logarithmic height, the map $\varphi_{2,c}$ defined over any $K$ has no $K$-rational points of exact period greater than 6. Moreover, the same authors \cite{Ingram} also showed that the smallest upper bound on the size of PrePer$(\varphi_{2,c}, K)$ is 15. A year later, in 2014, Doyle-Faber-Krumm \cite{Doyle} also gave computational evidence on 250000 pairs $(K, \varphi_{2,c})$ which not only established the same claim [\cite{Doyle}, Thm 1.2] as that of Hutz-Ingram \cite{Ingram} on the upper bound of the size of PrePer$(\varphi_{2,c}, K)$, but it also covered Poonen's claims in \cite{Poonen} on $\varphi_{2,c}$ over $\mathbb{Q}$. Three years later, in 2018, Doyle \cite{Doy} adjusted the computations in his aforementioned cited work with Faber and Krumm; and after which he then made the following conjecture on any quadratic map over any $K = \mathbb{Q}(\sqrt{D'})$:

\begin{conj}\label{do}[\cite{Doy}, Conjecture 1.4]
Let $K\slash \mathbb{Q}$ be a quadratic field and let $f\in K[z]$ be a quadratic polynomial. \newline Then, $\#$PrePer$(f, K)\leq 15$.
\end{conj}

Recall in \cite{BK33} we attempted to understand (on the level of rings $\mathbb{Z}_{p}$ and $\mathbb{F}_{p}[t]$ independently) the possibility and validity of periodic version of Conjecture \ref{silver-morton}. In this article, we again wish to continue with this attempt of hoping to understand (again on the level of rings $\mathbb{Z}_{p}$ and $\mathbb{F}_{p}[t]$ independently) the possibility and validity of  periodic version of \ref{silver-morton}. That is, in Section \ref{sec2}, \ref{sec3}, \ref{sec4} and \ref{sec5} we consider polynomial maps of any odd degree $p^{\ell}$ and of any even degree $(p-1)^{\ell}\geq 4$ defined independently over $K$ replaced with $\mathbb{Z}_{p}$ and over $K$ replaced with $\mathbb{F}_{p}[t]$; all of this again done in the attempt of understanding the possibility and validity of the following version \ref{per}:

\begin{conj} \label{silver-morton 1}($(D,1)$-version of Conjecture \ref{silver-morton})\label{per}
Fix integers $D \geq 1$ and $d \geq 2$. There exists a constant $C'= C'(D, d)$ such that for all number fields $K/{\mathbb{Q}}$ of degree at most $D$, and all morphisms $\varphi: {\mathbb{P}}^1(K) \rightarrow {\mathbb{P}}^1(K)$ of degree $d$ over $K$, the total number of periodic points of a morphism $\varphi$ is at most $C'$, i.e., \#Per$(\varphi, \mathbb{P}^1(K)) \leq C'$.
\end{conj}

\subsection*{History on the Connection Between the Size of Per$(\varphi_{d, c}, K)$ and the Coefficient $c$}

In the year 1994, Walde and Russo not only proved [\cite{Russo}, Corollary 4] that for a quadratic map $\varphi_{2,c}$ defined over $\mathbb{Q}$ with a periodic point, the denominator of a rational point $c$, denoted as den$(c)$, is a square but they also proved that den$(c)$ is even, whenever $\varphi_{2,c}$ admits a rational cycle of length $\ell \geq 3$. Moreover, Walde-Russo also proved [\cite{Russo}, Cor. 6, Thm 8 and Cor. 7] that the size \#Per$(\varphi_{2, c}, \mathbb{Q})\leq 2$, whenever den$(c)$ is an odd integer. 

Three years later, in the year 1997, Call-Goldstine \cite{Call} proved that the size of PrePer$(\varphi_{2,c},\mathbb{Q})$ can be bounded above in terms of the number of distinct odd primes dividing den$(c)$. We restate formally this result of Call-Goldstine as the following theorem, in which $GCD(a, e)$ refers to the greatest common divisor of $a$, $e \in \mathbb{Z}$:

\begin{thm}\label{2.3.1}[\cite{Call}, Theorem 6.9]
Let $e>0$ be an integer and let $s$ be the number of distinct odd prime factors of e. Define $\varepsilon  = 0$, $1$, $2$, if $4\nmid e$, if $4\mid e$ and $8 \nmid e$, if $8 \mid e$, respectively. Let $c = a/e^2$, where $a\in \mathbb{Z}$ and $GCD(a, e) = 1$. If $c \neq -2$, then the total number of $\mathbb{Q}$-preperiodic points of $\varphi_{2, c}$ is at most $2^{s + 2 + \varepsilon} + 1$. Moreover, a quadratic map $\varphi_{2, -2}$ has exactly six rational preperiodic points.
\end{thm}

Eight years later, after the work of Call-Goldstine, in the year 2005, Benedetto \cite{detto} studied polynomial maps $\varphi$ of arbitrary degree $d\geq 2$ defined over an arbitrary global field $K$, and then established the following result on the relationship between the size of the set PrePre$(\varphi, K)$ and the number of bad primes of $\varphi$ in $K$:

\begin{thm}\label{main} [\cite{detto}, Main Theorem]
Let $K$ be a global field, $\varphi\in K[z]$ be a polynomial of degree $d\geq 2$ and $s$ be the number of bad primes of $\varphi$ in $K$. The number of preperiodic points of $\varphi$ in $\mathbb{P}^N(K)$ is at most $O(\text{s log s})$. 
\end{thm}
 
\noindent Since we know that Theorem \ref{main} applies to any $\varphi$ of degree $d\geq 2$ defined over a global function field $K$ (say, $K=\mathbb{F}_{p}(t)$), this also means that we can apply Theorem \ref{main} to any $\varphi$ of odd or even degree $d>2$ defined over $K$; and as a result obtain the upper bound predicted in Theorem \ref{main} on the number of $K$-preperiodic points. 

Seven years after the work of Benedetto, in the year 2012, Narkiewicz's work \cite{Narkie} not only showed that any $\varphi_{d,c}$ defined over $\mathbb{Q}$ with odd degree $d\geq 3$ has no rational periodic points of exact period $n > 1$, but also showed that the total number of $\mathbb{Q}$-preperiodic points is at most 4. We restate this result here as the following: 

\begin{thm} \label{theorem 3.2.1}\cite{Narkie}
For any integer $n > 1$ and any odd integer $d\geq 3$, there is no $c\in \mathbb{Q}$ such that $\varphi_{d,c}$ defined by $\varphi_{d, c}(z)$ for all $c,z \in \mathbb{Q}$ has rational periodic points of exact period $n$. Moreover, $\#PrePer(\varphi_{d, c}, \mathbb{Q}) \leq 4$. 
\end{thm} 

Seven years later, after some work of Benedetto and other several authors working on non-archimedean dynamics, in the year 2012, Adam-Fares [\cite{Ada}, Proposition 15] studied the dynamical system $(K, \ x^{p^{\ell}}+c)$ where $K$ is a local field equipped with a discrete valuation and $\ell\in \mathbb{Z}_{\geq 1}$. In the case $K = \mathbb{Q}_{p}$, they showed that the polynomial $\varphi_{p^{\ell},c}(x) = x^{p^{\ell}} + c$ where $c\in \mathbb{Z}_{p}$, either has $p$ fixed points or a periodic orbit of exact period $p$ in $\mathbb{Q}_{p}$.

Three years after \cite{Narkie}, in 2015, Hutz \cite{Hutz} developed an algorithm determining effectively all $\mathbb{Q}$-preperiodic points of a morphism defined over a given number field $K$; from which he then made the following conjecture: 

\begin{conj} \label{conjecture 3.2.1}[\cite{Hutz}, Conjecture 1a]
For any integer $n > 2$, there is no even degree $d > 2$ and no point $c \in \mathbb{Q}$ such that the polynomial map $\varphi_{d, c}$ has rational points of exact period $n$.
Moreover, \#PrePer$(\varphi_{d, c}, \mathbb{Q}) \leq 4$. 
\end{conj}

\noindent On the note whether any theoretical progress has yet been made on Conjecture \ref{conjecture 3.2.1}, more recently, Panraksa \cite{par2} proved among many other results that the quartic polynomial $\varphi_{4,c}(z)\in\mathbb{Q}[z]$ has rational points of exact period $n = 2$. Moreover, he also proved that $\varphi_{d,c}(z)\in\mathbb{Q}[z]$ has no rational points of exact period $n = 2$ for any $c \in \mathbb{Q}$ with $c \neq -1$ and $d = 6$, $2k$ with $3 \mid 2k-1$. The interested reader may find these mentioned results of Panraksa in his unconditional Thms 2.1, 2.4 and also see his Thm 1.7 conditioned on the \textit{abc}-conjecture in \cite{par2}.

Twenty-eight years later, after the work of Walde-Russo, in the year 2022, Eliahou-Fares proved [\cite{Shalom2}, Theorem 2.12] that the denominator of a rational point $-c$, denoted as den$(-c)$ is divisible by 16, whenever $\varphi_{2,-c}$ defined by $\varphi_{2, -c}(z) = z^2 - c$ for all $c, z\in \mathbb{Q}$ admits a rational cycle of length $\ell \geq 3$. Moreover, they also proved [\cite{Shalom2}, Proposition 2.8] that the size \#Per$(\varphi_{2, -c}, \mathbb{Q})\leq 2$, whenever den$(-c)$ is an odd integer. Motivated by \cite{Call}, Eliahou-Fares \cite{Shalom2} also proved that the size of Per$(\varphi_{2, -c}, \mathbb{Q})$ can be bounded above by using information on den$(-c)$, namely, information in terms of the number of distinct primes dividing den$(-c)$. Moreover, they in \cite{Shalom1} also showed that the upper bound is four, whenever $c\in \mathbb{Q^*} = \mathbb{Q}\setminus\{0\}$. We restate here their results as:

\begin{cor}\label{sha}[\cite{Shalom2, Shalom1}, Cor. 3.11 and Cor. 4.4, respectively]
Let $c\in \mathbb{Q}$ such that den$(c) = d^2$ with $d\in 4 \mathbb{N}$. Let $s$ be the number of distinct primes dividing $d$. Then, the total number of $\mathbb{Q}$-periodic points of $\varphi_{2, -c}$ is at most $2^s + 2$. Moreover, for $c\in \mathbb{Q^*}$ such that the den$(c)$ is a power of a prime number. Then, $\#$Per$(\varphi_{2, c}, \mathbb{Q}) \leq 4$.
\end{cor}

\noindent Once again, the purpose of this follow-up article is to inspect further (using elementary arguments) the above connection in the case of polynomial maps $\varphi_{p^{\ell}, c}$ and $\varphi_{(p-1)^{\ell},c}$ defined independently, first, over the ring $\mathbb{Z}_{p}$ of all $p$-adic integers and then over a polynomial ring $\mathbb{F}_{p}[t]$ for any given prime $p> 2$ and for any given $\ell\in \mathbb{Z}_{\geq 1}$; and doing all of it from a spirit that is inspired by some of the many striking developments in arithmetic statistics.

\section{On the Number of $n$-Periodic $\mathbb{Z}_{p}\slash p\mathbb{Z}_{p}$-Points of any Family of Polynomial Maps $\varphi_{p^{\ell},c}$}\label{sec2}

In this section, we count the number of distinct $n$-periodic $p$-adic integral points of any $\varphi_{p^{\ell},c}$ modulo   $p\mathbb{Z}_{p}$, where $p\geq 3$ is any prime, $\ell\in \mathbb{Z} _{\geq 1}$ and $n\in \mathbb{Z}_{\geq 2}$ are any fixed integers. To do so, we let $p\geq 3$ be any prime, $\ell\in \mathbb{Z} _{\geq 1}$, $n\in \mathbb{Z}_{\geq 2}$ be any fixed integers, $c\in \mathbb{Z}_{p}$ be any $p$-adic integer, and then define $n$-periodic point-counting function 
\begin{equation}\label{X_{c}}
X_{c}^{(n)}(p) := \# \Biggl\{ z\in \mathbb{Z}_{p}\slash p\mathbb{Z}_{p}   : \begin{aligned} \varphi_{p^{\ell},c}^{n-1}(z) -z \not \equiv 0 \ \text{(mod $p\mathbb{Z}_{p}$)} \\ \ \varphi_{p^{\ell},c}^{n}(z) - z \equiv 0 \ \text{(mod $p\mathbb{Z}_{p}$)} \end{aligned} \Biggr\}.
\end{equation}\noindent Setting $\ell =1$, so that the map $\varphi_{p^{\ell}, c} = \varphi_{p,c}$, we then first prove the following theorem and its generalization \ref{2.2}:

\begin{thm} \label{2.1}
Let $\varphi_{3, c}$ be a cubic map defined by $\varphi_{3, c}(z) = z^3 + c$ for all $c, z\in\mathbb{Z}_{3}$, and $X_{c}^{(n)}(3)$ be the number defined as in \textnormal{(\ref{X_{c}})}. Then $X_{c}^{(n)}(3) = 3$ for every coefficient $c\in 3\mathbb{Z}_{3}$; otherwise $X_{c}^{(n)}(3) = 0$ for every point $c \not \in  3\mathbb{Z}_{3}$.
\end{thm}
\begin{proof}
Let $f(z)= \varphi_{3,c}^n(z)-z = \varphi_{3,c}(\varphi_{3,c}^{n-1}(z)) - z = (\varphi_{3,c}^{n-1}(z))^3 - z + c$, and so $f(z)= (\varphi_{3,c}^{n-1}(z))^3 - z + c$. Now applying the multinomial theorem repeatedly on the term $(\varphi_{3,c}^{n-1}(z))^3$ right after applying the binomial theorem on $(z^3 + c)^3$, we then obtain $(\varphi_{3,c}^{n-1}(z))^3$ is a monic polynomial in $z$ of degree $3^n$ with $3$-adic integer coefficients in multiples of $c$. Thus, we may then write $(\varphi_{3,c}^{n-1}(z))^3 = z^{3^{n}} + h(z)$, where $h(z)$ is a non-constant polynomial in $z$ of deg$(h)<3^n$ with $3$-adic integer coefficients in multiples of $c$; and so $f(z)= z^{3^{n}} + h(z) - z + c$. Now for every coefficient $c\in 3\mathbb{Z}_{3}$, reducing $f(z)$ modulo prime ideal $3\mathbb{Z}_{3}$, it then follows $f(z)\equiv z^{3^n} - z$ (mod $3\mathbb{Z}_{3}$), since also $h(z)\in c\mathbb{Z}_{3}[z]$ and so $h(z)\equiv 0$ (mod $3\mathbb{Z}_{3}$); and so now the reduced polynomial $f(z)$ modulo $3\mathbb{Z}_{3}$ is a polynomial defined over a finite field $\mathbb{Z}_{3}\slash 3\mathbb{Z}_{3}$ of order $3$. So now, recall from a well-known fact that the cubic polynomial $h(x)=x^3-x$ vanishes at every $z\in \mathbb{Z}_{3}\slash 3\mathbb{Z}_{3}$; and so $z^3 = z$ for every $z\in \mathbb{Z}_{3}\slash3\mathbb{Z}_{3}$. But then we note $z^{3^n}= (z^3)^{3^{n-1}} = (z^3)^{3^{n-2}} = z^{3^{n-2}}$ for every $z\in \mathbb{Z}_{3}\slash3\mathbb{Z}_{3}$. Now since $n\geq 2$ and so $n-2\geq 0$, then if $n-2 = 0$ and so $z^{3^{n-2}} = z$, it then follows $z^{3^n} = z$  for every $z\in \mathbb{Z}_{3}\slash3\mathbb{Z}_{3}$; and so $f(z)\equiv 0$ (mod $3\mathbb{Z}_{3}$) for every point $z\in \mathbb{Z}_{3}\slash3\mathbb{Z}_{3}$. Otherwise, if $n-2 > 0$, then since $n$ is a fixed integer, we may continue performing the above procedure of decreasing the exponent $n-2$ of $z^{3^{n-2}}= z^{3^n}$ for every $z\in \mathbb{Z}_{3}\slash3\mathbb{Z}_{3}$ until $n-2$ is equal to zero; from which we again obtain $f(z)\equiv 0$ (mod $3\mathbb{Z}_{3}$) for every point $z\in \mathbb{Z}_{3}\slash3\mathbb{Z}_{3}$. But now we then conclude $X_{c}^{(n)}(3) = 3$. 

Finally, we now show $X_{c}^{(n)}(3) = 0$ for every coefficient $c\not \equiv 0$ (mod $3\mathbb{Z}_{3}$) and for every fixed integer $n\in \mathbb{Z}_{\geq 2}$. To see this, let's for the sake of a contradiction, suppose $f(z)\equiv 0$ (mod $3\mathbb{Z}_{3}$) for some $z\in \mathbb{Z}_{3}\slash 3\mathbb{Z}_{3}$ and for every coefficient $c\not \equiv 0$ (mod $3\mathbb{Z}_{3}$) and every fixed $n\in \mathbb{Z}_{\geq 2}$. Now since from earlier $f(z)= z^{3^{n}} + h(z) - z + c$ where $h(z)\in c\mathbb{Z}_{3}[z]$, it then also follows $z^{3^{n}} + h(z) - z + c \equiv 0$ (mod $3\mathbb{Z}_{3}$) for some $z\in \mathbb{Z}_{3}\slash 3\mathbb{Z}_{3}$ and for every $c\not \equiv 0$ (mod $3\mathbb{Z}_{3}$) and every fixed $n$. So now, using an earlier observation that $z^{3^n} = z$ for every $z\in \mathbb{Z}_{3}\slash3\mathbb{Z}_{3}$ and every fixed $n\in \mathbb{Z}_{\geq 2}$, we may then rewrite $z^{3^{n}} -z + h(z)  + c \equiv 0$ (mod $3\mathbb{Z}_{3}$) for some $z\in \mathbb{Z}_{3}\slash 3\mathbb{Z}_{3}$ and for every $c\not \equiv 0$ (mod $3\mathbb{Z}_{3}$) to then obtain $h(z) + c \equiv 0$ (mod $3\mathbb{Z}_{3}$) for some $z\in \mathbb{Z}_{3}\slash 3\mathbb{Z}_{3}$ and for every $c\not \equiv 0$ (mod $3\mathbb{Z}_{3}$). Now looking at the multinomial expansion of $(\varphi_{3,c}^{n-1}(z))^3$, we then obtain $h(z)\equiv \sum_{i=1}^{n-1}c^{3^{n-i}}$ (mod $3\mathbb{Z}_{3}$); and so $h(z) + c \equiv \sum_{i=1}^{n-1}c^{3^{n-i}} + c$ (mod $3\mathbb{Z}_{3}$) and so $\sum_{i=1}^{n-1}c^{3^{n-i}} + c \equiv 0$ (mod $3\mathbb{Z}_{3}$). 
But now we note that $\sum_{i=1}^{n-1}c^{3^{n-i}} + c \equiv 0$ (mod $3\mathbb{Z}_{3}$) can happen if $\sum_{i=1}^{n-1}c^{3^{n-i}}\equiv 0$ (mod $3\mathbb{Z}_{3}$) and also $c\equiv 0$ (mod $3\mathbb{Z}_{3}$); and so follows a contradiction. This then means  $f(x)=\varphi_{3,c}^n(x)-x$ has no roots in $\mathbb{Z}_{3}\slash 3\mathbb{Z}_{3}$ for every coefficient $c\not \in 3\mathbb{Z}_{3}$ and every fixed $n\in \mathbb{Z}_{\geq 2}$; and so we conclude $X_{c}^{(n)}(3) = 0$. This then completes the whole proof, as needed.
\end{proof} 
We now wish to generalize Theorem \ref{2.1} to any polynomial map $\varphi_{p, c}$ for any prime $p\geq 3$. More precisely, we prove that the number of distinct $n$-periodic $p$-adic integral points of any  $\varphi_{p, c}$ modulo $p\mathbb{Z}_{p}$ is either $p$ or zero:

\begin{thm} \label{2.2}
Let $p\geq 3$ be any fixed prime, and $\varphi_{p, c}$ be defined by $\varphi_{p, c}(z) = z^p + c$ for all $c, z\in\mathbb{Z}_{p}$. Let $X_{c}^{(n)}(p)$ be as in \textnormal{(\ref{X_{c}})}. Then $X_{c}^{(n)}(p) = p$ for every coefficient $c\in p\mathbb{Z}_{p}$; otherwise $X_{c}^{(n)}(p) = 0$ for every $c \not \in p\mathbb{Z}_{p}$.
\end{thm}
\begin{proof}
By applying a similar argument as in the Proof of Theorem \ref{2.1}, we then obtain the count as desired. That is, let $f(z)= \varphi_{p,c}^n(z)-z = \varphi_{p,c}(\varphi_{p,c}^{n-1}(z)) - z = (\varphi_{p,c}^{n-1}(z))^p - z + c$, and so $f(z)= (\varphi_{p,c}^{n-1}(z))^p - z + c$. Now that applying the multinomial theorem repeatedly on $(\varphi_{p,c}^{n-1}(z))^p$ right after applying the binomial theorem on $(z^p + c)^p$, it then follows $(\varphi_{p,c}^{n-1}(z))^p$ is a monic polynomial in $z$ of degree $p^n$ with $p$-adic integer coefficients in multiples of $c$. Hence, we may then write $(\varphi_{p,c}^{n-1}(z))^p = z^{p^{n}} + h(z)$, where $h(z)$ is a non-constant polynomial in $z$ of deg$(h)<p^n$ with $p$-adic integer coefficients in multiples of $c$; and so $f(z)= z^{p^{n}} + h(z) - z + c$. Now for every coefficient $c\in p\mathbb{Z}_{p}$, reducing $f(z)$ modulo prime ideal $p\mathbb{Z}_{p}$, it then follows $f(z)\equiv z^{p^n} - z$ (mod $p\mathbb{Z}_{p}$), since also  $h(z)\in c\mathbb{Z}_{p}[z]$ and so $h(z)\equiv 0$ (mod $p\mathbb{Z}_{p}$); and so $f(z)$ modulo $p\mathbb{Z}_{p}$ is now a polynomial defined over a finite field $\mathbb{Z}_{p}\slash p\mathbb{Z}_{p}$ of order $p$. So now, recall a well-known fact about polynomials over finite fields that the monic polynomial $h(x)=x^p-x$ vanishes every $z \in \mathbb{Z}_{p}\slash p \mathbb{Z}_{p}$; and so $z^p = z$ for every $z\in \mathbb{Z}_{p}\slash p\mathbb{Z}_{p}$. But then we note $z^{p^n}= (z^p)^{p^{n-1}} = (z^p)^{p^{n-2}} = z^{p^{n-2}}$ for every $z\in \mathbb{Z}_{p}\slash p\mathbb{Z}_{p}$. So now, since $n\geq 2$ and so $n-2\geq 0$, then if $n-2 = 0$ and so $z^{p^{n-2}} = z$, then this yields $z^{p^n} = z$  for every $z\in \mathbb{Z}_{p}\slash p\mathbb{Z}_{p}$; and so $f(z)\equiv 0$ (mod $p\mathbb{Z}_{p}$) for every point $z\in \mathbb{Z}_{p}\slash p\mathbb{Z}_{p}$. Otherwise, if $n-2 > 0$, then since $n$ is a fixed integer, we may continue performing the above procedure of decreasing $n-2$ of $z^{p^{n-2}}=z^{p^n}$ for every $z\in \mathbb{Z}_{p}\slash p\mathbb{Z}_{p}$ until $n-2$ is equal to zero; from which we then again obtain $f(z)\equiv 0$ (mod $p\mathbb{Z}_{p}$) for every point $z\in \mathbb{Z}_{p}\slash p\mathbb{Z}_{p}$. But now we then conclude $X_{c}^{(n)}(p) = p$. 

Finally, we now show $X_{c}^{(n)}(p) = 0$ for every coefficient $c \not \in p\mathbb{Z}_{p}$ and for every fixed integer $n\in \mathbb{Z}_{\geq 2}$. As before, let's for the sake of a contradiction, suppose $f(z)\equiv 0$ (mod $p\mathbb{Z}_{p}$) for some point $z\in \mathbb{Z}_{p}\slash p\mathbb{Z}_{p}$ and for every $c\not \equiv 0$ (mod $p\mathbb{Z}_{p}$) and every fixed $n\in \mathbb{Z}_{\geq 2}$. So then, since from earlier $f(z)= z^{p^{n}} + h(z) - z + c$ where $h(z)\in c\mathbb{Z}_{p}[z]$, it then also follows $z^{p^{n}} + h(z) - z + c \equiv 0$ (mod $p\mathbb{Z}_{p}$) for some $z\in \mathbb{Z}_{p}\slash p\mathbb{Z}_{p}$ and for every $c\not \equiv 0$ (mod $p\mathbb{Z}_{p}$) and every fixed $n$. Now recall from earlier that $z^{p^n} = z$ for every $z\in \mathbb{Z}_{p}\slash p\mathbb{Z}_{p}$ and every fixed $n\in \mathbb{Z}_{\geq 2}$, we may then rewrite $z^{p^{n}} -z + h(z) + c \equiv 0$ (mod $p\mathbb{Z}_{p}$) for some $z\in \mathbb{Z}_{p}\slash p\mathbb{Z}_{p}$ and for every $c\not \equiv 0$ (mod $p\mathbb{Z}_{p}$) to obtain $h(z) + c \equiv 0$ (mod $p\mathbb{Z}_{p}$) for some $z\in \mathbb{Z}_{p}\slash p\mathbb{Z}_{p}$ and for every $c\not \equiv 0$ (mod $p\mathbb{Z}_{p}$). Moreover, looking at the multinomial expansion of $(\varphi_{p,c}^{n-1}(z))^p$, we then obtain $h(z)\equiv \sum_{i=1}^{n-1}c^{p^{n-i}}$ (mod $p\mathbb{Z}_{p}$); and so $h(z) + c \equiv \sum_{i=1}^{n-1}c^{p^{n-i}} + c$ (mod $p\mathbb{Z}_{p}$) and so $\sum_{i=1}^{n-1}c^{p^{n-i}} + c \equiv 0$ (mod $p\mathbb{Z}_{p}$). But now we note that $\sum_{i=1}^{n-1}c^{p^{n-i}} + c \equiv 0$ (mod $p\mathbb{Z}_{p}$) can happen if $\sum_{i=1}^{n-1}c^{p^{n-i}}\equiv 0$ (mod $p\mathbb{Z}_{p}$) and also $c\equiv 0$ (mod $p\mathbb{Z}_{p}$); and so follows a contradiction. This then means $f(x)=\varphi_{p,c}^n(x)-x$ has no roots in $\mathbb{Z}_{p}\slash p\mathbb{Z}_{p}$ for every coefficient $c\not \in p\mathbb{Z}_{p}$ and every fixed $n\in \mathbb{Z}_{\geq 2}$; and so we conclude $X_{c}^{(n)}(p) = 0$. This then completes the whole proof, as needed.
\end{proof}

Finally, we wish to generalize Theorem \ref{2.2} further to any $\varphi_{p^{\ell}, c}$ for any prime $p\geq 3$ and any $\ell \in \mathbb{Z}_{\geq 1}$. That is, we prove the number of distinct $n$-periodic $p$-adic integral points of any $\varphi_{p^{\ell}, c}$ modulo $p\mathbb{Z}_{p}$ is $p$ or zero:

\begin{thm}\label{2.3}
Let $p\geq 3$ be any fixed prime integer, and $\ell \geq 1$ be any fixed integer. Let $\varphi_{p^{\ell}, c}$ be a polynomial map defined by $\varphi_{p^{\ell}, c}(z) = z^{p^{\ell}} + c$ for all points $c, z\in\mathbb{Z}_{p}$, and $X_{c}^{(n)}(p)$ be the number defined as in \textnormal{(\ref{X_{c}})}. Then we have $X_{c}^{(n)}(p) = p$  for every coefficient $c\in p\mathbb{Z}_{p}$; otherwise we have $X_{c}^{(n)}(p) = 0$ for every coefficient $c\not \in p\mathbb{Z}_{p}$.
\end{thm}

\begin{proof}
By applying the same counting argument as in the Proof of Theorem \ref{2.3}, we then obtain the count as desired. That is, let $f(z)= \varphi_{p^{\ell},c}^n(z)-z = (\varphi_{p^{\ell},c}^{n-1}(z))^{p^{\ell}} - z + c$, and so $f(z)= (\varphi_{p^{\ell},c}^{n-1}(z))^{p^{\ell}} - z + c$. So now, applying the multinomial theorem repeatedly on $(\varphi_{p^{\ell},c}^{n-1}(z))^{p^{\ell}}$ after applying the binomial theorem on $(z^{p^{\ell}} + c)^{p^{\ell}}$, it then follows $(\varphi_{p^{\ell},c}^{n-1}(z))^{p^{\ell}}$ is a monic polynomial in $z$ of degree $p^{n\ell}$ with $p$-adic integer coefficients in multiples of $c$. Hence, we may write $(\varphi_{p^{\ell},c}^{n-1}(z))^{p^{\ell}} = z^{p^{n\ell}} + h(z)$, where $h(z)$ is a non-constant polynomial in $z$ of deg$(h)<p^{n\ell}$ with $p$-adic integer coefficients in multiples of $c$; and so the reduced polynomial $f(z)= z^{p^{n\ell}} + h(z) - z + c$. Now for every coefficient $c\in p\mathbb{Z}_{p}$, reducing $f(z)$ modulo $p\mathbb{Z}_{p}$, it then follows $f(z)\equiv z^{p^{n\ell}} - z$ (mod $p\mathbb{Z}_{p}$), since also $h(z)\in c\mathbb{Z}_{p}[z]$ and so $h(z)\equiv 0$ (mod $p\mathbb{Z}_{p}$); and so $f(z)$ modulo $p\mathbb{Z}_{p}$ is now a polynomial defined over a finite field $\mathbb{Z}_{p}\slash p\mathbb{Z}_{p}$. Now recall (as a fact) $z^p = z$ for every $z\in \mathbb{Z}_{p}\slash p\mathbb{Z}_{p}$, it then follows $z^{p^{n\ell}}= (z^p)^{p^{n\ell-1}} = (z^p)^{p^{n\ell-2}} = z^{p^{n\ell-2}}$ for every $z\in \mathbb{Z}_{p}\slash p\mathbb{Z}_{p}$. So now, since $n\ell\geq 2$ and so $n\ell-2\geq 0$, then if $n\ell-2 = 0$ and so $z^{p^{n\ell-2}} = z$, then this yields $z^{p^{n\ell}} = z$ for every $z\in \mathbb{Z}_{p}\slash p\mathbb{Z}_{p}$; and so $f(z)\equiv 0$ (mod $p\mathbb{Z}_{p}$) for every point $z\in \mathbb{Z}_{p}\slash p\mathbb{Z}_{p}$. Otherwise, if $n\ell-2 > 0$, then since $n\ell$ is a fixed integer, we may then  continue performing the above procedure of decreasing the exponent $n\ell-2$ 
of $z^{p^{n\ell-2}} = z^{p^{n\ell}}$ for every element $z\in \mathbb{Z}_{p}\slash p\mathbb{Z}_{p}$ until $n\ell-2$ is equal to zero; and from we then again obtain $f(z)\equiv 0$ (mod $p\mathbb{Z}_{p}$) for every point $z\in \mathbb{Z}_{p}\slash p\mathbb{Z}_{p}$. Hence, as before we then conclude $X_{c}^{(n)}(p) = p$. 

Finally, we now show $X_{c}^{(n)}(p) = 0$ for every coefficient $c \not \in p\mathbb{Z}_{p}$ and for every fixed integers $\ell\in \mathbb{Z}_{\geq 1}$ and $n\in \mathbb{Z}_{\geq 2}$. As before, let's for the sake of a contradiction, suppose $f(z)\equiv 0$ (mod $p\mathbb{Z}_{p}$) for some point $z\in \mathbb{Z}_{p}\slash p\mathbb{Z}_{p}$ and for every coefficient $c\not \equiv 0$ (mod $p\mathbb{Z}_{p}$) and every fixed $\ell\in \mathbb{Z}_{\geq 1}$ and $n\in \mathbb{Z}_{\geq 2}$. So now, since from earlier $f(z)= z^{p^{n\ell}} + h(z) - z + c$ where $h(z)\in c\mathbb{Z}_{p}[z]$, it then also follow $z^{p^{n\ell}} + h(z) - z + c \equiv 0$ (mod $p\mathbb{Z}_{p}$) for some point $z\in \mathbb{Z}_{p}\slash p\mathbb{Z}_{p}$ and for every $c\not \equiv 0$ (mod $p\mathbb{Z}_{p}$) and for every fixed $\ell$ and $n$. Moreover, recalling from earlier that $z^{p^{n\ell}} = z$ for every element $z\in \mathbb{Z}_{p}\slash p\mathbb{Z}_{p}$ and for every fixed $\ell$ and $n$, we may rewrite $z^{p^{n\ell}} - z + h(z)  + c \equiv 0$ (mod $p\mathbb{Z}_{p}$) for some $z\in \mathbb{Z}_{p}\slash p\mathbb{Z}_{p}$ and for every $c\not \equiv 0$ (mod $p\mathbb{Z}_{p}$) to obtain $h(z) + c \equiv 0$ (mod $p\mathbb{Z}_{p}$) for some $z\in \mathbb{Z}_{p}\slash p\mathbb{Z}_{p}$ and for every $c\not \equiv 0$ (mod $p\mathbb{Z}_{p}$). So now, looking at the multinomial expansion of the term $(\varphi_{p^{\ell},c}^{n-1}(z))^{p^{\ell}}$, we then obtain $h(z)\equiv \sum_{i=1}^{n-1}c^{p^{n\ell-i}}$ (mod $p\mathbb{Z}_{p}$); 
and so $h(z) + c \equiv \sum_{i=1}^{n-1}c^{p^{n\ell-i}} + c $ (mod $p\mathbb{Z}_{p}$) and so $\sum_{i=1}^{n-1}c^{p^{n\ell-i}} + c  \equiv 0$ (mod $p\mathbb{Z}_{p}$). But now we note that the congruence $\sum_{i=1}^{n-1}c^{p^{n\ell-i}} + c \equiv 0$ (mod $p\mathbb{Z}_{p}$) can also happen whenever $\sum_{i=1}^{n-1}c^{p^{n\ell-i}}  \equiv 0$ (mod $p\mathbb{Z}_{p}$) and also $ c \equiv 0$ (mod $p\mathbb{Z}_{p}$); and from which follows a contradiction. This then means that $f(x)=\varphi_{p^{\ell},c}^n(x)-x$ has no roots in $\mathbb{Z}_{p}\slash p\mathbb{Z}_{p}$ for every coefficient $c\not \in p\mathbb{Z}_{p}$ and for every fixed $\ell\in \mathbb{Z}_{\geq 1}$ and every fixed $n\in \mathbb{Z}_{\geq 2}$; and so we then conclude $X_{c}^{(n)}(p) = 0$. This then completes the whole proof, as desired.
\end{proof}

\begin{rem}\label{rem2.4}
With now Theorem \ref{2.3}, we may then to each $n$-periodic point of any $\varphi_{p^{\ell},c}$ associate an $n$-periodic orbit. In doing so, we then obtain a dynamical translation of Theorem \ref{2.3}, namely, that the number of distinct $n$-periodic $p$-adic integral orbits that any $\varphi_{p^{\ell},c}$ has when iterated on the space $\mathbb{Z}_{p} \slash p\mathbb{Z}_{p}$ is $p$ or zero. As mentioned in Intro.\ref{sec1} that the count obtained in Theorem \ref{2.3} on the number of distinct $n$-periodic $p$-adic integral points of any $\varphi_{p^{\ell},c}$ modulo $p\mathbb{Z}_{p}$ may depend only on $p$ (and so depend on deg$(\varphi_{p^{\ell},c})$ for every fixed $\ell\in \mathbb{Z}_{\geq 1}$) in one of the possibilities; or the count may be independent of $p$ (and so independent on deg$(\varphi_{p^{\ell},c})$ for every fixed $\ell$). Moreover, the expected total count (namely, $p+0 = p$ for every fixed period $n\in \mathbb{Z}_{\geq 2}$) of the number of distinct $n$-periodic $p$-adic integral points in the whole family of maps $\varphi_{p^{\ell},c}$ modulo $p\mathbb{Z}_{p}$ may not only also depend on $p$ (and so depend on deg$(\varphi_{p^{\ell},c})$ for every fixed $\ell\in \mathbb{Z}_{\geq 1}$), but also may grow to infinity when degree $p^{\ell}\to \infty$. Consequently, we may then have $X_{c}^{(n)}(p)\to \infty$ or $X_{c}^{(n)}(p)\to 0$ as $p\to \infty$; a somewhat interesting phenomenon differing significantly from a phenomenon observed in Remark \ref{rem3.4}, however, coinciding somewhat surprising with a phenomenon observed in [\cite{BK3}, Remark 3.4 (for every $\ell \in \{1, p\}$)] and also currently here in Remark \ref{rem4.4}.
\end{rem}

\begin{rem}\label{rem2.5}
Recall in \cite{BK3} that the fixed point-counting function $X_{c}(p) = p$ (for every $\ell \in \{1,p\}$) or $0$ for every fixed prime $p\geq 3$ and every coefficient $c\equiv 0$ (mod $p\mathbb{Z}_{p}$) or $c\not \equiv 0$ (mod $p\mathbb{Z}_{p}$). Moreover, recall in Theorem \ref{2.3} that for any fixed (period) $n\in \mathbb{Z}_{\geq 2}$, the $n$-periodic point-counting function $X_{c}^{(n)}(p) = p$ or $0$ for every fixed $p$ and every $c\equiv 0$ (mod $p\mathbb{Z}_{p}$) or $c\not \equiv 0$ (mod $p\mathbb{Z}_{p}$). But now for every fixed (period) $n\in \mathbb{Z}_{\geq 2}$ and every $\ell \in \{1,p\}$, we then note $X_{c}^{(n)}(p) = X_{c}(p) = p$ or $0$ for every fixed $p$ and every $c\equiv 0$ (mod $p\mathbb{Z}_{p}$) or $c\not \equiv 0$ (mod $p\mathbb{Z}_{p}$). Moreover, for every coefficient $c\equiv 0$ (mod $p\mathbb{Z}_{p}$) and for every $\ell \in \{1, p\}$, it also follow from [\cite{BK3}, Proof of Theorem 3.3] and Proof of Theorem \ref{2.3} that every $n$-periodic $p$-adic integral point (and hence every $n$-periodic $p$-adic integral orbit) of any $\varphi_{p^{\ell},c}$ modulo $p\mathbb{Z}_{p}$ is a fixed $p$-adic integral point (and hence a fixed $p$-adic integral orbit). This also means that for every fixed (period) $n\in \mathbb{Z}_{\geq 1}$, every $n$-periodic orbit of $\varphi_{p^{\ell},c}$ modulo $p\mathbb{Z}_{p}$ is a fixed orbit, and moreover every reduced map $\varphi_{p^{\ell},c}$ modulo $p\mathbb{Z}_{p}$ has exactly $p$ distinct fixed orbits or zero; a somewhat interesting precise arithmetic-geometric insight on all the $n$-periodic orbits of any $\varphi_{p^{\ell},c}$ modulo $p\mathbb{Z}_{p}$. 
\end{rem}

\section{On Number of $n$-Periodic $\mathbb{Z}_{p}\slash p\mathbb{Z}_{p}$-Points of any Family of Polynomial Maps $\varphi_{(p-1)^{\ell},c}$}\label{sec3}

As in Section \ref{sec2}, we in this section also wish to count the number of distinct $n$-periodic $p$-adic integral points of any polynomial map $\varphi_{(p-1)^{\ell},c}$ modulo prime ideal $p\mathbb{Z}_{p}$, where $p\geq 5$ is any prime, $\ell \in \mathbb{Z}_{\geq 1}$ and $n\in \mathbb{Z}_{\geq 2}$ are any fixed integers. With that in mind, we again let $p\geq 5$ be any prime, $\ell \in \mathbb{Z}_{\geq 1}$, $n\in \mathbb{Z}_{\geq 2}$ be any fixed integers, $c\in \mathbb{Z}_{p}$ be any $p$-adic integer, and then define analogously the following $n$-periodic point-counting function
\begin{equation}\label{Y_{c}}
Y_{c}^{(n)}(p) := \# \Biggl\{ z\in \mathbb{Z}_{p}\slash p\mathbb{Z}_{p}   : \begin{aligned} \varphi_{(p-1)^{\ell},c}^{n-1}(z) -z \not \equiv 0 \ \text{(mod $p\mathbb{Z}_{p}$)} \\ \ \varphi_{(p-1)^{\ell},c}^{n}(z) - z \equiv 0 \ \text{(mod $p\mathbb{Z}_{p}$)} \end{aligned} \Biggr\}.
\end{equation}\noindent Again, setting $\ell =1$ and so $\varphi_{(p-1)^{\ell}, c} = \varphi_{p-1,c}$, we first prove the following theorem and its generalization \ref{3.2}:

\begin{thm} \label{3.1}
Let $\varphi_{4, c}$ be a quartic map defined by $\varphi_{4, c}(z) = z^4 + c$ for all $c, z\in\mathbb{Z}_{5}$, and $Y_{c}^{(n)}(5)$ be defined as in \textnormal{(\ref{Y_{c}})}. Then $Y_{c}^{(n)}(5) = 1$ for every $c\equiv \pm 1 \ (mod \ 5\mathbb{Z}_{5})$ and fixed (even) $n$ or $Y_{c}^{(n)}(5) = 2$ for every $c\in 5\mathbb{Z}_{5}$; otherwise $Y_{c}^{(n)}(5) = 0$ for every $c \equiv -1\ (mod \ 5\mathbb{Z}_{5})$ and fixed odd $n$ or $c\not \equiv \pm 1, 0 \ (mod \ 5\mathbb{Z}_{5})$ and fixed (even) $n$. 
\end{thm}
\begin{proof}
Let $g(z)= \varphi_{4,c}^n(z)-z = \varphi_{4,c}(\varphi_{4,c}^{n-1}(z)) - z = (\varphi_{4,c}^{n-1}(z))^4 - z + c$, and so $g(z)= (\varphi_{4,c}^{n-1}(z))^4 - z + c$. Now applying the multinomial theorem repeatedly on the term $(\varphi_{4,c}^{n-1}(z))^4$ right after applying the binomial theorem on $(z^4 + c)^4$, we then obtain that $(\varphi_{4,c}^{n-1}(z))^4$ is a monic polynomial in $z$ of degree $4^n$ with $5$-adic integer coefficients in multiples of $c$. Thus, we may then write $(\varphi_{4,c}^{n-1}(z))^4 = z^{4^{n}} + h(z)$, where $h(z)$ is a non-constant polynomial in $z$ of deg$(h)<4^n$ with $5$-adic integer coefficients in multiples of $c$; and so $g(z)= z^{4^{n}} + h(z) - z + c$. So now, for every coefficient $c\in 5\mathbb{Z}_{5}$, reducing $g(z)$ modulo prime ideal $5\mathbb{Z}_{5}$, it then follows $g(z)\equiv z^{4^n} - z$ (mod $5\mathbb{Z}_{5}$), since also $h(z)\in c\mathbb{Z}_{5}[z]$ and so $h(z)\equiv 0$ (mod $5\mathbb{Z}_{5}$); and so now the reduced polynomial $g(z)$ modulo $5\mathbb{Z}_{5}$ is a polynomial defined over a finite field $\mathbb{Z}_{5}\slash 5\mathbb{Z}_{5}$ of order $5$. Now recall from a well-known fact that the quartic monic polynomial $h(x)=x^4 -1$ vanishes at every element $z\in (\mathbb{Z}_{5}\slash 5\mathbb{Z}_{5})^{\times} = \mathbb{Z}_{5}\slash 5\mathbb{Z}_{5}\setminus \{0\}$; and so $z^4 = 1$ for every element $z\in (\mathbb{Z}_{5}\slash 5\mathbb{Z}_{5})^{\times}$. But now we note $z^{4^n}= (z^4)^{4^{n-1}} = 1$ for every element $z\in (\mathbb{Z}_{5}\slash 5\mathbb{Z}_{5})^{\times}$ and for every fixed $n\in \mathbb{Z}_{\geq 2}$; and so $g(z)\equiv 1 - z$ (mod $5\mathbb{Z}_{5}$) for every point $z\in (\mathbb{Z}_{5}\slash 5\mathbb{Z}_{5})^{\times}$ and so $g(z)$ has a root in $\mathbb{Z}_{5}\slash 5\mathbb{Z}_{5}$, namely, $z\equiv 1$ (mod $5\mathbb{Z}_{5}$). Moreover, since $z$ is a linear factor of $g(z)\equiv z(z^{4^n-1} - 1)$ (mod $5\mathbb{Z}_{5}$), it then follows $z\equiv 0$ (mod $5\mathbb{Z}_{5}$) is also root of $g(z)$ modulo $5\mathbb{Z}_{5}$. But now we then conclude $Y_{c}^{(n)}(5) = 2$. To see $Y_{c}^{(n)}(5) = 1$ for every coefficient $c\equiv 1$ (mod $5$) and for every fixed $n\in \mathbb{Z}_{\geq 2}$, we first note that writing $\varphi_{4,c}^{n-1}(z) = \underbrace{((((z^4 + c)^4 + c)^4 + c)^4 + \cdots + c)^4 + c}_\text{$(n-1)$ times}$ and then reducing $\varphi_{4,c}^{n-1}(z)$ modulo $5\mathbb{Z}_{5}$ along with $c\equiv 1$ (mod $5\mathbb{Z}_{5}$) and also with $z^4 = 1$ for every element $z\in (\mathbb{Z}_{5}\slash 5\mathbb{Z}_{5})^{\times}$, it then follows $\varphi_{4,c}^{n-1}(z)\equiv 2$ (mod $5\mathbb{Z}_{5}$) and $(\varphi_{4,c}^{n-1}(z))^4\equiv 1$ (mod $5\mathbb{Z}_{5}$) for every fixed $n\in \mathbb{Z}_{\geq 2}$. But then $g(z)=(\varphi_{4,c}^{n-1}(z))^4 - z + c\equiv 2-z$ (mod $5\mathbb{Z}_{5}$) for every point $z\in (\mathbb{Z}_{5}\slash 5\mathbb{Z}_{5})^{\times}$ and so $g(z)$ modulo $5\mathbb{Z}_{5}$ has a root in $\mathbb{Z}_{5}\slash5\mathbb{Z}_{5}$, namely, $z\equiv 2$ (mod $5\mathbb{Z}_{5}$); and so we then conclude $Y_{c}^{(n)}(5) = 1$. We now show  $Y_{c}^{(n)}(5) = 1$ for every coefficient $c\equiv -1$ (mod $5\mathbb{Z}_{5}$) and for every fixed even integer $n\in \mathbb{Z}_{\geq 2}$. As before, we note that since $c\equiv -1$ (mod $5\mathbb{Z}_{5}$) and also since $z^4 = 1$ for every $z\in (\mathbb{Z}_{5}\slash 5\mathbb{Z}_{5})^{\times}$, reducing $\varphi_{4,c}^n(z)$ modulo $5\mathbb{Z}_{5}$, it then follows  $\varphi_{4,c}^{n}(z)\equiv -1$ (mod $5\mathbb{Z}_{5}$) for every fixed even $n\in \mathbb{Z}_{\geq 2}$. But then $g(z)= \varphi_{4,c}^n(z)-z \equiv -(1+z)$ (mod $5\mathbb{Z}_{5}$) for every point $z\in (\mathbb{Z}_{5}\slash 5\mathbb{Z}_{5})^{\times}$ and every fixed even $n$ and so $g(z)$ modulo $5\mathbb{Z}_{5}$ has a root in $\mathbb{Z}_{5}\slash5\mathbb{Z}_{5}$, namely, $z\equiv -1$ (mod $5\mathbb{Z}_{5}$); and so we conclude $Y_{c}^{(n)}(5) = 1$. 

Finally, we now show $Y_{c}^{(n)}(5) = 0$ for every coefficient $c \equiv -1$ (mod $5\mathbb{Z}_{5}$) and every fixed odd integer $n\in \mathbb{Z}_{\geq 3}$ or for every coefficient $c\not \equiv \pm1, 0$ (mod $5\mathbb{Z}_{5}$) and fixed (even) integer $n\in \mathbb{Z}_{\geq 2}$. To see this, we note that since $c\equiv -1$ (mod $5\mathbb{Z}_{5}$) and also since $z^4 = 1$ for every $z\in (\mathbb{Z}_{5}\slash 5\mathbb{Z}_{5})^{\times}$, reducing $\varphi_{4,c}^n(z)$ modulo $5\mathbb{Z}_{5}$, it then follows $\varphi_{4,c}^{n}(z)\equiv 0$ (mod $5\mathbb{Z}_{5}$) for every fixed odd $n\in \mathbb{Z}_{\geq 3}$; and so $g(z)= \varphi_{4,c}^n(z)-z \equiv -z$ (mod $5\mathbb{Z}_{5}$) for every point $z\in (\mathbb{Z}_{5}\slash 5\mathbb{Z}_{5})^{\times}$. But now we note $z\equiv 0$ (mod $5\mathbb{Z}_{5}$) is a root of $g(z)$ modulo $5\mathbb{Z}_{5}$ for every coefficient $c\equiv -1$ (mod $5\mathbb{Z}_{5}$) and every fixed odd $n\in \mathbb{Z}_{\geq 3}$, and also $z\equiv 0$ (mod $5\mathbb{Z}_{5}$) is also root of $g(z)$ modulo $5\mathbb{Z}_{5}$ for every coefficient $c\equiv 0$ (mod $5\mathbb{Z}_{5}$) and every fixed odd $n\in \mathbb{Z}_{\geq 3}$, as seen earlier; and from which we then obtain a contradiction that $1 \equiv 0$ (mod $5\mathbb{Z}_{5}$). This then means $g(x)=\varphi_{4,c}^n(x)-x$ has no roots in $\mathbb{Z}_{5} / 5\mathbb{Z}_{5}$ for every coefficient $c \equiv -1$ (mod $5\mathbb{Z}_{5}$) and every fixed odd $n\in \mathbb{Z}_{\geq 3}$; and so we then conclude $Y_{c}^{(n)}(5) = 0$. Otherwise, suppose $g(z)\equiv 0$ (mod $5\mathbb{Z}_{5}$) for some point $z\in (\mathbb{Z}_{5}\slash 5\mathbb{Z}_{5})^{\times}$ and for every coefficient $c\not \equiv \pm1, 0$ (mod $5\mathbb{Z}_{5}$) and every fixed (even) $n\in \mathbb{Z}_{\geq 2}$; and so (from earlier) $z^{4^n} + h(z)-z+c\equiv 0$ (mod $5\mathbb{Z}_{5}$) where $h(z)\in c\mathbb{Z}_{5}[z]$. So then, since $z^{4^n}=1$ for every $z\in (\mathbb{Z}_{5}\slash 5\mathbb{Z}_{5})^{\times}$ and every fixed even $n\in \mathbb{Z}_{\geq 2}$, we may then write $z^{4^n} -z + h(z)+c\equiv 0$ (mod $5\mathbb{Z}_{5}$) to obtain $(1-z) + (h(z)+c)\equiv 0$ (mod $5$). But now we note that the congruence $(1-z) + (h(z)+c)\equiv 0$ (mod $5\mathbb{Z}_{5}$) can happen if $1-z\equiv 0$ (mod $5\mathbb{Z}_{5}$) and also $h(z)+c\equiv 0$ (mod $5\mathbb{Z}_{5}$). Moreover, recall from earlier that $1-z\equiv 0$ (mod $5\mathbb{Z}_{5}$) when $c\equiv 0$ (mod $5\mathbb{Z}_{5}$); which then also contradicts $c\not \equiv \pm1, 0$ (mod $5\mathbb{Z}_{5}$). This then also means $g(x)=\varphi_{4,c}^n(x)-x$ has no roots in $\mathbb{Z}_{5} / 5\mathbb{Z}_{5}$ for every coefficient $c\not \equiv \pm1, 0$ (mod $5\mathbb{Z}_{5}$) and every fixed (even) $n\in \mathbb{Z}_{\geq 2}$; and so we then also conclude $Y_{c}^{(n)}(5) = 0$. This then completes the whole proof, as needed.
\end{proof} 
We now wish to generalize Theorem \ref{3.1} to any map $\varphi_{p-1, c}$ for any given prime $p\geq 5$. More precisely, we prove that the number of distinct $n$-periodic $p$-adic integral points of any  $\varphi_{p-1, c}$ modulo $p\mathbb{Z}_{p}$ is $1$ or $2$ or $0$:

\begin{thm} \label{3.2}
Let $p\geq 5$ be any fixed prime, and $\varphi_{p-1, c}$ be defined by $\varphi_{p-1, c}(z) = z^{p-1}+c$ for all $c, z\in\mathbb{Z}_{p}$. Let $Y_{c}^{(n)}(p)$ be as in \textnormal{(\ref{Y_{c}})}. Then $Y_{c}^{(n)}(p) = 1$ for any $c\equiv \pm 1 \ (mod \ p\mathbb{Z}_{p})$ and fixed (even) $n$ or $Y_{c}^{(n)}(p) = 2$ for all $c \in p\mathbb{Z}_{p}$; else $Y_{c}^{(n)}(p) = 0$ for every $c \equiv -1\ (mod \ p)$ and fixed odd $n$ or $c\not \equiv \pm 1, 0 \ (mod \ p\mathbb{Z}_{p})$ and fixed (even) $n$.
\end{thm}
\begin{proof}
By applying a similar argument as in the Proof of Theorem \ref{3.1}, we then obtain the count as desired. That is, let $g(z)= \varphi_{p-1,c}^n(z)-z = \varphi_{p-1,c}(\varphi_{p-1,c}^{n-1}(z)) - z = (\varphi_{p-1,c}^{n-1}(z))^{p-1} - z + c$, and so $g(z)= (\varphi_{p-1,c}^{n-1}(z))^{p-1} - z + c$. So now, applying the multinomial theorem repeatedly on $(\varphi_{p-1,c}^{n-1}(z))^{p-1}$ right after applying the binomial theorem on $(z^{p-1} + c)^{p-1}$, it then follows $(\varphi_{p-1,c}^{n-1}(z))^{p-1}$ is a monic polynomial in $z$ of degree $(p-1)^n$ with $p$-adic integer coefficients in multiples of $c$. Hence, we may then write $(\varphi_{p-1,c}^{n-1}(z))^{p-1} = z^{(p-1)^{n}} + h(z)$, where $h(z)$ is a non-constant polynomial in $z$ of deg$(h)<(p-1)^n$ with $p$-adic integer coefficients in multiples of $c$; and so $g(z)= z^{(p-1)^{n}} + h(z) - z + c$. Now for every coefficient $c\in p\mathbb{Z}_{p}$, reducing $g(z)$ modulo prime ideal $p\mathbb{Z}_{p}$, it then follows  $g(z)\equiv z^{(p-1)^n} - z$ (mod $p\mathbb{Z}_{p}$), since also $h(z)\in c\mathbb{Z}_{p}[z]$ and so $h(z)\equiv 0$ (mod $p\mathbb{Z}_{p}$); and so now $g(z)$ modulo $p\mathbb{Z}_{p}$ is a polynomial defined over a finite field $\mathbb{Z}_{p}\slash p\mathbb{Z}_{p}$ of order $p$. So now, recall from a well known fact that the monic polynomial $h(x)=x^p-1$ vanishes at every $z\in (\mathbb{Z}_{p}\slash p\mathbb{Z}_{p})^{\times} = \mathbb{Z}_{p}\slash p\mathbb{Z}_{p}\setminus \{0\}$ and so $z^{p-1} = 1$ for every $z\in (\mathbb{Z}_{p}\slash p\mathbb{Z}_{p})^{\times}$, it then also follows $z^{(p-1)^n}= (z^{p-1})^{(p-1)^{n-1}} = 1$ for every element $z\in (\mathbb{Z}_{p}\slash p\mathbb{Z}_{p})^{\times}$ and every fixed $n\in \mathbb{Z}_{\geq 2}$. But then $g(z)\equiv 1 - z$ (mod $p\mathbb{Z}_{p}$) for every point $z\in (\mathbb{Z}_{p}\slash p\mathbb{Z}_{p})^{\times}$; and so $g(z)$ has a root in $\mathbb{Z}_{p}\slash p\mathbb{Z}_{p}$, namely, $z\equiv 1$ (mod $p\mathbb{Z}_{p}$). Moreover, since $z$ is a linear factor of $g(z)\equiv z(z^{(p-1)^n-1} - 1)$ (mod $p\mathbb{Z}_{p}$), it then follows $z\equiv 0$ (mod $p\mathbb{Z}_{p}$) is also root of $g(z)$ modulo $p\mathbb{Z}_{p}$. But now we then conclude $Y_{c}^{(n)}(p) = 2$. To see $Y_{c}^{(n)}(p) = 1$ for every coefficient $c\equiv 1$ (mod $p\mathbb{Z}_{p}$) and for every fixed $n\in \mathbb{Z}_{\geq 2}$, we note that writing $\varphi_{p-1,c}^{n-1}(z) = \underbrace{((((z^{p-1} + c)^{p-1} + c)^{p-1} + c)^{p-1} + \cdots + c)^{p-1} + c}_\text{$(n-1)$ times}$ and then reducing $\varphi_{p-1,c}^{n-1}(z)$ modulo $p\mathbb{Z}_{p}$ along with $c\equiv 1$ (mod $p\mathbb{Z}_{p}$) and also since (as a fact) $z^{p-1} = 1$ for every $z\in (\mathbb{Z}_{p}\slash p\mathbb{Z}_{p})^{\times}$, it then follows $\varphi_{p-1,c}^{n-1}(z)\equiv 2$ (mod $p\mathbb{Z}_{p}$) and $(\varphi_{p-1,c}^{n-1}(z))^{p-1}\equiv 1$ (mod $p\mathbb{Z}_{p}$) for every fixed $n\in \mathbb{Z}_{\geq 2}$, since also $2^{p-1} \equiv 1$ (mod $p\mathbb{Z}_{p}$). But then $g(z)=(\varphi_{p-1,c}^{n-1}(z))^{p-1} - z + c\equiv 2-z$ (mod $p\mathbb{Z}_{p}$) for every point $z\in (\mathbb{Z}_{p}\slash p\mathbb{Z}_{p})^{\times}$ and so $g(z)$ modulo $p\mathbb{Z}_{p}$ has a root in $\mathbb{Z}_{p}\slash p\mathbb{Z}_{p}$, namely, $z\equiv 2$ (mod $p\mathbb{Z}_{p}$); and so we conclude $Y_{c}^{(n)}(p) = 1$. We now show $Y_{c}^{(n)}(p) = 1$ for every coefficient $c\equiv -1$ (mod $p\mathbb{Z}_{p}$) and fixed even integer $n\in \mathbb{Z}_{\geq 2}$. As before, since $c\equiv -1$ (mod $p\mathbb{Z}_{p}$) and also since $z^{p-1} = 1$ for every $z\in (\mathbb{Z}_{p}\slash p\mathbb{Z}_{p})^{\times}$, reducing $\varphi_{p-1,c}^n(z)$ modulo $p\mathbb{Z}_{p}$, it then follows $\varphi_{p-1,c}^{n}(z)\equiv -1$ (mod $p\mathbb{Z}_{p}$) for every fixed even $n$. But then $g(z)= \varphi_{p-1,c}^n(z)-z \equiv -(1+z)$ (mod $p\mathbb{Z}_{p}$) for every point $z\in (\mathbb{Z}_{p}\slash p\mathbb{Z}_{p})^{\times}$ and so $g(z)$ modulo $p\mathbb{Z}_{p}$ has a root in $\mathbb{Z}_{p}\slash p\mathbb{Z}_{p}$, namely, $z\equiv -1$ (mod $p\mathbb{Z}_{p}$); and so we then conclude $Y_{c}^{(n)}(p) = 1$. 

Finally, we now show $Y_{c}^{(n)}(p) = 0$ for every coefficient $c \equiv -1$ (mod $p\mathbb{Z}_{p}$) and every fixed odd integer $n\in \mathbb{Z}_{\geq 3}$ or for every coefficient $c\not \equiv \pm1, 0$ (mod $p$) and every fixed (even) integer $n\in \mathbb{Z}_{\geq 2}$. To see this, we note that since $c\equiv -1$ (mod $p\mathbb{Z}_{p}$) and also since (as a fact) $z^{p-1} = 1$ for every $z\in (\mathbb{Z}_{p}\slash p\mathbb{Z}_{p})^{\times}$, reducing $\varphi_{p-1,c}^n(z)$ modulo $p\mathbb{Z}_{p}$, it then follows $\varphi_{p-1,c}^{n}(z)\equiv 0$ (mod $p\mathbb{Z}_{p}$) for every fixed odd $n\in \mathbb{Z}_{\geq 3}$; and so $g(z)= \varphi_{p-1,c}^n(z)-z \equiv -z$ (mod $p\mathbb{Z}_{p}$) for every point $z\in (\mathbb{Z}_{p}\slash p\mathbb{Z}_{p})^{\times}$. But now we note $z\equiv 0$ (mod $p\mathbb{Z}_{p}$) is a root of $g(z)$ modulo $p\mathbb{Z}_{p}$ for every coefficient $c\equiv -1$ (mod $p\mathbb{Z}_{p}$) and every fixed odd $n\in \mathbb{Z}_{\geq 3}$, and also $z\equiv 0$ (mod $p\mathbb{Z}_{p}$) is also root of $g(z)$ modulo $p\mathbb{Z}_{p}$ for every coefficient $c\equiv 0$ (mod $p\mathbb{Z}_{p}$) and every odd fixed $n\in \mathbb{Z}_{\geq 3}$, as seen earlier; from which then obtain a contradiction that $1  \equiv 0$ (mod $p\mathbb{Z}_{p}$). This then means $g(x)=\varphi_{p-1,c}^n(x)-x$ has no roots in $\mathbb{Z}_{p} / p\mathbb{Z}_{p}$ for every coefficient $c\equiv -1$ (mod $p\mathbb{Z}_{p}$) and fixed odd $n\in \mathbb{Z}_{\geq 3}$; and so we conclude $Y_{c}^{(n)}(p) = 0$. Otherwise, suppose $g(z)\equiv 0$ (mod $p\mathbb{Z}_{p}$) for some $z\in (\mathbb{Z}_{p}\slash p\mathbb{Z}_{p})^{\times}$ and for every coefficient $c\not \equiv \pm1, 0$ (mod $p\mathbb{Z}_{p}$) and fixed (even) $n\in \mathbb{Z}_{\geq 2}$; and so (from earlier) $z^{(p-1)^n} + h(z)-z+c\equiv 0$ (mod $p\mathbb{Z}_{p}$) where $h(z)\in c\mathbb{Z}_{p}[z]$. So now, using (as a fact) $z^{(p-1)^n}=1$ for every $z\in (\mathbb{Z}_{p}\slash p\mathbb{Z}_{p})^{\times}$ and every fixed even $n$, we may then write $z^{(p-1)^n} -z + h(z) +c\equiv 0$ (mod $p\mathbb{Z}_{p}$) to obtain $(1-z) + (h(z)+c)\equiv 0$ (mod $p\mathbb{Z}_{p}$). But now we note $(1-z) + (h(z)+c)\equiv 0$ (mod $p\mathbb{Z}_{p}$) can happen if $1-z\equiv 0$ (mod $p\mathbb{Z}_{p}$) and also $h(z)+c\equiv 0$ (mod $p\mathbb{Z}_{p}$). Moreover, recall from earlier $1-z\equiv 0$ (mod $p\mathbb{Z}_{p}$) when $c\equiv 0$ (mod $p\mathbb{Z}_{p}$); which then also contradicts $c\not \equiv \pm1, 0$ (mod $p\mathbb{Z}_{p}$). This then also means $g(x)=\varphi_{p-1,c}^n(x)-x$ has no roots in $\mathbb{Z}_{p} / p\mathbb{Z}_{p}$ for every coefficient $c\not \equiv \pm1, 0$ (mod $p\mathbb{Z}_{p}$) and every fixed (even) $n\in \mathbb{Z}_{\geq 2}$; from which we then also conclude $Y_{c}^{(n)}(p) = 0$. This then completes the whole proof, as needed.
\end{proof}

Finally, we wish to generalize Theorem \ref{3.2} further to any $\varphi_{(p-1)^{\ell}, c}$ for any prime $p\geq 5$ and any $\ell\in \mathbb{Z}_{\geq 1}$. That is, we prove that the number of distinct $n$-periodic points of any $\varphi_{(p-1)^{\ell}, c}$ modulo $p\mathbb{Z}_{p}$ is also $1$ or $2$ or $0$:

\begin{thm} \label{3.3}
Let $p\geq 5$ be any fixed prime integer, and $\ell \geq 1$ be any integer. Let $\varphi_{(p-1)^{\ell}, c}$ be a polynomial map defined by $\varphi_{(p-1)^{\ell}, c}(z) = z^{(p-1)^{\ell}} + c$ for all $c, z\in\mathbb{Z}_{p}$, and $Y_{c}^{(n)}(p)$ be the number defined as in \textnormal{(\ref{Y_{c}})}. Then $Y_{c}^{(n)}(p) = 1$ for any coefficient $c\equiv \pm 1 \ (mod \ p\mathbb{Z}_{p})$ and fixed (even) $n$ or $Y_{c}^{(n)}(p) = 2$ for any $c\in p\mathbb{Z}_{p}$; otherwise the number $Y_{c}^{(n)}(p) = 0$ for any $c \equiv -1\ (mod \ p\mathbb{Z}_{p})$ and fixed odd $n$ or $c\not \equiv \pm 1, 0 \ (mod \ p\mathbb{Z}_{p})$ and fixed (even) $n$.
\end{thm}

\begin{proof}
By applying the same argument as in the Proof of Theorem \ref{3.3}, we then obtain the count as desired. That is, let $g(z)= \varphi_{(p-1)^{\ell}, c}^n(z)-z = \varphi_{(p-1)^{\ell},c}(\varphi_{(p-1)^{\ell},c}(z))-z$, and so $g(z)= (\varphi_{(p-1)^{\ell},c}^{n-1}(z))^{(p-1)^{\ell}} - z + c$. So now, as before applying the multinomial theorem on $(\varphi_{p-1,c}^{n-1}(z))^{(p-1)^{\ell}}$ right after applying the binomial theorem on $(z^{(p-1)^{\ell}} + c)^{(p-1)^{\ell}}$, it then follows $(\varphi_{(p-1)^{\ell},c}^{n-1}(z))^{(p-1)^{\ell}}$ is a monic polynomial in $z$ of degree $(p-1)^{n\ell}$ with $p$-adic integer coefficients in multiples of $c$. Hence, we may write $(\varphi_{(p-1)^{\ell},c}^{n-1}(z))^{(p-1)^{\ell}} = z^{(p-1)^{n\ell}} + h(z)$, where $h(z)$ is a non-constant polynomial in $z$ of deg$(h)<(p-1)^{n\ell}$ with $p$-adic integer coefficients in multiples of $c$; and so $g(z)= z^{(p-1)^{n\ell}} + h(z) - z + c$. Now for every coefficient $c \in p\mathbb{Z}_{p}$, reducing $g(z)$ modulo $p\mathbb{Z}_{p}$, it then follows $g(z)\equiv z^{(p-1)^{n\ell}} - z$ (mod $p\mathbb{Z}_{p}$); and so now $g(z)$ modulo $p\mathbb{Z}_{p}$ is a polynomial defined over a field $\mathbb{Z}_{p}\slash p\mathbb{Z}_{p}$. Now recall (as a fact) $z^{p-1} = 1$ for every $z\in (\mathbb{Z}_{p}\slash p\mathbb{Z}_{p})^{\times}$, it then also follows $z^{(p-1)^{n\ell}} \equiv 1$ (mod $p\mathbb{Z}_{p}$) for every $z\in (\mathbb{Z}_{p}\slash p\mathbb{Z}_{p})^{\times}$ and for every fixed $\ell$ and $n$. But then $g(z)\equiv 1 - z$ (mod $p\mathbb{Z}_{p}$) for every point $z\in (\mathbb{Z}_{p}\slash p\mathbb{Z}_{p})^{\times}$; and so $g(z)$ modulo $p\mathbb{Z}_{p}$ has a root in $\mathbb{Z}\slash p\mathbb{Z}$, namely, $z\equiv 1$ (mod $p\mathbb{Z}_{p}$). Moreover, since $z$ is a linear factor of $g(z)\equiv z(z^{(p-1)^{n\ell}-1} - 1)$ (mod $p\mathbb{Z}_{p}$), it then follows $z\equiv 0$ (mod $p\mathbb{Z}_{p}$) is also a root of $g(z)$ modulo $p\mathbb{Z}_{p}$ in $\mathbb{Z}_{p}\slash p\mathbb{Z}_{p}$. But now we then conclude $Y_{c}^{(n)}(p) = 2$. To see $Y_{c}^{(n)}(p) = 1$ for every coefficient $c\equiv 1$ (mod $p\mathbb{Z}_{p}$) and for every fixed $\ell \in \mathbb{Z}_{\geq 1}$ and $n\in \mathbb{Z}_{\geq 2}$, we again note that writing $\varphi_{(p-1)^{\ell},c}^{n-1}(z) = \underbrace{((((z^{(p-1)^{\ell}} + c)^{(p-1)^{\ell}} + c)^{(p-1)^{\ell}} + c)^{(p-1)^{\ell}} + \cdots + c)^{(p-1)^{\ell}} + c}_\text{$(n-1)$ times}$ and reducing $\varphi_{(p-1)^{\ell},c}^{n-1}(z)$ modulo $p\mathbb{Z}_{p}$ along with $c\equiv 1$ (mod $p\mathbb{Z}_{p}$) and also $z^{(p-1)^{\ell}} = 1$ for every element $z\in (\mathbb{Z}_{p}\slash p\mathbb{Z}_{p})^{\times}$, we then obtain $\varphi_{(p-1)^{\ell},c}^{n-1}(z)\equiv 2$ (mod $p\mathbb{Z}_{p}$) and $(\varphi_{(p-1)^{\ell},c}^{n-1}(z))^{(p-1)^{\ell}}\equiv 1$ (mod $p\mathbb{Z}_{p}$) for every fixed $\ell$ and $n$, since also $2^{(p-1)^{\ell}} \equiv 1$ (mod $p\mathbb{Z}_{p}$) for every $\ell \in \mathbb{Z}_{\geq 1}$. But then $g(z)=(\varphi_{(p-1)^{\ell},c}^{n-1}(z))^{(p-1)^{\ell}} - z + c\equiv 2-z$ (mod $p\mathbb{Z}_{p}$) for every point $z\in (\mathbb{Z}_{p}\slash p\mathbb{Z}_{p})^{\times}$ and so $g(z)$ modulo $p\mathbb{Z}_{p}$ has a root in $\mathbb{Z}_{p}\slash p\mathbb{Z}_{p}$, namely, $z\equiv 2$ (mod $p\mathbb{Z}_{p}$); and so we conclude $Y_{c}^{(n)}(p) = 1$. We now show $Y_{c}^{(n)}(p) = 1$ for every coefficient $c\equiv -1$ (mod $p\mathbb{Z}_{p}$) and for every fixed $\ell \in \mathbb{Z}_{\geq 1}$ and fixed even integer $n\in \mathbb{Z}_{\geq 2}$. Again, note that since $c\equiv -1$ (mod $p\mathbb{Z}_{p}$) and also since $z^{(p-1)^{\ell}} = 1$ for every $z\in (\mathbb{Z}_{p}\slash p\mathbb{Z}_{p})^{\times}$, reducing $\varphi_{(p-1)^{\ell},c}^n(z)$ modulo $p\mathbb{Z}_{p}$, it then follows $\varphi_{(p-1)^{\ell},c}^{n}(z)\equiv -1$ (mod $p\mathbb{Z}_{p}$) for every fixed even $n$. But then $g(z)= \varphi_{(p-1)^{\ell},c}^n(z)-z \equiv -(1+z)$ (mod $p\mathbb{Z}_{p}$) for every point $z\in (\mathbb{Z}_{p}\slash p\mathbb{Z}_{p})^{\times}$ and fixed even $n$ and so $g(z)$ modulo $p\mathbb{Z}_{p}$ has a root in $\mathbb{Z}_{p}\slash p\mathbb{Z}_{p}$, namely, $z\equiv -1$ (mod $p\mathbb{Z}_{p}$); and so we conclude $Y_{c}^{(n)}(p) = 1$. 

Finally, we now show $Y_{c}^{(n)}(p) = 0$ for every coefficient $c \equiv -1$ (mod $p\mathbb{Z}_{p}$) and every fixed $\ell \in \mathbb{Z}_{\geq 1}$ and fixed odd $n\in \mathbb{Z}_{\geq 3}$ or for every coefficient $c\not \equiv \pm1, 0$ (mod $p\mathbb{Z}_{p}$) and every fixed $\ell \in \mathbb{Z}_{\geq 1}$ and fixed (even) $n\in \mathbb{Z}_{\geq 2}$. Since $c\equiv -1$ (mod $p\mathbb{Z}_{p}$) and also since $z^{(p-1)^{\ell}} = 1$ for every $z\in (\mathbb{Z}_{p}\slash p\mathbb{Z}_{p})^{\times}$, reducing $\varphi_{(p-1)^{\ell},c}^n(z)$ modulo $p\mathbb{Z}_{p}$, we then obtain $\varphi_{(p-1)^{\ell},c}^{n}(z)\equiv 0$ (mod $p\mathbb{Z}_{p}$) for every fixed $\ell$ and odd $n$; and so $g(z)= \varphi_{(p-1)^{\ell},c}^n(z)-z \equiv -z$ (mod $p\mathbb{Z}_{p}$) for every point $z\in (\mathbb{Z}_{p}\slash p\mathbb{Z}_{p})^{\times}$. But now we note $z\equiv 0$ (mod $p\mathbb{Z}_{p}$) is a root of $g(z)$ modulo $p\mathbb{Z}_{p}$ for every coefficient $c\equiv -1$ (mod $p\mathbb{Z}_{p}$) and every fixed $\ell\in \mathbb{Z}_{\geq 1}$ and fixed odd $n\in \mathbb{Z}_{\geq 3}$, and also $z\equiv 0$ (mod $p\mathbb{Z}_{p}$) is also root of $g(z)$ modulo $p\mathbb{Z}_{p}$ for every coefficient $c\equiv 0$ (mod $p\mathbb{Z}_{p}$) and every fixed $\ell\in \mathbb{Z}_{\geq 1}$ and fixed odd $n\in \mathbb{Z}_{\geq 3}$, as seen earlier; and from which we then obtain a contradiction that $1  \equiv 0$ (mod $p\mathbb{Z}_{p}$). This then means $\varphi_{(p-1)^{\ell},c}^n(x)-x$ has no roots in $\mathbb{Z}_{p}\slash p\mathbb{Z}_{p}$ for every coefficient $c\equiv -1$ (mod $p\mathbb{Z}_{p}$) and for every fixed $\ell\in \mathbb{Z}_{\geq 1}$ and fixed odd $n\in \mathbb{Z}_{\geq 3}$; and so we conclude $Y_{c}^{(n)}(p) = 0$. Otherwise, suppose $g(z)\equiv 0$ (mod $p\mathbb{Z}_{p}$) for some point $z\in (\mathbb{Z}_{p}\slash p\mathbb{Z}_{p})^{\times}$ and for every coefficient $c\not \equiv \pm1, 0$ (mod $p\mathbb{Z}_{p}$) and every fixed $\ell\in \mathbb{Z}_{\geq 1}$ and fixed even $n\in \mathbb{Z}_{\geq 2}$; and so (from earlier) $z^{(p-1)^{n\ell}} + h(z)-z+c\equiv 0$ (mod $p\mathbb{Z}_{p}$) where $h(z)\in c\mathbb{Z}_{p}[z]$. So then, using $z^{(p-1)^{n\ell}}=1$ for every $z\in (\mathbb{Z}_{p}\slash p\mathbb{Z}_{p})^{\times}$ and fixed  $\ell \geq 1$ and fixed $n$, we may then write $z^{(p-1)^{n\ell}} -z + h(z)+c\equiv 0$ (mod $p\mathbb{Z}_{p}$) to obtain $(1-z) + (h(z)+c)\equiv 0$ (mod $p\mathbb{Z}_{p}$). But now we note $(1-z) + (h(z)+c)\equiv 0$ (mod $p\mathbb{Z}_{p}$) can happen if $1-z\equiv 0$ (mod $p\mathbb{Z}_{p}$) and also $h(z)+c\equiv 0$ (mod $p\mathbb{Z}_{p}$). Moreover, recall from earlier that $1-z\equiv 0$ (mod $p\mathbb{Z}_{p}$) when $c\equiv 0$ (mod $p\mathbb{Z}_{p}$); which then contradicts $c\not \equiv \pm1, 0$ (mod $p\mathbb{Z}_{p}$). This then also means $g(x)=\varphi_{(p-1)^{\ell},c}^n(x)-x$ has no roots in $\mathbb{Z}_{p}\slash p\mathbb{Z}_{p}$ for every coefficient $c\not \equiv \pm1, 0$ (mod $p\mathbb{Z}_{p}$) and fixed $\ell\in \mathbb{Z}_{\geq 1}$ and fixed even $n\in \mathbb{Z}_{\geq 2}$; and so we also conclude $Y_{c}^{(n)}(p) = 0$. This then completes the whole proof, as desired.
\end{proof}

\begin{rem}\label{rem3.4}
As before, with now Theorem \ref{3.3}, we may then to each distinct $n$-periodic $p$-adic integral point of $\varphi_{(p-1)^{\ell},c}$ associate an $n$-periodic $p$-adic integral orbit. In doing so, we then also obtain a dynamical translation of Theorem \ref{3.3}, namely, that the number of distinct $n$-periodic $p$-adic integral orbits that any $\varphi_{(p-1)^{\ell},c}$ has when iterated on the space $\mathbb{Z}_{p} \slash p\mathbb{Z}_{p}$ is $1$ or $2$ or $0$. Furthermore, as mentioned in Intro.\ref{sec1} that the count obtained in Theorem \ref{3.3} on the number of distinct $n$-periodic $p$-adic integral points of any $\varphi_{(p-1)^{\ell},c}$ modulo $p\mathbb{Z}_{p}$ is independent of $p$ (and so independent of deg$(\varphi_{(p-1)^{\ell},c})$ for every fixed $\ell \in \mathbb{Z}_{\geq 1}$) in each of the possibilities. Moreover, we may also observe that the expected total count of the number of distinct $n$-periodic $p$-adic integral points in the whole family of maps $\varphi_{(p-1)^{\ell},c}$ modulo $p\mathbb{Z}_{p}$ (namely, $1 + 2 + 0 =3$ for every fixed odd period $n\in \mathbb{Z}_{\geq 3}$ or $1 + 1 + 2 + 0 =4$ for every fixed even period $n\in \mathbb{Z}_{\geq 2}$) is not only also independent of $p$ (and so independent of deg$(\varphi_{(p-1)^{\ell},c})$ for every fixed $\ell \in \mathbb{Z}_{\geq 1}$), but is also a constant $3$ or $4$ even when $(p-1)^{\ell}\to \infty$; a somewhat interesting phenomenon differing significantly from a phenomenon that we remarked in Remark \ref{rem2.4}, however, coinciding somewhat surprising with an observation made in [\cite{BK3}, Remark 4.4] and [\cite{BK222}, Remark 3.5]. 
\end{rem}

\begin{rem}
As before, recall in \cite{BK3} we proved that the fixed point-counting function $Y_{c}(p) = 1, 2$ or $0$ for every fixed prime $p\geq 5$ and every coefficient $c\equiv 1, 0$ (mod $p\mathbb{Z}_{p}$) or $c\equiv -1$ (mod $p\mathbb{Z}_{p}$). Moreover, recall in the Proof of Theorem \ref{3.2} we proved that for every fixed odd (period) $n\in \mathbb{Z}_{\geq 3}$, the points $z\equiv 1, 0, 2$ (mod $p\mathbb{Z}_{p}$) are $n$-periodic $p$-adic integral points of any $\varphi_{(p-1)^{\ell},c}$ modulo $p\mathbb{Z}_{p}$ (which in \cite{BK3} we also obtained as fixed $p$-adic integral points of any $\varphi_{(p-1)^{\ell},c}$ modulo $p\mathbb{Z}_{p}$) for every fixed prime $p\geq 5$ and every coefficient $c\equiv 1, 0$ (mod $p\mathbb{Z}_{p}$); from which we then concluded $Y_{c}^{(n)}(p)  = 1, 2$ or $0$ for every fixed prime $p$ and every $c\equiv 1, 0$ (mod $p\mathbb{Z}_{p}$) or $c\equiv -1$ (mod $p\mathbb{Z}_{p}$). But now for every fixed odd (period) $n\in \mathbb{Z}_{\geq 3}$, we then note that the counting function $Y_{c}^{(n)}(p) = Y_{c}(p) = 1, 2$ or $0$ for every fixed $p$ and every $c\equiv 1, 0$ (mod $p\mathbb{Z}_{p}$) or $c\equiv -1$ (mod $p\mathbb{Z}_{p}$). Consequently, for every fixed odd (period) $n\in \mathbb{Z}_{\geq 1}$, we then note that every $n$-periodic $p$-adic integral orbit of any $\varphi_{(p-1)^{\ell},c}$ modulo $p\mathbb{Z}_{p}$ is a fixed $p$-adic integral orbit, and moreover every reduced polynomial map $\varphi_{(p-1)^{\ell},c}$ modulo $p\mathbb{Z}_{p}$ may have one or two or no fixed $p$-adic integral orbits; a somewhat interesting precise arithmetic-geometric insight on all odd $n$-periodic $p$-adic integral orbits of every reduced map $\varphi_{(p-1)^{\ell},c}$ modulo $p\mathbb{Z}_{p}$. Furthermore, recall in the Proof of Theorem \ref{3.3} that for every fixed even (period) $n\in \mathbb{Z}_{\geq 2}$, the points $z\equiv 1, 0, 2, -1$ (mod $p\mathbb{Z}_{p}$) are $n$-periodic $p$-adic integral points of any $\varphi_{(p-1)^{\ell},c}$ modulo $p\mathbb{Z}_{p}$ for every fixed prime $p\geq 5$ and every coefficient $c\equiv \pm 1, 0$ (mod $p\mathbb{Z}_{p}$) or $c\not \equiv \pm 1, 0$ (mod $p\mathbb{Z}_{p}$); from which we then concluded $Y_{c}^{(n)}(p)  = 1, 2$ or $0$ for every fixed $p$ and for every $c\equiv \pm 1, 0$ (mod $p\mathbb{Z}_{p}$) or $c\not \equiv \pm 1, 0$ (mod $p\mathbb{Z}_{p}$). But now for every fixed even (period) $n\in \mathbb{Z}_{\geq 4}$, we then also note that the counting function $Y_{c}^{(n)}(p) = Y_{c}^{(2)}(p) = 1, 2$ or $0$ for every fixed $p\geq 5$ and every $c\equiv \pm 1, 0$ (mod $p\mathbb{Z}_{p}$) or $c\not \equiv \pm 1, 0$ (mod $p\mathbb{Z}_{p}$). As before, we now also note that for every fixed even (period) $n\in \mathbb{Z}_{\geq 2}$, every $n$-periodic orbit of any $\varphi_{(p-1)^{\ell},c}$ modulo $p\mathbb{Z}_{p}$ is a $2$-periodic orbit, and moreover every reduced polynomial map $\varphi_{(p-1)^{\ell},c}$ modulo $p\mathbb{Z}_{p}$ may have one or two or no $n$-periodic orbits; a somewhat interesting precise arithmetic-geometric insight on all the even $n$-periodic orbits of every $\varphi_{(p-1)^{\ell},c}$ modulo $p\mathbb{Z}_{p}$.    
\end{rem}

\section{On the Number of $n$-Periodic $\mathbb{F}_{p}[t]\slash (\pi)$-Points of any Family of Polynomial Maps $\varphi_{p^{\ell},c}$}\label{sec4}

As in Section \ref{sec2} and \ref{sec3}, we in this section also wish to count the number of distinct $n$-periodic points of any $\varphi_{p^{\ell},c}$ modulo prime $\pi$, where $p\geq 3$ is any prime, $\ell \in \mathbb{Z}_{\geq 1}$ and $n\in \mathbb{Z}_{\geq 2}$ are any fixed integers. To that end, we as before let $p\geq 3$ be any prime, $\ell \in \mathbb{Z}_{\geq 1}$, $n\in \mathbb{Z}_{\geq 2}$ be any fixed integers, $c\in \mathbb{F}_{p}[t]$ be any point and $\pi\in \mathbb{F}_{p}[t]$ be any fixed irreducible monic polynomial of degree $m\geq 1$, and then define $n$-periodic point-counting function
\begin{equation}\label{N_{ct}}
N_{c(t)}^{(n)}(\pi, p) := \# \Biggl\{ z\in \mathbb{F}_{p}[t] / (\pi)   : \begin{aligned} \varphi_{p^{\ell},c}^{n-1}(z) -z \not \equiv 0 \ \text{(mod $\pi$)} \\ \ \varphi_{p^{\ell},c}^{n}(z) - z \equiv 0 \ \text{(mod $\pi$)} \end{aligned} \Biggr\}.
\end{equation} \noindent Again, setting $\ell =1$ and thus $\varphi_{p^{\ell}, c} = \varphi_{p,c}$, we then first prove the following theorem and its generalization \ref{4.2}:

\begin{thm} \label{4.1}
Let $\varphi_{3, c}$ be a cubic map defined by $\varphi_{3, c}(z) = z^3 + c$ for all $c, z\in\mathbb{F}_{3}[t]$, and $N_{c(t)}^{(n)}(\pi, 3)$ be defined as in \textnormal{(\ref{N_{ct}})}. Then $N_{c(t)}^{(n)}(\pi, 3)=3$ for every coefficient $c\in (\pi)$; otherwise $N_{c(t)}^{(n)}(\pi, 3) = 0$ for any coefficient $c \not \in (\pi)$.
\end{thm}
\begin{proof}
Let $f_{c(t)}(z)= \varphi_{3,c}^n(z)-z = \varphi_{3,c}(\varphi_{3,c}^{n-1}(z)) - z = (\varphi_{3,c}^{n-1}(z))^3 - z + c$, and so $f_{c(t)}(z)= (\varphi_{3,c}^{n-1}(z))^3 - z + c$. Now applying the multinomial theorem repeatedly on $(\varphi_{3,c}^{n-1}(z))^3$ after applying the binomial theorem on $(z^3 + c)^3$, we then obtain $(\varphi_{3,c}^{n-1}(z))^3$ is a monic polynomial in $z$ of degree $3^n$ with $\mathbb{F}_{3}[t]$-coefficients in multiples of $c$. Hence, we may then write $(\varphi_{3,c}^{n-1}(z))^3 = z^{3^{n}} + h(z)$, where $h(z)$ is a non-constant polynomial in $z$ of deg$(h)<3^n$ with $\mathbb{F}_{3}[t]$-coefficients in multiples of $c$; and so $f_{c(t)}(z)= z^{3^{n}} + h(z) - z + c$. Now for every coefficient $c\in (\pi)=\pi \mathbb{F}_{3}[t]$, reducing $f_{c(t)}(z)$ modulo prime $\pi$, it then follows $f_{c(t)}(z)\equiv z^{3^n} - z$ (mod $\pi$), since also $h(z)\in c\mathbb{F}_{3}[t][z]$ and so $h(z)\equiv 0$ (mod $\pi$); and thus now $f(z)$ modulo $\pi$ is a polynomial defined over a finite field $\mathbb{F}_{3}[t]\slash (\pi)$ of order $3^{\text{deg } \pi} = 3^m$. So now, recall from a well-known fact that every subfield of a finite field $\mathbb{F}_{3}[t]\slash (\pi)$ is of order $3^r$ for some positive integer $r\mid m$, we then obtain the inclusion $\mathbb{F}_{3}\hookrightarrow \mathbb{F}_{3}[t]\slash (\pi)$ of fields, where $\mathbb{F}_{3}$ is a field of order $3$. Moreover, it also follows from a well-known fact that $h(x)=x^3-x$ vanishes at every $z\in \mathbb{F}_{3}\subset \mathbb{F}_{3}[t]\slash (\pi)$; and so $z^3 = z$ for every element $z\in \mathbb{F}_{3}$. But now we note $z^{3^n}= (z^3)^{3^{n-1}} = (z^3)^{3^{n-2}} = z^{3^{n-2}}$ for every element $z\in \mathbb{F}_{3}\subset \mathbb{F}_{3}[t]\slash (\pi)$. Now since $n\geq 2$ and so $n-2\geq 0$, then if $n-2 = 0$ and so $z^{3^{n-2}} = z$, it then follows $z^{3^n} = z$ for every element $z\in \mathbb{F}_{3}$; and so $f_{c(t)}(z)\equiv 0$ (mod $\pi$) for every point $z\in \mathbb{F}_{3}\subset \mathbb{F}_{3}[t]\slash (\pi)$. Otherwise, if $n-2 > 0$, then since $n$ is a fixed integer, we may then continue performing the above procedure of decreasing the exponent $n-2$ of $z^{3^{n-2}}=z^{3^n}$ for every $z\in \mathbb{F}_{3}$ until $n-2$ is equal to zero; from which we then again obtain $f_{c(t)}(z)\equiv 0$ (mod $\pi$) for every point $z\in \mathbb{F}_{3}\subset \mathbb{F}_{3}[t]\slash (\pi)$. But now we then conclude $N_{c(t)}^{(n)}(\pi, 3) = 3$.

Finally, we now show $N_{c(t)}^{(n)}(\pi, 3) = 0$ for every coefficient $c\not \equiv 0$ (mod $\pi$) and every fixed integer $n\in \mathbb{Z}_{\geq 2}$. To see this, let's for the sake of a contradiction, suppose $f_{c(t)}(z)\equiv 0$ (mod $\pi$) for some point $z\in \mathbb{F}_{3}[t]\slash (\pi)$ and for every coefficient $c\not \equiv 0$ (mod $\pi$) and every fixed $n\in \mathbb{Z}_{\geq 2}$. Now since from earlier $f_{c(t)}(z)= z^{3^{n}} + h(z) - z + c$ where $h(z)\in c\mathbb{F}_{3}[t][z]$, it then also follows $z^{3^{n}} + h(z) - z + c \equiv 0$ (mod $\pi$) for some $z\in \mathbb{F}_{3}[t]\slash (\pi)$ and for every $c\not \equiv 0$ (mod $\pi$) and every fixed $n\in \mathbb{Z}_{\geq 2}$. So now, recall from earlier that $z^{3^n} = z$ for every $z\in \mathbb{F}_{3}\subset \mathbb{F}_{3}[t]\slash (\pi)$ and every fixed $n\in \mathbb{Z}_{\geq 2}$, we may then rewrite $z^{3^{n}} -z + h(z) + c \equiv 0$ (mod $\pi$) for some $z\in \mathbb{F}_{3}[t]\slash (\pi)$ and for every $c\not \equiv 0$ (mod $\pi$) to then obtain $h(z) + c \equiv 0$ (mod $\pi$) for some $z\in \mathbb{F}_{3}\subset \mathbb{F}_{3}[t]\slash (\pi)$ and for every $c\not \equiv 0$ (mod $\pi$). Moreover, looking at the multinomial expansion of $(\varphi_{3,c}^{n-1}(z))^3$, we then obtain $h(z)\equiv \sum_{i=1}^{n-1}c^{3^{n-i}}$ (mod $\pi$), since also $\mathbb{F}_{3}[t]\slash (\pi)$ is of characteristic $3$; and so $\sum_{i=1}^{n-1}c^{3^{n-i}} + c \equiv 0$ (mod $\pi$). 
But now we note that congruence $\sum_{i=1}^{n-1}c^{3^{n-i}} + c \equiv 0$ (mod $\pi$) can also happen if $\sum_{i=1}^{n-1}c^{3^{n-i}}\equiv 0$ (mod $\pi$) and also $c\equiv 0$ (mod $\pi$); and so follows a contradiction. Otherwise, suppose $f_{c(t)}(\alpha)\equiv 0$ (mod $\pi$) and so (from earlier) $\alpha^{3^{n}} + h(\alpha) - \alpha + c \equiv 0$ (mod $\pi$) for some $\alpha \in \mathbb{F}_{3}[t]\slash (\pi) \setminus \mathbb{F}_{3}$ and for every $c\not \equiv 0$ (mod $\pi$) and every fixed $n\in \mathbb{Z}_{\geq 2}$. So then, since $z^{3^{n}}=z$ for every $z\in \mathbb{F}_{3^{n}}\subset \mathbb{F}_{3}[t]\slash (\pi)$ and for every fixed $n\mid m$, then if a root $\alpha \in \mathbb{F}_{3^{n}}\subset \mathbb{F}_{3}[t]\slash (\pi)\setminus \mathbb{F}_{3}$ and so $\alpha^{3^{n}}=\alpha$, it then follows $h(\alpha)  + c \equiv 0$ (mod $\pi$); and moreover since $h(\alpha)\equiv \sum_{i=1}^{n-1}c^{3^{n-i}}$ (mod $\pi$), it then then $\sum_{i=1}^{n-1}c^{3^{n-i}} + c \equiv 0$ (mod $\pi$); and so also follows a contradiction. Otherwise, if $\alpha \not \in \mathbb{F}_{3^{n}}$, then since $h(\alpha)\equiv \sum_{i=1}^{n-1}c^{3^{n-i}}$ (mod $\pi$), it then follows $\alpha^{3^{n}} - \alpha +\sum_{i=1}^{n-1}c^{3^{n-i}}  + c \equiv 0$ (mod $\pi$). But then we note $\alpha^{3^{n}} - \alpha +\sum_{i=1}^{n-1}c^{3^{n-i}}  + c \equiv 0$ (mod $\pi$) can also occur if $\alpha^{3^{n}} - \alpha \equiv 0$ (mod $\pi$) and also $\sum_{i=1}^{n-1}c^{3^{n-i}}  + c \equiv 0$ (mod $\pi$); and so also a contradiction. This then overall means $f_{c(t)}(x) =\varphi_{3,c}^n(x)-x$ has no roots in $\mathbb{F}_{3}\subset \mathbb{F}_{3}[t]\slash (\pi)$ for every coefficient $c\not \in (\pi)$ and every fixed $n\in \mathbb{Z}_{\geq 2}$; and so we then conclude $N_{c(t)}^{(n)}(\pi, 3) = 0$. This then completes the whole proof, as needed.
\end{proof} 
We now wish to generalize Theorem \ref{4.1} to any polynomial map $\varphi_{p, c}$ for any given prime $p\geq 3$. More precisely, we show that the number of distinct $n$-Periodic $\mathbb{F}_{p}[t]$-points of any polynomial map $\varphi_{p, c}$ is $p$ or zero:

\begin{thm} \label{4.2}
Let $p\geq 3$ be any fixed prime integer, and $\varphi_{p, c}$ be defined by $\varphi_{p, c}(z) = z^p + c$ for all $c, z\in\mathbb{F}_{p}[t]$. Let $N_{c(t)}^{(n)}(\pi, p)$ be as in \textnormal{(\ref{N_{ct}})}. Then $N_{c(t)}^{(n)}(\pi, p) = p$ for every $c\in (\pi)$; otherwise $N_{c(t)}^{(n)}(\pi, p) = 0$ for every $c \not \in (\pi)$.
\end{thm}
\begin{proof}
By applying a similar argument as in the Proof of Theorem \ref{4.1}, we then obtain the count as desired. That is, let $f_{c(t)}(z)= \varphi_{p,c}^n(z)-z = \varphi_{p,c}(\varphi_{p,c}^{n-1}(z)) - z = (\varphi_{p,c}^{n-1}(z))^p - z + c$, and so $f_{c(t)}(z)= (\varphi_{p,c}^{n-1}(z))^p - z + c$. So now, applying the multinomial theorem repeatedly on $(\varphi_{p,c}^{n-1}(z))^p$ after applying the binomial theorem on $(z^p + c)^p$, it then follows $(\varphi_{p,c}^{n-1}(z))^p$ is a monic polynomial in $z$ of degree $p^n$ with $\mathbb{F}_{p}[t]$-coefficients in multiples of $c$. Thus, we may then write $(\varphi_{p,c}^{n-1}(z))^p = z^{p^{n}} + h(z)$, where $h(z)$ is a non-constant polynomial in $z$ of deg$(h)<p^n$ with $\mathbb{F}_{p}[t]$-coefficients in multiples of $c$; and so $f_{c(t)}(z)= z^{p^{n}} + h(z) - z + c$. Now for every coefficient $c\in (\pi)$, reducing $f_{c(t)}(z)$ modulo prime $\pi$, it then follows $f_{c(t)}(z)\equiv z^{p^n} - z$ (mod $\pi$), since also $h(z)\in c\mathbb{F}_{p}[t][z]$ and so $h(z)\equiv 0$ (mod $\pi$); and so now $f_{c(t)}(z)$ modulo $\pi$ is a polynomial defined over a finite field $\mathbb{F}_{p}[t]\slash (\pi)$ of order $p^{\text{deg }\pi} = p^m$. So now, recall a well-known fact that every subfield of a finite field $\mathbb{F}_{p}[t]\slash (\pi)$ is of order $p^r$ for some positive integer $r\mid m$, we then obtain the inclusion $\mathbb{F}_{p}\hookrightarrow \mathbb{F}_{p}[t]\slash (\pi)$ of fields; and moreover recall (as a fact) that $z^p = z$ for every $z\in \mathbb{F}_{p}\subset \mathbb{F}_{p}[t]\slash (\pi)$. But now we note $z^{p^n}= (z^p)^{p^{n-1}} = (z^p)^{p^{n-2}} = z^{p^{n-2}}$ for every $z\in \mathbb{F}_{p}$. So now, since $n\geq 2$ and so $n-2\geq 0$, then if $n-2 = 0$ and so $z^{p^{n-2}} = z$, it then follows $z^{p^n} = z$  for every $z\in \mathbb{F}_{p}$; and so the reduced polynomial $f_{c(t)}(z)\equiv 0$ (mod $\pi$) for every point $z\in \mathbb{F}_{p}\subset \mathbb{F}_{p}[t]\slash (\pi)$. Otherwise, if $n-2 > 0$, then since $n$ is a fixed integer, we may then continue performing the above procedure of decreasing the exponent $n-2$ of  $z^{p^{n-2}} = z^{p^n}$ for every $z\in \mathbb{F}_{p}$ until $n-2$ is equal to zero; and from which we then again obtain that $f_{c(t)}(z)\equiv 0$ (mod $\pi$) for every point $z\in \mathbb{F}_{p}\subset \mathbb{F}_{p}[t]\slash (\pi)$. But now we then conclude $N_{c(t)}^{(n)}(\pi, p) = p$.

Finally, we now show $N_{c(t)}^{(n)}(\pi, p) = 0$ for every coefficient $c \not \in (\pi)$ and every fixed integer $n\in \mathbb{Z}_{\geq 2}$. To see this, let's for the sake of a contradiction, suppose $f_{c(t)}(z)\equiv 0$ (mod $\pi$) for some point $z\in \mathbb{F}_{p}[t]\slash (\pi)$ and for every coefficient $c\not \equiv 0$ (mod $\pi$) and every fixed $n\in \mathbb{Z}_{\geq 2}$. So then, since from earlier $f_{c(t)}(z)= z^{p^{n}} + h(z) - z + c$ where $h(z)\in c\mathbb{F}_{p}[t][z]$, it then also follows $z^{p^{n}} + h(z) - z + c \equiv 0$ (mod $\pi$) for some point $z\in \mathbb{F}_{p}[t]\slash (\pi)$ and for every $c\not \equiv 0$ (mod $\pi$) and every fixed $n$. So now, since from earlier $z^{p^n} = z$ for every $z\in \mathbb{F}_{p}\subset \mathbb{F}_{p}[t]\slash (\pi)$ and every fixed $n$, we may then rewrite $z^{p^{n}} -z + h(z) + c \equiv 0$ (mod $\pi$) for some $z\in \mathbb{F}_{p}[t]\slash (\pi)$ and for every $c\not \equiv 0$ (mod $\pi$) to obtain $h(z) + c \equiv 0$ (mod $\pi$) for some $z\in \mathbb{F}_{p}\subset \mathbb{F}_{p}[t]\slash (\pi)$ and for every $c\not \equiv 0$ (mod $\pi$). Moreover, looking at the multinomial expansion of $(\varphi_{p,c}^{n-1}(z))^p$, it then follows $h(z)\equiv \sum_{i=1}^{n-1}c^{p^{n-i}}$ (mod $\pi$), since also $\mathbb{F}_{p}[t]\slash (\pi)$ is of characteristic $p$; and so $\sum_{i=1}^{n-1}c^{p^{n-i}} + c \equiv 0$ (mod $\pi$). But now we note that $\sum_{i=1}^{n-1}c^{p^{n-i}} + c \equiv 0$ (mod $\pi$) can also happen if $\sum_{i=1}^{n-1}c^{p^{n-i}}\equiv 0$ (mod $\pi$) and also $c\equiv 0$ (mod $\pi$); and so follows a contradiction. Otherwise, suppose $f_{c(t)}(\alpha)\equiv 0$ (mod $\pi$) and so (from earlier) $\alpha^{p^{n}} + h(\alpha) - \alpha + c \equiv 0$ (mod $\pi$) for some point $\alpha \in \mathbb{F}_{p}[t]\slash (\pi) \setminus \mathbb{F}_{p}$ and for every $c\not \equiv 0$ (mod $\pi$) and every fixed $n\in \mathbb{Z}_{\geq 2}$. So then, since $z^{p^{n}}=z$ for every $z\in \mathbb{F}_{p^{n}}\subset \mathbb{F}_{p}[t]\slash (\pi)$ and every fixed $n\mid m$, then if a root $\alpha \in \mathbb{F}_{p^{n}}\subset \mathbb{F}_{p}[t]\slash (\pi)\setminus \mathbb{F}_{p}$ and so $\alpha^{p^{n}}=\alpha$, it then follows $h(\alpha)  + c \equiv 0$ (mod $\pi$); and moreover since $h(\alpha)\equiv \sum_{i=1}^{n-1}c^{p^{n-i}}$ (mod $\pi$), it then also follows $\sum_{i=1}^{n-1}c^{p^{n-i}} + c \equiv 0$ (mod $\pi$); and so also follows a contradiction. Otherwise, if $\alpha \not \in \mathbb{F}_{p^{n}}$, then since $h(\alpha)\equiv \sum_{i=1}^{n-1}c^{p^{n-i}}$ (mod $\pi$), it then follows $\alpha^{p^{n}} - \alpha +\sum_{i=1}^{n-1}c^{p^{n-i}}  + c \equiv 0$ (mod $\pi$). But then we note $\alpha^{p^{n}} - \alpha +\sum_{i=1}^{n-1}c^{p^{n-i}}  + c \equiv 0$ (mod $\pi$) can also occur whenever $\alpha^{p^{n}} - \alpha \equiv 0$ (mod $\pi$) and also $\sum_{i=1}^{n-1}c^{p^{n-i}}  + c \equiv 0$ (mod $\pi$); and so also follows a contradiction. This then overall means $f_{c(t)}(x) =\varphi_{p,c}^n(x)-x$ has no roots in $\mathbb{F}_{p}\subset \mathbb{F}_{p}[t]\slash (\pi)$ for every coefficient $c\not \in (\pi)$ and every fixed $n\in \mathbb{Z}_{\geq 2}$; and thus we then conclude  $N_{c(t)}^{(n)}(\pi, p) = 0$. This then completes the whole proof, as needed.
\end{proof}

Finally, we now wish to generalize Theorem \ref{4.2} further to any $\varphi_{p^{\ell}, c}$ for any prime $p\geq 3$ and any $\ell \in \mathbb{Z}_{\geq 1}$. More precisely, we prove that the number of distinct $n$-periodic points of any $\varphi_{p^{\ell}, c}$ modulo $\pi$ is also $p$ or zero:

\begin{thm} \label{4.3}
Let $p\geq 3$ be any fixed prime integer, and $\ell \geq 1$ be any fixed integer. Let $\varphi_{p^{\ell}, c}$ be a polynomial map defined by $\varphi_{p^{\ell}, c}(z) = z^{p^{\ell}} + c$ for all points $c, z\in\mathbb{F}_{p}[t]$, and $N_{c(t)}^{(n)}(\pi, p)$ be the number defined as in \textnormal{(\ref{N_{ct}})}. Then  $N_{c(t)}^{(n)}(\pi, p) = p$  for every coefficient $c\in (\pi)$; otherwise we have $N_{c(t)}^{(n)}(\pi, p) = 0$ for every coefficient $c \not \in (\pi)$. 
\end{thm}

\begin{proof}
By applying a similar argument as in the Proof of Theorem \ref{4.2}, we then obtain the count as desired. That is, let $f_{c(t)}(z)= \varphi_{p^{\ell},c}^n(z)-z = (\varphi_{p^{\ell},c}^{n-1}(z))^{p^{\ell}} - z + c$, and so $f_{c(t)}(z)= (\varphi_{p^{\ell},c}^{n-1}(z))^{p^{\ell}} - z + c$. Now applying the multinomial theorem repeatedly on the term $(\varphi_{p^{\ell},c}^{n-1}(z))^{p^{\ell}}$ after applying the binomial theorem on $(z^{p^{\ell}} + c)^{p^{\ell}}$, we then obtain $(\varphi_{p^{\ell},c}^{n-1}(z))^{p^{\ell}}$ is a monic polynomial in $z$ of degree $p^{n\ell}$ with $\mathbb{F}_{p}[t]$-coefficients in multiples of $c$. Hence, we may then write $(\varphi_{p^{\ell},c}^{n-1}(z))^{p^{\ell}} = z^{p^{n\ell}} + h(z)$, where $h(z)$ is a non-constant polynomial in $z$ of deg$(h)<p^{n\ell}$ with $\mathbb{F}_{p}[t]$-coefficients in multiples of $c$; and so $f_{c(t)}(z)= z^{p^{n\ell}} + h(z) - z + c$. So now, for every coefficient $c\in (\pi)$, reducing $f_{c(t)}(z)$ modulo prime $\pi$, it then follows  $f_{c(t)}(z)\equiv z^{p^{n\ell}} - z$ (mod $\pi$), since also $h(z)\in c\mathbb{F}_{p}[t][z]$ and so $h(z)\equiv 0$ (mod $\pi$); and so the reduced polynomial $f_{c(t)}(z)$ modulo $\pi$ is now a polynomial defined over a finite field $\mathbb{F}_{p}[t]\slash (\pi)$ of order $p^m$. Now recall (from a fact) the inclusion $\mathbb{F}_{p}\hookrightarrow \mathbb{F}_{p}[t]\slash (\pi)$ of fields, and also recall $z^p = z$ for every element $z\in \mathbb{F}_{p}$, it then follows $z^{p^{n\ell}}= (z^p)^{p^{n\ell-1}} = (z^p)^{p^{n\ell-2}} = z^{p^{n\ell-2}}$ for every element $z\in \mathbb{F}_{p}$. So now, since $n\ell\geq 2$ and so $n\ell-2\geq 0$, then if $n\ell-2 = 0$ and so $z^{p^{n\ell-2}} = z$, it then follows $z^{p^{n\ell}} = z$  for every element $z\in \mathbb{F}_{p}$; and so $f_{c(t)}(z)\equiv 0$ (mod $\pi$) for every point $z\in \mathbb{F}_{p}\subset \mathbb{F}_{p}[t]\slash (\pi)$. Otherwise, if $n\ell-2 > 0$, then since $n, \ell $ (and so $n\ell$) is a fixed integer, we may then continue performing the above procedure of decreasing the exponent $n\ell-2$ of $z^{p^{n\ell-2}}=z^{p^{n\ell}}$ for every $z\in \mathbb{F}_{p}$ until $n\ell-2$ is equal to zero; and from which we then again obtain that $f_{c(t)}(z)\equiv 0$ (mod $\pi$) for every point $z\in \mathbb{F}_{p}\subset \mathbb{F}_{p}[t]\slash (\pi)$. Hence, we now conclude $N_{c(t)}^{(n)}(\pi, p) = p$.

Finally, we now show $N_{c(t)}^{(n)}(\pi, p) = 0$ for every coefficient $c \not \in (\pi)$ and every fixed integers $\ell\in \mathbb{Z}_{\geq 1}$ and $n\in \mathbb{Z}_{\geq 2}$. As before, let's for the sake of a contradiction, suppose $f_{c(t)}(z)\equiv 0$ (mod $\pi$) for some $z\in \mathbb{F}_{p}[t]\slash (\pi)$ and for every $c\not \equiv 0$ (mod $\pi$) and every fixed $\ell\in \mathbb{Z}_{\geq 1}$ and $n\in \mathbb{Z}_{\geq 2}$. So then, since from earlier $f_{c(t)}(z)= z^{p^{n\ell}} + h(z) - z + c$ where $h(z)\in c\mathbb{F}_{p}[t][z]$, it then follows $z^{p^{n\ell}} + h(z) - z + c \equiv 0$ (mod $\pi$) for some $z\in \mathbb{F}_{p}[t]\slash (\pi)$ and for every $c\not \equiv 0$ (mod $\pi$) and every fixed $\ell$ and $n$. So now, recall from earlier $z^{p^{n\ell}} = z$ for every $z\in \mathbb{F}_{p}\subset \mathbb{F}_{p}[t]\slash (\pi)$ and every fixed $\ell$ and $n$, we may then rewrite $z^{p^{m\ell}} -z + h(z) + c \equiv 0$ (mod $\pi$) for some $z\in \mathbb{F}_{p}[t]\slash (\pi)$ and for every $c\not \equiv 0$ (mod $\pi$) to obtain $h(z) + c \equiv 0$ (mod $\pi$) for some $z\in \mathbb{F}_{p}\subset \mathbb{F}_{p}[t]\slash (\pi)$ and for every $c\not \equiv 0$ (mod $\pi$). Moreover, looking at the multinomial expansion of $(\varphi_{p^{\ell},c}^{n-1}(z))^{p^{\ell}}$, we then obtain $h(z)\equiv \sum_{i=1}^{n-1}c^{p^{n\ell-i}}$ (mod $\pi$); and so $\sum_{i=1}^{n-1}c^{p^{n\ell-i}} + c \equiv 0$ (mod $\pi$). 
But now we note $\sum_{i=1}^{n-1}c^{p^{n\ell-i}} + c \equiv 0$ (mod $\pi$) can also happen if $\sum_{i=1}^{n-1}c^{p^{n\ell-i}}\equiv 0$ (mod $\pi$) and also $c\equiv 0$ (mod $\pi$); and so follows a contradiction. Otherwise, suppose $f_{c(t)}(\alpha)\equiv 0$ (mod $\pi$) and so (from earlier) $\alpha^{p^{n\ell}} + h(\alpha) - \alpha + c \equiv 0$ (mod $\pi$) for some $\alpha \in \mathbb{F}_{p}[t]\slash (\pi) \setminus \mathbb{F}_{p}$ and for every $c\not \equiv 0$ (mod $\pi$) and fixed $n$ and $\ell$. So then, since $z^{p^{n\ell}}=z$ for every $z\in \mathbb{F}_{p^{n\ell}}\subset \mathbb{F}_{p}[t]\slash (\pi)$ and every fixed $n\ell\mid m$, then if a root $\alpha \in \mathbb{F}_{p^{n\ell}}\subset \mathbb{F}_{p}[t]\slash (\pi)\setminus \mathbb{F}_{p}$ and so $\alpha^{p^{n\ell}}=\alpha$, it then follows $h(\alpha)  + c \equiv 0$ (mod $\pi$); and since $h(\alpha)\equiv \sum_{i=1}^{n-1}c^{p^{n\ell-i}}$ (mod $\pi$), it then also follows $\sum_{i=1}^{n-1}c^{p^{n\ell-i}} + c \equiv 0$ (mod $\pi$); and so also follows a contradiction. 
Otherwise, if $\alpha \not \in \mathbb{F}_{p^{n\ell}}$, then since $h(\alpha)\equiv \sum_{i=1}^{n-1}c^{p^{n\ell-i}}$ (mod $\pi$), it then follows $\alpha^{p^{n\ell}} - \alpha +\sum_{i=1}^{n-1}c^{p^{n\ell-i}}  + c \equiv 0$ (mod $\pi$). But then we note $\alpha^{p^{n}} - \alpha +\sum_{i=1}^{n-1}c^{p^{n\ell-i}}  + c \equiv 0$ (mod $\pi$) can also occur if $\alpha^{p^{n\ell}} - \alpha \equiv 0$ (mod $\pi$) and also $\sum_{i=1}^{n-1}c^{p^{n\ell-i}}  + c \equiv 0$ (mod $\pi$); and so also follows a contradiction. It then overall follows $f_{c(t)}(x) =\varphi_{p^{\ell},c}^n(x)-x$ has no roots in $\mathbb{F}_{p}\subset \mathbb{F}_{p}[t]\slash (\pi)$ for every coefficient $c\not \in (\pi)$ and every fixed $\ell\in \mathbb{Z}_{\geq 1}$ and $n\in \mathbb{Z}_{\geq 2}$; and so we  conclude  $N_{c(t)}^{(n)}(\pi, p) = 0$. This then completes the whole proof, as needed.
\end{proof}

\begin{rem}\label{rem4.4}
With Theorem \ref{4.3}, we may then to each distinct $n$-periodic point of $\varphi_{p^{\ell},c}$ associate an $n$-periodic orbit. In doing so, we then also obtain a dynamical translation of Theorem \ref{4.3}, namely, that the number of distinct $n$-periodic orbits that any $\varphi_{p^{\ell},c}$ has when iterated on the space $\mathbb{F}_{p}[t]\slash (\pi)$ is $p$ or zero. As mentioned in Intro.\ref{sec1} that the count obtained in Theorem \ref{4.3} on the number of distinct $n$-periodic points of any $\varphi_{p^{\ell},c}$ modulo $\pi$ may on depend on $p$ (and hence depend on deg$(\varphi_{p^{\ell},c})$ for every fixed $\ell\in \mathbb{Z}_{\geq 1}$) and not on $m=$ deg$(\pi)$ in one of the possibilities; or the count obtained in Theorem \ref{4.3} may be neither depend on $p$ (and hence on deg$(\varphi_{p^{\ell},c})$ for every fixed $\ell\in \mathbb{Z}_{\geq 1}$) nor on $m$ in the other possibility. Moreover, the expected total number (namely, $p+0 = p$ for every fixed period $n\in \mathbb{Z}_{\geq 2}$) of distinct $n$-periodic points in the whole family of maps $\varphi_{p^{\ell},c}$ modulo $\pi$ may not only also depend on $p$ (and so depend on deg$(\varphi_{p^{\ell},c})$ for every fixed $\ell\in \mathbb{Z}_{\geq 1}$) and not on $m$, but may also grow to infinity when degree $p^{\ell}\to \infty$. As a result, we may have $N_{c(t)}^{(n)}(\pi, p)\to \infty$ or $N_{c(t)}^{(n)}(\pi, p)\to 0$ as $p\to \infty$; a somewhat interesting phenomenon differing significantly from a phenomenon observed in Remark \ref{rem5.4}, however, coinciding somewhat surprising with a phenomenon in [\cite{BK3}, Remark 5.4 (when $\ell \in \{1, p\}$)] and in Remark \ref{rem2.4}.   
\end{rem}

\begin{rem}
Recall in \cite{BK3} we proved that the function $N_{c(t)}(\pi, p) = p$ (for every $\ell \in \{1, p\}$) or $0$ for every fixed prime $p\geq 3$ and every coefficient $c\equiv 0$ (mod $\pi$) or $c\not \equiv 0$ (mod $\pi$). Moreover, recall in Theorem \ref{4.3} that for every fixed (period) $n\in \mathbb{Z}_{\geq 2}$, the function $N_{c(t)}(\pi, p) = p$ or $0$ for every fixed $p\geq 3$ and every coefficient $c\equiv 0$ (mod $\pi$) or $c\not \equiv 0$ (mod $\pi$). But now for every fixed (period) $n\in \mathbb{Z}_{\geq 2}$, we then note $N_{c(t)}^{(n)}(\pi, p) = N_{c(t)}(\pi, p) = p$ (for every $\ell \in \{1, p\}$) or $0$ for every fixed $p$ and every $c\equiv 0$ (mod $\pi$) or $c\not \equiv 0$ (mod $\pi$). Moreover, for every $c\equiv 0$ (mod $\pi$) and for every $\ell \in \{1, p\}$, it also follows from [\cite{BK3}, Proof of Theorem 5.3] and Proof of Theorem \ref{4.3} that every $n$-periodic $\mathbb{F}_{p}[t]$-point (and thus every $n$-periodic $\mathbb{F}_{p}[t]$-orbit) of any $\varphi_{p^{\ell},c}$ modulo $\pi$ is a fixed $\mathbb{F}_{p}[t]$-point (and thus a fixed $\mathbb{F}_{p}[t]$-orbit). This also means that for every fixed (period) $n\in \mathbb{Z}_{\geq 1}$, every $n$-periodic orbit of any $\varphi_{p^{\ell},c}$ modulo $\pi$ is a fixed orbit, and moreover every $\varphi_{p^{\ell},c}$ modulo $\pi$ has $p$ distinct fixed orbits or zero; a somewhat interesting precise arithmetic-geometric insight on all the $n$-periodic orbits of every $\varphi_{p^{\ell},c}$ modulo $\pi$. 
\end{rem}

\section{On Number of $n$-Periodic $\mathbb{F}_{p}[t]\slash (\pi)$-Points of any Family of Polynomial Maps $\varphi_{(p-1)^{\ell},c}$}\label{sec5}

As in Section \ref{sec4}, we in this section also wish to count the number of distinct $n$-periodic points of any $\varphi_{(p-1)^{\ell},c}$ modulo prime $\pi$, where $p\geq 5$ is any prime, $\ell \in \mathbb{Z}_{\geq 1}$  and $n\in \mathbb{Z}_{\geq 2}$ are any fixed integers. As before, we again let $p\geq 5$ be any prime, $\ell \in \mathbb{Z}_{\geq 1}, n\in \mathbb{Z}_{\geq 2}$ be any fixed integers, $c\in \mathbb{F}_{p}[t]$ be any point, and $\pi\in \mathbb{F}_{p}[t]$ be any fixed irreducible monic polynomial of degree $m\geq 1$, and then similarly define $n$-periodic point-counting function
\begin{equation}\label{M_{ct}}
M_{c(t)}^{(n)}(\pi, p) := \# \Biggl\{ z\in \mathbb{F}_{p}[t] / (\pi)   : \begin{aligned} \varphi_{(p-1)^{\ell},c}^{n-1}(z) -z \not \equiv 0 \ \text{(mod $\pi$)} \\ \ \varphi_{(p-1)^{\ell},c}^{n}(z) - z \equiv 0 \ \text{(mod $\pi$)} \end{aligned} \Biggr\}.
\end{equation} \noindent Again, setting $\ell =1$ and so $\varphi_{(p-1)^{\ell}, c} = \varphi_{p-1,c}$, we first prove the following theorem and its generalization \ref{5.2}:

\begin{thm} \label{5.1}
Let $\varphi_{4, c}$ be a quartic map defined by $\varphi_{4, c}(z) = z^4 + c$ for all $c, z\in\mathbb{F}_{5}[t]$, and $M_{c(t)}^{(n)}(\pi, 5)$ be as in \textnormal{(\ref{M_{ct}})}. Then $M_{c(t)}^{(n)}(\pi, 5) = 1$ for any $c\equiv \pm 1 \ (mod \ \pi)$ and fixed (even) $n$ or $M_{c(t)}^{(n)}(\pi, 5) = 2$ for any $c\in (\pi)$; otherwise $M_{c(t)}^{(n)}(\pi, 5) = 0$ for any $c \equiv -1\ (mod \ \pi)$ and fixed odd $n$ or any $c\not \equiv \pm 1, 0 \ (mod \ \pi)$ and fixed (even) $n$.
\end{thm}
\begin{proof}
Let $g_{c(t)}(z)= \varphi_{4,c}^n(z)-z = \varphi_{4,c}(\varphi_{4,c}^{n-1}(z)) - z = (\varphi_{4,c}^{n-1}(z))^4 - z + c$, and so $g(z)= (\varphi_{4,c}^{n-1}(z))^4 - z + c$. Now applying the multinomial theorem repeatedly on $(\varphi_{4,c}^{n-1}(z))^4$ after applying the binomial theorem on $(z^4 + c)^4$, we then obtain $(\varphi_{4,c}^{n-1}(z))^4$ is a monic polynomial in $z$ of degree $4^n$ with $\mathbb{F}_{5}[t]$-coefficients in multiples of $c$. Thus, we may then write $(\varphi_{4,c}^{n-1}(z))^4 = z^{4^{n}} + h(z)$, where $h(z)$ is a non-constant polynomial in $z$ of deg$(h)<4^n$ with $\mathbb{F}_{5}[t]$-coefficients in multiples of $c$; and so $g_{c(t)}(z)= z^{4^{n}} + h(z) - z + c$. Now for every coefficient $c\in (\pi) = \pi \mathbb{F}_{5}[t]$, reducing $g_{c(t)}(z)$ modulo prime $\pi$, it then follows $g_{c(t)}(z)\equiv z^{4^n} - z$ (mod $\pi$), since also $h(z)\in c\mathbb{F}_{5}[t][z]$ and so $h(z)\equiv 0$ (mod $\pi$); and so the reduced polynomial $g_{c(t)}(z)$ modulo $\pi$ is now a polynomial defined over a finite field $\mathbb{F}_{5}[t]\slash (\pi)$ of order $5^{\text{deg }\pi}=5^m$. So now, since every subfield of a finite field $\mathbb{F}_{5}[t]\slash (\pi)$ is of order $5^r$ for some positive integer $r\mid m$, we then obtain the inclusion $\mathbb{F}_{5}\hookrightarrow \mathbb{F}_{5}[t]\slash (\pi)$ of fields, where $\mathbb{F}_{5}$ is a field of order $5$. Moreover, it also follows from a well-known fact that $h(x)=x^4-1$ vanishes at every nonzero $z\in \mathbb{F}_{5}\subset \mathbb{F}_{5}[t]\slash (\pi)$; and so $z^4 = 1$ for every $z\in \mathbb{F}_{5}^{\times}=\mathbb{F}_{5}\setminus \{0\}$. But now we note $z^{4^n}= (z^4)^{4^{n-1}} = 1$ for every $z\in \mathbb{F}_{5}^{\times}$; and so $g_{c(t)}(z)\equiv 1 - z$ (mod $\pi$) for every point $z\in \mathbb{F}_{5}^{\times}\subset \mathbb{F}_{5}[t]\slash (\pi)$ and so $g_{c(t)}(z)$ has a root in $\mathbb{F}_{5}[t]\slash (\pi)$, namely, $z\equiv 1$ (mod $\pi$). Moreover, since $z$ is a linear factor of $g_{c(t)}(z)\equiv z(z^{4^n-1} - 1)$ (mod $\pi$), it then follows $z\equiv 0$ (mod $\pi$) is also root of $g_{c(t)}(z)$ modulo $\pi$ in $\mathbb{F}_{5}[t]\slash (\pi)$. But now we then conclude $M_{c(t)}^{(n)}(\pi, 5) = 2$. To see $M_{c(t)}^{(n)}(\pi, 5) = 1$ for every coefficient $c\equiv 1$ (mod $\pi$) and for every fixed integer $n\in \mathbb{Z}_{\geq 2}$, we first note that writing $\varphi_{4,c}^{n-1}(z) = \underbrace{((((z^4 + c)^4 + c)^4 + c)^4 + \cdots + c)^4 + c}_\text{$(n-1)$ times}$ and then reducing $\varphi_{4,c}^{n-1}(z)$ modulo $\pi$ along with $c\equiv 1$ (mod $\pi$) and also using $z^4 = 1$ for every element $z\in \mathbb{F}_{5}^{\times}$, we then obtain $\varphi_{4,c}^{n-1}(z)\equiv 2$ (mod $\pi$) and $(\varphi_{4,c}^{n-1}(z))^4\equiv 1$ (mod $\pi$) for every fixed $n\in \mathbb{Z}_{\geq 2}$, since also $2^4 = 1$ in $\mathbb{F}_{5}$. But then $g_{c(t)}(z)=(\varphi_{4,c}^{n-1}(z))^4 - z + c\equiv 2-z$ (mod $\pi$) for every point $z\in \mathbb{F}_{5}^{\times}\subset \mathbb{F}_{5}[t]\slash (\pi)$ and so $g_{c(t)}(z)$ modulo $\pi$ has a root in $\mathbb{F}_{5}[t]\slash (\pi)$, namely, $z\equiv 2$ (mod $\pi$); and so we conclude $M_{c(t)}^{(n)}(\pi, 5) = 1$. We now show $M_{c(t)}^{(n)}(\pi, 5) = 1$ for every coefficient $c\equiv -1$ (mod $\pi$) and for every fixed even integer $n\in \mathbb{Z}_{\geq 2}$. As before, since $c\equiv -1$ (mod $\pi$) and also since $z^4 = 1$ for every $z\in \mathbb{F}_{5}^{\times}$, reducing $\varphi_{4,c}^m(z)$ modulo $\pi$, we then obtain $\varphi_{4,c}^{n}(z)\equiv -1$ (mod $\pi$) for every fixed even $n\in \mathbb{Z}_{\geq 2}$. But then $g_{c(t)}(z)= \varphi_{4,c}^n(z)-z \equiv -(1+z)$ (mod $\pi$) for every point $z\in \mathbb{F}_{5}^{\times}$ and for every fixed even $n\in \mathbb{Z}_{\geq 2}$; and so $g_{c(t)}(z)$ modulo $\pi$ has a root in $\mathbb{F}_{5}[t]\slash (\pi)$, namely, $z\equiv -1$ (mod $\pi$) and so we then conclude $M_{c(t)}^{(n)}(\pi, 5) = 1$. 

Finally, we now show $M_{c(t)}^{(n)}(\pi, 5) = 0$ for every coefficient $c \equiv -1$ (mod $\pi$) and every fixed odd integer $n\in \mathbb{Z}_{\geq 3}$ or for every coefficient $c\not \equiv \pm1, 0$ (mod $\pi$) and every fixed (even) integer $n\in \mathbb{Z}_{\geq 2}$. To see this, we note that since $c\equiv -1$ (mod $\pi$) and also since $z^4 = 1$ for every $z\in \mathbb{F}_{5}^{\times}$, reducing $\varphi_{4,c}^n(z)$ modulo $\pi$, we then obtain $\varphi_{4,c}^{n}(z)\equiv 0$ (mod $\pi$) for every fixed odd $n$; and so $g_{c(t)}(z)= \varphi_{4,c}^n(z)-z \equiv -z$ (mod $\pi$) for every point $z\in \mathbb{F}_{5}^{\times}\subset \mathbb{F}_{5}[t]\slash (\pi)$ and fixed odd $n$. But now we note $z\equiv 0$ (mod $\pi$) is a root of $g_{c(t)}(z)$ modulo $\pi$ for every coefficient $c\equiv -1$ (mod $\pi$) and every fixed odd $n\in \mathbb{Z}_{\geq 3}$, and also $z\equiv 0$ (mod $\pi$) is also root of $g_{c(t)}(z)$ modulo $\pi$ for every coefficient $c\equiv 0$ (mod $\pi$) and every fixed odd $n\in \mathbb{Z}_{\geq 3}$ , as seen earlier; and from which then follows a contradiction that $1 \equiv 0$ (mod $\pi$). This then means $g_{c(t)}(x)=\varphi_{4,c}^n(x)-x$ has no roots in $\mathbb{F}_{5}[t]\slash (\pi)$ for every coefficient $c\equiv -1$ (mod $\pi$) and every fixed odd integer $n\in \mathbb{Z}_{\geq 3}$; and so we then conclude $M_{c(t)}^{(n)}(\pi, 5) = 0$. Otherwise, suppose $g_{c(t)}(z)\equiv 0$ (mod $\pi$) for some point $z\in \mathbb{F}_{5}[t]\slash (\pi)\setminus \{0\}$ and for every coefficient $c\not \equiv \pm1, 0$ (mod $\pi$) and every fixed (even) integer $n\in \mathbb{Z}_{\geq 2}$; and so (from earlier) $z^{4^n} + h(z)-z+c\equiv 0$ (mod $\pi$). So then, since $z^{4^n}=1$ for every $z\in \mathbb{F}_{5}^{\times}$ and every fixed (even) $n$, we may then rewrite  $z^{4^n} -z + h(z)+c\equiv 0$ (mod $\pi$) to obtain $(1-z) + (h(z)+c)\equiv 0$ (mod $\pi$). But now we note that the congruence $(1-z) + (h(z)+c)\equiv 0$ (mod $\pi$) can happen if $1-z\equiv 0$ (mod $\pi$) and also $h(z)+c\equiv 0$ (mod $\pi$). Moreover, recall from earlier $1-z\equiv 0$ (mod $\pi$) whenever $c\equiv 0$ (mod $\pi$); and which then contradicts the condition $c\not \equiv \pm1, 0$ (mod $\pi$). It then also follows $g_{c(t)}(x)=\varphi_{4,c}^n(x)-x$ has no roots in $\mathbb{F}_{5}[t]\slash (\pi)$ for every coefficient $c\not \equiv \pm1, 0$ (mod $\pi$) and every fixed (even) integer $n\in \mathbb{Z}_{\geq 2}$; and so we then conclude $M_{c(t)}^{(n)}(\pi, 5) = 0$. This then completes the whole proof, as desired.
\end{proof} 
We now wish to generalize Theorem \ref{5.1} to any polynomial map $\varphi_{p-1, c}$ for any given prime $p\geq 5$. More precisely, we prove that the number of distinct $n$-periodic points of any map $\varphi_{p-1, c}$ modulo $\pi$ is  $1$ or $2$ or zero:

\begin{thm} \label{5.2}
Let $p\geq 5$ be any fixed prime integer, and consider a family of polynomial maps $\varphi_{p-1, c}$ defined by $\varphi_{p-1, c}(z) = z^{p-1}+c$ for all points $c, z\in\mathbb{F}_{p}[t]$. Let $M_{c(t)}^{(n)}(\pi, p)$ be the number defined as in \textnormal{(\ref{M_{ct}})}. Then $M_{c(t)}^{(n)}(\pi, p) = 1$ for every $c\equiv \pm 1 \ (mod \ \pi)$ and fixed (even) integer $n$ or $M_{c(t)}^{(n)}(\pi, p) = 2$ for every $c\in (\pi)$; otherwise $M_{c(t)}^{(n)}(\pi, p) = 0$ for any $c \equiv -1\ (mod \ \pi)$ and fixed odd $n$ or any $c\not \equiv \pm 1, 0 \ (mod \ \pi)$ and fixed (even) $n$.
\end{thm}
\begin{proof}
By applying a similar argument as in the Proof of Theorem \ref{5.1}, we then obtain the count as desired. That is, let $g_{c(t)}(z)= \varphi_{p-1,c}^n(z)-z = \varphi_{p-1,c}(\varphi_{p-1,c}^{n-1}(z)) - z = (\varphi_{p-1,c}^{n-1}(z))^{p-1} - z + c$, and so $g_{c(t)}(z)= (\varphi_{p-1,c}^{n-1}(z))^{p-1} - z + c$. So now, applying multinomial theorem repeatedly on $(\varphi_{p-1,c}^{n-1}(z))^{p-1}$ after applying binomial theorem on $(z^{p-1} + c)^{p-1}$, it then follows $(\varphi_{p-1,c}^{n-1}(z))^{p-1}$ is a monic polynomial in $z$ of degree $(p-1)^n$ with $\mathbb{F}_{p}[t]$-coefficients in multiples of $c$. Hence, we may then write $(\varphi_{p-1,c}^{n-1}(z))^{p-1} = z^{(p-1)^{n}} + h(z)$, where $h(z)$ is a non-constant polynomial in $z$ of deg$(h)<(p-1)^n$ with $\mathbb{F}_{p}[t]$-coefficients in multiples of $c$; and so $g_{c(t)}(z)= z^{(p-1)^{n}} + h(z) - z + c$. Now for every coefficient $c\in (\pi) = \pi\mathbb{F}_{p}[t]$, reducing $g_{c(t)}(z)$ modulo prime $\pi$, it then follows $g_{c(t)}(z)\equiv z^{(p-1)^n} - z$ (mod $\pi$), since also $h(z)\in c\mathbb{F}_{p}[t][z]$ and so $h(z)\equiv 0$ (mod $\pi$); and so now $g_{c(t)}(z)$ modulo $\pi$ is a polynomial defined over a finite field $\mathbb{F}_{p}[t]\slash (\pi)$. So now, recall (from a well-known fact) the inclusion $\mathbb{F}_{p}\hookrightarrow \mathbb{F}_{p}[t]\slash (\pi)$ of fields; and also recall (from a well-known fact) that $h(x)=x^{p-1}-1$ vanishes at every nonzero $z\in \mathbb{F}_{p}\subset \mathbb{F}_{p}[t]\slash (\pi)$ and so $z^{p-1} = 1$ for every $z\in \mathbb{F}_{p}^{\times} = \mathbb{F}_{p}\setminus \{0\}$. It then follows $z^{(p-1)^n}= (z^{p-1})^{(p-1)^{n-1}} = 1$ for every element $z\in \mathbb{F}_{p}^{\times}$; and so $g(z)\equiv 1 - z$ (mod $\pi$) for every point $z\in \mathbb{F}_{p}^{\times}\subset \mathbb{F}_{p}[t]\slash (\pi)$ and so $g(z)$ has a root in $\mathbb{F}_{p}[t]\slash (\pi)$, namely, $z\equiv 1$ (mod $\pi$). Moreover, since $z$ is a linear factor of $g(z)\equiv z(z^{(p-1)^n-1} - 1)$ (mod $\pi$), it then follows $z\equiv 0$ (mod $\pi$) is also root of $g(z)$ modulo $\pi$. But now we then conclude $M_{c(t)}^{(n)}(\pi, p) = 2$. To see $M_{c(t)}^{(n)}(\pi, p) = 1$ for every coefficient $c\equiv 1$ (mod $\pi$) and every fixed $n\in \mathbb{Z}_{\geq 2}$, we first note that writing $\varphi_{p-1,c}^{n-1}(z) = \underbrace{((((z^{p-1} + c)^{p-1} + c)^{p-1} + c)^{p-1} + \cdots + c)^{p-1} + c}_\text{$(n-1)$ times}$ and then reducing $\varphi_{p-1,c}^{n-1}(z)$ modulo $\pi$ along with $c\equiv 1$ (mod $\pi$) and also using $z^{p-1} = 1$ for every $z\in \mathbb{F}_{p}^{\times}$, we then obtain $\varphi_{p-1,c}^{n-1}(z)\equiv 2$ (mod $\pi$) and $(\varphi_{p-1,c}^{n-1}(z))^{p-1}\equiv 1$ (mod $\pi$) for every fixed $n$, since $2^{p-1} = 1$ in $\mathbb{F}_{p}^{\times}$. But then $g_{c(t)}(z)=(\varphi_{p-1,c}^{n-1}(z))^{p-1} - z + c\equiv 2-z$ (mod $\pi$) for every point $z\in \mathbb{F}_{p}^{\times}\subset \mathbb{F}_{p}[t]\slash (\pi)$ and so $g_{c(t)}(z)$ modulo $\pi$ has a root in $\mathbb{F}_{p}[t]\slash (\pi)$, namely, $z\equiv 2$ (mod $\pi$); and so we conclude $M_{c(t)}^{(n)}(\pi, p) = 1$. We now show $M_{c(t)}^{(n)}(\pi, p) = 1$ for every coefficient $c\equiv -1$ (mod $\pi$) and every fixed even integer $n\in \mathbb{Z}_{\geq 2}$. As before, since $c\equiv -1$ (mod $\pi$) and also since $z^{p-1} = 1$ for every $z\in \mathbb{F}_{p}^{\times}$, reducing $\varphi_{p-1,c}^n(z)$ modulo $\pi$, we then obtain $\varphi_{p-1,c}^{n}(z)\equiv -1$ (mod $\pi$) for every fixed even $n\in \mathbb{Z}_{\geq 2}$. But then $g_{c(t)}(z)= \varphi_{p-1,c}^n(z)-z \equiv -(1+z)$ (mod $\pi$) for every point $z\in \mathbb{F}_{p}^{\times}$ and so $g_{c(t)}(z)$ modulo $\pi$ has a root in $\mathbb{F}_{p}[t]\slash (\pi)$, namely, $z\equiv -1$ (mod $\pi$); and so we then conclude $M_{c(t)}^{(n)}(\pi, p) = 1$. 

Finally, we now show $M_{c(t)}^{(n)}(\pi, p) = 0$ for every coefficient $c \equiv -1$ (mod $\pi$) and every fixed odd integer $n\in \mathbb{Z}_{\geq 3}$ or for every coefficient $c\not \equiv \pm1, 0$ (mod $\pi$) and every fixed (even) integer $n\in \mathbb{Z}_{\geq 2}$. To see this, we note that since $c\equiv -1$ (mod $\pi$) and allso since $z^{p-1} = 1$ for every $z\in \mathbb{F}_{p}^{\times}$, reducing $\varphi_{p-1,c}^n(z)$ modulo $\pi$, we then obtain $\varphi_{p-1,c}^{n}(z)\equiv 0$ (mod $\pi$) for every fixed odd  $n\in \mathbb{Z}_{\geq 3}$; and so $g_{c(t)}(z)= \varphi_{p-1,c}^n(z)-z \equiv -z$ (mod $\pi$) for every point $z\in \mathbb{F}_{p}^{\times}\subset \mathbb{F}_{p}[t]\slash (\pi)$ and fixed odd integer $n\geq 3$. But then we note $z\equiv 0$ (mod $\pi$) is a root of $g_{c(t)}(z)$ modulo $\pi$ for every coefficient $c\equiv -1$ (mod $\pi$) and every fixed odd $n\in \mathbb{Z}_{\geq 3}$, and also $z\equiv 0$ (mod $\pi$) is also root of $g_{c(t)}(z)$ modulo $\pi$ for every coefficient $c\equiv 0$ (mod $\pi$) and every fixed odd $n\in \mathbb{Z}_{\geq 3}$, as seen earlier; and from which we then obtain a contradiction that $1 \equiv 0$ (mod $\pi$). This then means $g_{c(t)}(x)=\varphi_{p-1,c}^n(x)-x$ has no roots in $\mathbb{F}_{p}[t]\slash (\pi)$ for every coefficient $c\equiv -1$ (mod $\pi$) and every fixed odd $n\in \mathbb{Z}_{\geq 3}$; and so we then conclude $M_{c(t)}^{(n)}(\pi, p) = 0$. Otherwise, suppose $g_{c(t)}(z)\equiv 0$ (mod $\pi$) for some point $z\in \mathbb{F}_{p}[t]\slash (\pi)\setminus \{0\}$ and for every coefficient $c\not \equiv \pm1, 0$ (mod $\pi$) and every fixed (even) integer $n\in \mathbb{Z}_{\geq 3}$; and so (from earlier) $z^{(p-1)^n} + h(z)-z+c\equiv 0$ (mod $\pi$). So now, using that $z^{(p-1)^n}=1$ for every $z\in \mathbb{F}_{p}^{\times}$, we may then rewrite $z^{(p-1)^n} -z + h(z) +c\equiv 0$ (mod $\pi$) to obtain $(1-z) + (h(z)+c)\equiv 0$ (mod $\pi$). But now we note that $(1-z) + (h(z)+c)\equiv 0$ (mod $\pi$) can happen if $1-z\equiv 0$ (mod $\pi$) and also $h(z)+c\equiv 0$ (mod $\pi$). Moreover, recall from earlier $1-z\equiv 0$ (mod $\pi$) when $c\equiv 0$ (mod $\pi$); and which then contradicts $c\not \equiv \pm1, 0$ (mod $\pi$). It then also follows $g_{c(t)}(x)=\varphi_{p-1,c}^n(x)-x$ has no roots in $\mathbb{F}_{p}[t]\slash (\pi)$ for every coefficient $c\not \equiv \pm1, 0$ (mod $\pi$) and fixed (even) $n\in \mathbb{Z}_{\geq 2}$; and so we conclude $M_{c(t)}^{(n)}(\pi, p) = 0$. This then completes the whole proof, as desired. 
\end{proof}

Finally, we wish to generalize Theorem \ref{5.2} further to any $\varphi_{(p-1)^{\ell}, c}$ for any prime $p\geq 5$ and any $\ell\in \mathbb{Z}_{\geq 1}$. Specifically, we show the number of distinct $m$-periodic points of any $\varphi_{(p-1)^{\ell}, c}$ modulo $\pi$ is also $1$ or $2$ or zero:

\begin{thm} \label{5.3}
Let $p\geq 5$ be any fixed prime integer, and $\ell \geq 1$ be any integer. Let $\varphi_{(p-1)^{\ell}, c}$ be a polynomial map defined by $\varphi_{(p-1)^{\ell}, c}(z) = z^{(p-1)^{\ell}} + c$ for all $c, z\in\mathbb{F}_{p}[t]$, and $M_{c(t)}^{(n)}(\pi, p)$ be as in \textnormal{(\ref{M_{ct}})}. Then $M_{c(t)}^{(n)}(\pi, p) = 1$ for every coefficient $c\equiv \pm 1 \ (mod \ \pi)$ and fixed (even) integer $n$ or $M_{c(t)}^{(n)}(\pi, p) = 2$ for every $c\in (\pi)$; otherwise $M_{c(t)}^{(n)}(\pi, p) = 0$ for every point $c \equiv -1\ (mod \ \pi)$ and fixed odd $n$ or every $c\not \equiv \pm 1, 0 \ (mod \ \pi)$ and fixed (even) $n$.
\end{thm}
\begin{proof}
By applying a similar argument as in the Proof of Theorem \ref{5.2}, we then obtain the count as desired. That is, let $g_{c(t)}(z)= \varphi_{(p-1)^{\ell}, c}^n(z)-z = \varphi_{(p-1)^{\ell},c}(\varphi_{(p-1)^{\ell},c}^{n-1}(z))-z = (\varphi_{(p-1)^{\ell},c}^{n-1}(z))^{(p-1)^{\ell}}-z + c$, and so $g_{c(t)}(z) = (\varphi_{(p-1)^{\ell},c}^{n-1}(z))^{(p-1)^{\ell}}-z + c$. Now applying multinomial theorem on the term $(\varphi_{p-1,c}^{n-1}(z))^{(p-1)^{\ell}}$, we then obtain $(\varphi_{p-1,c}^{n-1}(z))^{(p-1)^{\ell}}$ is a monic polynomial in $z$ of degree $(p-1)^{n\ell}$ with $\mathbb{F}_{p}[t]$-coefficients in multiples of $c$. Thus, we may then write $(\varphi_{(p-1)^{\ell},c}^{n-1}(z))^{(p-1)^{\ell}} = z^{(p-1)^{n\ell}} + h(z)$, where $h(z)$ is a non-constant polynomial in $z$ of deg$(h)<(p-1)^n$ with $\mathbb{F}_{p}[t]$-coefficients in multiples of $c$; and so $g_{c(t)}(z)= z^{(p-1)^{n\ell}} + h(z) - z + c$. So now, for every coefficient $c \in (\pi)$, reducing $g_{c(t)}(z)$ modulo prime $\pi$, it then follows $g_{c(t)}(z)\equiv z^{(p-1)^{n\ell}} - z$ (mod $\pi$); and so now $g_{c(t)}(z)$ modulo $p\mathcal{O}_{K}$ is a polynomial defined over a finite field $\mathbb{F}_{p}[t]\slash (\pi)$. Now recall the inclusion $\mathbb{F}_{p}\hookrightarrow\mathbb{F}_{p}[t]\slash (\pi)$ of fields and also that $z^{p-1} = 1$ for every element $z\in \mathbb{F}_{p}^{\times}$, it then also follows $z^{(p-1)^{n\ell}} = 1$ for every $z\in \mathbb{F}_{p}^{\times}\subset \mathbb{F}_{p}[t]\slash (\pi)$ and for every $\ell \in \mathbb{Z}_{\geq 1}$ and every $n \in \mathbb{Z}_{\geq 2}$. But then $g_{c(t)}(z)\equiv 1 - z$ (mod $\pi$) for every point $z\in \mathbb{F}_{p}^{\times}\subset\mathbb{F}_{p}[t]\slash (\pi)$ and so $g_{c(t)}(z)$ modulo $\pi$ has a root in $\mathbb{F}_{p}[t]\slash (\pi)$, namely, $z\equiv 1$ (mod $\pi$). Moreover, since $z$ is a linear factor of $g_{c(t)}(z)\equiv z(z^{(p-1)^{n\ell}-1} - 1)$ (mod $\pi$), it then follows $z\equiv 0$ (mod $\pi$) is also a root of $g_{c(t)}(z)$ modulo $\pi$ in $\mathbb{F}_{p}[t]\slash (\pi)$. But now we then conclude  $M_{c(t)}^{(n)}(\pi, p) = 2$. To see $M_{c(t)}^{(n)}(\pi, p) = 1$ for every coefficient $c\equiv 1$ (mod $\pi$) and for every fixed $\ell \in \mathbb{Z}_{\geq 1}$ and $n\in \mathbb{Z}_{\geq 2}$, we note that writing $\varphi_{(p-1)^{\ell},c}^{n-1}(z) = \underbrace{((((z^{(p-1)^{\ell}} + c)^{(p-1)^{\ell}} + c)^{(p-1)^{\ell}} + c)^{(p-1)^{\ell}} + \cdots + c)^{(p-1)^{\ell}} + c}_\text{$(n-1)$ times}$ and reducing $\varphi_{(p-1)^{\ell},c}^{n-1}(z)$ modulo $\pi$ along with $c\equiv 1$ (mod $p\mathcal{O}_{K}$) and also since $z^{(p-1)^{\ell}} = 1$ for every $z\in \mathbb{F}_{p}^{\times}$, we then obtain $\varphi_{(p-1)^{\ell},c}^{n-1}(z)\equiv 2$ (mod $\pi$) and $(\varphi_{(p-1)^{\ell},c}^{n-1}(z))^{(p-1)^{\ell}}\equiv 1$ (mod $\pi$) for any fixed $\ell$ and $n$, since also $2^{(p-1)^{\ell}} = 1$ in $\mathbb{F}_{p}$. But then $g_{c(t)}(z)=(\varphi_{(p-1)^{\ell},c}^{n-1}(z))^{(p-1)^{\ell}} - z + c\equiv 2-z$ (mod $\pi$) for every point $z\in \mathbb{F}_{p}^{\times}\subset \mathbb{F}_{p}[t]\slash (\pi)$ and so $g_{c(t)}(z)$ modulo $p\mathcal{O}_{K}$ has a root in $\mathbb{F}_{p}[t]\slash (\pi)$, namely, $z\equiv 2$ (mod $\pi$); and so we conclude $M_{c(t)}^{(n)}(\pi, p) = 1$. We now show $M_{c(t)}^{(n)}(\pi, p) = 1$ for every coefficient $c\equiv -1$ (mod $\pi$) and every fixed $\ell \in \mathbb{Z}_{\geq 1}$ and fixed even integer $n\in \mathbb{Z}_{\geq 2}$. Since $c\equiv -1$ (mod $\pi$) and $z^{(p-1)^{\ell}} = 1$ for every $z\in \mathbb{F}_{p}^{\times}$, reducing $\varphi_{(p-1)^{\ell},c}^n(z)$ modulo $\pi$, we then obtain $\varphi_{(p-1)^{\ell},c}^{n}(z)\equiv -1$ (mod $\pi$) for every fixed even $n$. But then $g_{c(t)}(z)= \varphi_{(p-1)^{\ell},c}^n(z)-z \equiv -(1+z)$ (mod $\pi$) for every $z\in \mathbb{F}_{p}^{\times}\subset \mathbb{F}_{p}[t]\slash (\pi)$ and so $g_{c(t)}(z)$ modulo $\pi$ has a root in $\mathbb{F}_{p}[t]\slash (\pi)$; and so we conclude $M_{c(t)}^{(n)}(\pi, p) = 1$. 

Finally, we now show $M_{c(t)}^{(n)}(\pi, p) = 0$ for every coefficient $c \equiv -1$ (mod $\pi$) and every fixed odd integer $n\in \mathbb{Z}_{\geq 3}$ or for every coefficient $c\not \equiv \pm1, 0$ (mod $\pi$) and every fixed (even) integer $n\in \mathbb{Z}_{\geq 2}$ and every $\ell \in \mathbb{Z}_{ \geq 1}$. To see this, we note as before that since $c\equiv -1$ (mod $\pi$) and $z^{(p-1)^{\ell}} = 1$ for every $z\in \mathbb{F}_{p}^{\times}$, reducing $\varphi_{(p-1)^{\ell},c}^n(z)$ modulo $\pi$, we then obtain $\varphi_{(p-1)^{\ell},c}^{n}(z)\equiv 0$ (mod $\pi$) for every fixed $\ell$ and fixed odd $n$; and so $g_{c(t)}(z)= \varphi_{(p-1)^{\ell},c}^n(z)-z \equiv -z$ (mod $\pi$) for every point $z\in \mathbb{F}_{p}^{\times}\subset \mathbb{F}_{p}[t]\slash (\pi)$. But then we note $z\equiv 0$ (mod $\pi$) is a root of $g_{c(t)}(z)$ modulo $\pi$ for every coefficient $c\equiv -1$ (mod $\pi$) and every fixed $\ell \in \mathbb{Z}_{ \geq 1}$ and fixed odd $n \in \mathbb{Z}_{ \geq 3}$, and also $z\equiv 0$ (mod $\pi$) is also root of $g_{c(t)}(z)$ modulo $\pi$ for every coefficient $c\equiv 0$ (mod $\pi$) and every fixed $\ell \in \mathbb{Z}_{ \geq 1}$ and fixed odd $n \in \mathbb{Z}_{ \geq 3}$, as seen earlier; and from which then follows a contradiction that $1 \equiv 0$ (mod $\pi$). This then means $g_{c(t)}(x)=\varphi_{(p-1)^{\ell},c}^n(x)-x$ has no roots in $\mathbb{F}_{p}[t]\slash (\pi)$ for every coefficient $c\equiv -1$ (mod $\pi$) and every fixed $\ell \in \mathbb{Z}_{ \geq 1}$ and fixed odd $n \in \mathbb{Z}_{ \geq 3}$; and so we conclude $M_{c(t)}^{(n)}(\pi, p) = 0$. Otherwise, suppose $g_{c(t)}(z)\equiv 0$ (mod $\pi$) for some point $z\in \mathbb{F}_{p}[t]\slash (\pi)\setminus \{0\}$ and every coefficient $c\not \equiv \pm1, 0$ (mod $\pi$) and every fixed $\ell \in \mathbb{Z}_{ \geq 1}$ and fixed (even) $n \in \mathbb{Z}_{ \geq 2}$; and so (from earlier) $z^{(p-1)^{n\ell}} + h(z)-z+c\equiv 0$ (mod $\pi$). So then, using that $z^{(p-1)^{n\ell}}=1$ for every $z\in \mathbb{F}_{p}^{\times}$ and for every fixed $\ell$ and (even) $n$, we then rewrite $z^{(p-1)^{n\ell}} -z + h(z) +c\equiv 0$ (mod $\pi$) to obtain $(1-z) + (h(z)+c)\equiv 0$ (mod $\pi$). But now we note that the congruence $(1-z) + (h(z)+c)\equiv 0$ (mod $\pi$) can happen if $1-z\equiv 0$ (mod $\pi$) and also $h(z)+c\equiv 0$ (mod $\pi$). Moreover, recall from earlier that $1-z\equiv 0$ (mod $\pi$) when $c\equiv 0$ (mod $\pi$); which then contradicts $c\not \equiv \pm1, 0$ (mod $\pi$). It then also follows $g_{c(t)}(x)=\varphi_{(p-1)^{\ell},c}^n(x)-x$ has no roots in $\mathbb{F}_{p}[t]\slash (\pi)$ for every coefficient $c\not \equiv \pm1, 0$ (mod $\pi$) and for every fixed $\ell \in \mathbb{Z}_{ \geq 1}$ and fixed (even) integer $n \in \mathbb{Z}_{ \geq 2}$; and so we then conclude $M_{c(t)}^{(n)}(\pi, p) = 0$. This then completes the whole proof, as desired.
\end{proof}

\begin{rem}\label{rem5.4}
With now Theorem \ref{5.3}, we may also to each distinct $n$-periodic point of $\varphi_{(p-1)^{\ell},c}$ associate an $n$-periodic orbit. In doing so, we then also obtain a dynamical translation of Theorem \ref{5.3}, namely, that the number of distinct $n$-periodic orbits that any $\varphi_{(p-1)^{\ell},c}$ has when iterated on the space $\mathbb{F}_{p}[t]\slash (\pi)$ is $1$ or $2$ or $0$. Again, as mentioned in Intro.\ref{sec1} that the count obtained in Theorem \ref{5.3} on the number of distinct $n$-periodic points of any $\varphi_{(p-1)^{\ell},c}$ modulo $\pi$ is independent of $p$ (and so independent of deg$(\varphi_{(p-1)^{\ell},c})$ for every fixed $\ell \in \mathbb{Z}_{\geq 1}$) and $m=$ deg$(\pi)$ in each of the possibilities. Moreover, we may also observe that the expected total number (namely, $1 + 2 + 0 =3$ for every fixed odd period $n\in \mathbb{Z}_{\geq 3}$ or $1 + 1 + 2 + 0 =4$ for every fixed even period $n\in \mathbb{Z}_{\geq 2}$) of distinct $n$-periodic points in the whole family of maps $\varphi_{(p-1)^{\ell},c}$ modulo $\pi$ is not only also independent of $p$ (and so independent of deg$(\varphi_{(p-1)^{\ell},c})$ for every fixed $\ell \in \mathbb{Z}_{\geq 1}$) and $m$, but is also a constant equal to $3$ or $4$ even when $(p-1)^{\ell}$ or $m\to \infty$; a somewhat interesting phenomenon differing significantly from a phenomenon in Remark \ref{rem4.4}, however, coinciding somewhat surprising with [\cite{BK3, BK222}, Rem.6.4 and 3.5, resp.].  
\end{rem}

\begin{rem}
As before, recall in \cite{BK3} we proved that the function $M_{c(t)}(\pi, p) = 1, 2$ or $0$ for every fixed prime $p\geq 5$ and every coefficient $c\equiv 1, 0$ (mod $\pi$) or $c\equiv -1$ (mod $\pi$). Moreover, recall in the Proof of Theorem \ref{5.3} we proved that for every fixed odd (period) $n\in \mathbb{Z}_{\geq 3}$, the points $z\equiv 1, 0, 2$ (mod $\pi$) are $n$-periodic points of any $\varphi_{(p-1)^{\ell},c}$ modulo $\pi$ (which we (in \cite{BK3}) also obtained as fixed points of any $\varphi_{(p-1)^{\ell},c}$ modulo $\pi$) for every fixed prime $p\geq 5$ and every coefficient $c\equiv 1, 0$ (mod $\pi$); from which we then concluded $M_{c(t)}^{(n)}(\pi, p)  = 1, 2$ or $0$ for every fixed $p$ and  every $c\equiv 1, 0$ (mod $\pi$) or $c\equiv -1$ (mod $\pi$). But now for every fixed odd (period) $n\in \mathbb{Z}_{\geq 3}$, we then note that the counting function $M_{c(t)}^{(n)}(\pi, p) = M_{c(t)}(\pi, p) = 1, 2$ or $0$ for every fixed $p$ and every $c\equiv 1, 0$ (mod $\pi$) or $c\equiv -1$ (mod $\pi$). Consequently, for every fixed odd (period) $n\in \mathbb{Z}_{\geq 1}$, we then note that every $n$-periodic orbit of every $\varphi_{(p-1)^{\ell},c}$ modulo $\pi$ is a fixed orbit, and moreover every reduced map $\varphi_{(p-1)^{\ell},c}$ modulo $\pi$ has $1$ or $2$ or no fixed orbits; a somewhat interesting precise arithmetic-geometric insight on all odd $n$-periodic orbits of every $\varphi_{(p-1)^{\ell},c}$ modulo $\pi$. Furthermore, recall in the Proof of Theorem \ref{5.3} that for every fixed even (period) $n\in \mathbb{Z}_{\geq 2}$, the points $z\equiv 1, 0, 2, -1$ (mod $\pi$) are $n$-periodic points of any $\varphi_{(p-1)^{\ell},c}$ modulo $\pi$ for every fixed prime $p\geq 5$ and every coefficient $c\equiv \pm 1, 0$ (mod $\pi$) or $c\not \equiv \pm 1, 0$ (mod $\pi$); from which we then concluded $M_{c(t)}^{(n)}(\pi, p)  = 1, 2$ or $0$ for every fixed $p$ and every $c\equiv \pm 1, 0$ (mod $\pi$) or $c\not \equiv \pm 1, 0$ (mod $\pi$). But now for every fixed even (period) $n\in \mathbb{Z}_{\geq 4}$, we then also note that the counting function $M_{c(t)}^{(n)}(\pi, p) = M_{c(t)}^{(2)}(\pi, p) = 1, 2$ or $0$ for every fixed $p$ and every $c\equiv \pm 1, 0$ (mod $\pi$) or $c\not \equiv \pm 1, 0$ (mod $\pi$). As before, we now also note that for every fixed even (period) $n\in \mathbb{Z}_{\geq 2}$, every $n$-periodic orbit of any $\varphi_{(p-1)^{\ell},c}$ modulo $\pi$ is a $2$-periodic orbit, and moreover every reduced map $\varphi_{(p-1)^{\ell},c}$ modulo $\pi$ has one or two or no $n$-periodic orbits; and again a somewhat interesting precise arithmetic-geometric insight on all even $n$-periodic orbits of every map $\varphi_{(p-1)^{\ell},c}$ modulo $\pi$.    
\end{rem}

\section{On Dynamical Complexity of Forward Periodic Orbit Structure of $\overline{\varphi_{p^{\ell},c}}$ and $\overline{\varphi_{(p-1)^{\ell},c}}$}
Observe in Theorem \ref{2.3} that the number $X_{c}^{(n)}(p)$ is independent of the period $n$, and moreover we may have 
\begin{center}
    $\lim\limits_{n\to \infty} X_{c}^{(n)}(p) = p$ or $0$.
\end{center}

Now as mentioned in \cite{BK111, BK222} that one of the main objectives in classical dynamical systems is to understand \textit{all} orbits usually via topological and analytic techniques; and moreover in doing so, one may not only find that orbits can easily get very complicated but also important and interesting statistical questions (e.g., determining topological entropy) concerning measuring the complexity of the underlying system, may become intractable. So now, motivated (as in \cite{BK111, BK222}) by such a statistic (namely, topological entropy) and moreover since we may now also view $X_{c}^{(n)}(p)$ as $n$-periodic orbit-counting function, we in this section also wish to investigate very mildly the complexity of a polynomial discrete dynamical system $(\mathbb{Z}_{p}\slash p\mathbb{Z}_{p}$, $\varphi_{p^{\ell},c}$ modulo $p\mathbb{Z}_{p}$). With that in mind, we again analyze the behavior of the associated exponential growth rate $\rho(\overline{\varphi_{p^{\ell},c}})$ [\cite{Kat}, Page 196], where $\overline{\varphi_{p^{\ell},c}}:\mathbb{Z}_{p}\slash p\mathbb{Z}_{p} \to \mathbb{Z}_{p}\slash p\mathbb{Z}_{p}$ is the reduced polynomial map $\varphi_{p^{\ell},c}$ modulo $p\mathbb{Z}_{p}$; and in doing so we then also obtain immediately the following corollary showing that $n$-periodic orbit-counting function $X_{c}^{(n)}(p)$ grows by a factor $1=e^{\rho(\varphi_{p^{\ell},c})}$ as period $n\to \infty$ and thus $n$-periodic orbit-counting function $X_{c}^{(n)}(p)$ is a constant function: 
\begin{cor}\label{co6.1}
Assume Theorem \ref{2.3}, and let $n\in \mathbb{Z}_{\geq 2}$ be any period. Then the exponential growth rate of $n$-periodic orbit-counting function $X_{c}^{(n)}(p)$ exists and is equal to zero. More precisely, we have 

\begin{center}
    $\rho(\overline{\varphi_{p^{\ell},c}}) := \limsup\limits_{n\to \infty}\frac{\textnormal{log}(\textnormal{max} \{X_{c}^{(n)}(p),1\})}{n} = 0$.
\end{center} 
\end{cor}
\begin{proof}
Since we know from Theorem \ref{2.3} that the number $X_{c}^{(n)}(p) = p\text{ or } 0$ for every fixed period $n$, we then obtain $\frac{\text{log}(\text{max} \{X_{c}^{(n)}(p),1\})}{n} = \frac{\text{log }p}{n}$ or $0$. Now letting period $n\to \infty$, we then obtain $\rho(\overline{\varphi_{p^{\ell},c}}) = 0$ as desired.
\end{proof}

As before, we may also observe in Theorem \ref{3.3} that $Y_{c}^{(n)}(p)$ is independent of period $n$, and more to this  
\begin{center}
    $\lim\limits_{n\to \infty} Y_{c}^{(n)}(p) = 1, 2 \text{ or } 0.$
\end{center}

So now, as before we may also study very mildly the complexity of a polynomial discrete dynamical system $(\mathbb{Z}_{p}\slash p\mathbb{Z}_{p}$, $\varphi_{(p-1)^{\ell},c}$ modulo $p\mathbb{Z}_{p}$) by also determining the behavior of the associated exponential growth rate $\rho(\overline{\varphi_{(p-1)^{\ell},c}})$, where $\overline{\varphi_{(p-1)^{\ell},c}}$ is the polynomial map $\varphi_{(p-1)^{\ell},c}$ modulo $p\mathbb{Z}_{p}$. In doing so, we then also obtain immediately the following corollary showing that $n$-periodic orbit-counting function $Y_{c}^{(n)}(p)$ grows by a factor $1=e^{\rho(\overline{\varphi_{(p-1)^{\ell},c}})}$ as period $n\to \infty$ and so $n$-periodic orbit-counting function $Y_{c}^{(n)}(p)$ is also a constant function: 
\begin{cor}\label{co6.2}
Assume Theorem \ref{3.3}, and let $n\in \mathbb{Z}_{\geq 2}$ be any period. Then the exponential growth rate of $n$-periodic orbit-counting function $Y_{c}^{(n)}(p)$ exists and is equal to zero. More precisely, we have 
\begin{center}
    $\rho(\overline{\varphi_{(p-1)^{\ell},c}}) := \limsup\limits_{n\to \infty}\frac{\textnormal{log}(\textnormal{max} \{Y_{c}^{(n)}(p),1\})}{n} = 0$.
\end{center}
\end{cor}
\begin{proof}
Since we know from Theorem \ref{3.3} that the number $Y_{c}^{(n)}(p) = 1, 2 \text{ or } 0$ for every fixed period $n$, we then obtain $\frac{\text{log}(\text{max} \{Y_{c}^{(n)}(p),1\})}{n} = \frac{\text{log }2}{n}$ or $0$. But now letting $n\to \infty$, we then obtain  $\rho(\overline{\varphi_{(p-1)^{\ell},c}}) = 0$ as desired.
\end{proof}

Similarly, we also observe in Theorem \ref{4.3} that $N_{c(t)}^{(n)}(\pi, p)$ is independent of period $n$, and more to this  
\begin{center}
    $\lim\limits_{n\to \infty} N_{c(t)}^{(n)}(\pi, p) = p \text{ or } 0.$
\end{center}

As before, we may also wish to investigate very mildly the complexity of a polynomial discrete dynamical system $(\mathbb{F}_{p}[t]\slash (\pi)$, $\varphi_{p^{\ell},c(t)}$ modulo $\pi$) by also determining the behavior of the associated exponential growth rate $\rho(\overline{\varphi_{p^{\ell},c(t)}})$, where $\overline{\varphi_{p^{\ell},c(t)}}:\mathbb{F}_{p}[t]\slash (\pi)\to \mathbb{F}_{p}[t]\slash (\pi)$ is the reduced polynomial map $\varphi_{p^{\ell},c(t)}$ modulo $\pi$. With that in mind, we then obtain the following corollary showing that $n$-periodic orbit-counting function $N_{c(t)}^{(n)}(\pi, p)$ grows by a factor $1=e^{\rho(\overline{\varphi_{p^{\ell},c(t)}})}$ as period $n\to \infty$ and hence the function $N_{c(t)}^{(n)}(\pi, p)$ is also a constant function:

\begin{cor}\label{co6.3}
Assume Theorem \ref{4.3}, and let $n\in \mathbb{Z}_{\geq 2}$ be any period. Then the exponential growth rate of $n$-periodic orbit-counting function $N_{c(t)}^{(n)}(\pi, p)$ exists and is equal to zero. More precisely, we have 
\begin{center}
    $\rho(\overline{\varphi_{p^{\ell},c(t)}}) = \limsup\limits_{n\to \infty}\frac{\textnormal{log}(\textnormal{max} \{N_{c(t)}^{(n)}(\pi, p),1\})}{n} = 0$.
\end{center}
\end{cor}
\begin{proof}
By applying a similar argument as in the Proof of Corollary \ref{co6.1}, we then obtain the limit as desired. That is, since we know from Theorem \ref{4.3} that $N_{c(t)}^{(n)}(\pi, p) = p\text{ or } 0$ for every fixed period $n$, we then obtain $\frac{\text{log}(\text{max} \{N_{c(t)}^{(n)}(\pi, p),1\})}{n} = \frac{\text{log }p}{n}$ or $0$. But then letting period $n\to \infty$, we then obtain $\rho(\overline{\varphi_{p^{\ell},c(t)}}) = 0$ as desired.    
\end{proof}

As before, we also observe in Theorem \ref{5.3} that $M_{c(t)}^{(n)}(\pi, p)$ is also independent of $n$, and again we have   
\begin{center}
    $\lim\limits_{n\to \infty} M_{c(t)}^{(n)}(\pi, p) = 1, 2 \text{ or } 0.$
\end{center}

Similarly, we may also investigate very mildly the complexity of a polynomial discrete dynamical system $(\mathbb{F}_{p}[t]\slash (\pi)$, $\varphi_{(p-1)^{\ell},c(t)}$ modulo $\pi$) by also determining the behavior of the associated exponential growth rate $\rho(\overline{\varphi_{(p-1)^{\ell},c(t)}})$, where $\overline{\varphi_{(p-1)^{\ell},c(t)}}$ is the reduced polynomial map $\varphi_{(p-1)^{\ell},c(t)}$ modulo $\pi$. In doing so, we then also obtain immediately the following corollary showing that $n$-periodic orbit-counting function $M_{c(t)}^{(n)}(\pi, p)$ grows by a factor $1=e^{\rho(\overline{\varphi_{(p-1)^{\ell},c(t)}})}$ as period $n\to \infty$ and hence the orbit-counting function $M_{c(t)}^{(n)}(\pi, p)$ is also constant:

\begin{cor}
Assume Theorem \ref{5.3}, and let $n\in \mathbb{Z}_{\geq 2}$ be any period. Then the exponential growth rate of $n$-periodic orbit-counting function $M_{c(t)}^{(n)}(\pi, p)$ exists and is equal to zero. More precisely, we have 
\begin{center}
    $\rho(\overline{\varphi_{(p-1)^{\ell},c(t)}}) = \limsup\limits_{n\to \infty}\frac{\textnormal{log}(\textnormal{max} \{M_{c(t)}^{(n)}(\pi, p),1\})}{n} = 0$.
\end{center}
\end{cor}

\begin{proof}
By applying a similar argument as in the Proof of Corollary \ref{co6.3}, we then obtain the limit as desired. That is, since we know from Theorem \ref{5.3} that $M_{c(t)}^{(n)}(\pi, p) = 1, 2 \text{ or } 0$ for every fixed period $n$, we then obtain $\frac{\text{log}(\text{max} \{M_{c(t)}^{(n)}(\pi, p),1\})}{n} = \frac{\text{log }2}{n}$ or $0$. So now letting period $n\to \infty$, it then follows $\rho(\overline{\varphi_{(p-1)^{\ell},c(t)}}) = 0$ as needed.    
\end{proof}

\section{On Average Number of $n$-Periodic $\mathbb{Z}_{p}\slash p\mathbb{Z}_{p}$-Points of any Family of $\varphi_{p^{\ell},c}$ and $\varphi_{(p-1)^{\ell},c}$ }\label{sec6}

In this section, we wish to inspect independently the behavior of the $n$-periodic point-counting functions $X_{c}^{(n)}(p)$ and $Y_{c}^{(n)}(p)$ as $c\to \infty$. First, we wish to determine: \say{\textit{What is the average value of $X_{c}^{(n)}(p)$ as $c \to \infty$?}} The following corollary shows that the average value of the function $X_{c}^{(n)}(p)$ may be zero or unbounded as $c\to \infty$:

\begin{cor}\label{7.1}
Let $p\geq 3$ be any prime, and $n\in \mathbb{Z}_{\geq 2}$ be any fixed (period). Then the average value of $n$-periodic point-counting function $X_{c}^{(n)}(p)$ is zero or unbounded as $c\to\infty$. More precisely, we have
\begin{myitemize}
    \item[\textnormal{(a)}] \textnormal{Avg} $X^{(n)}_{c\neq pt}(p):= \lim\limits_{c \to\infty} \Large{\frac{\sum\limits_{3\leq p\leq c, \ p\nmid c \textnormal{ in } \mathbb{Z}_{p}}X_{c}^{(n)}(p)}{\Large{\sum\limits_{3\leq p\leq c, \ p\nmid c \textnormal{ in } \mathbb{Z}_{p}}1}}} =  0$. 
    
    \item[\textnormal{(b)}] \textnormal{Avg} $X^{(n)}_{c = pt}(p):= \lim\limits_{c \to\infty} \Large{\frac{\sum\limits_{3\leq p\leq c, \ p\mid c \textnormal{ in } \mathbb{Z}_{p}}X_{c}^{(n)}(p)}{\Large{\sum\limits_{3\leq p\leq c, \ p\mid c \textnormal{ in } \mathbb{Z}_{p}}1}}} =  \infty$.
\end{myitemize}

\end{cor}
\begin{proof}
Since we know from Theorem \ref{2.3} that the number $X_{c}^{(n)}(p) = 0$ for any prime $p\nmid c$ in $\mathbb{Z}_{p}$, it then follows $\lim\limits_{c\to\infty} \Large{\frac{\sum\limits_{3\leq p\leq c, \ p\nmid c \text{ in } \mathbb{Z}_{p}}X_{c}^{(n)}(p)}{\Large{\sum\limits_{3\leq p\leq c, \ p\nmid c \text{ in } \mathbb{Z}_{p}}1}}} = 0$; and so we conclude that the average value Avg $X_{c \neq pt}^{(n)}(p) = 0$. To see (b), we recall from Theorem \ref{2.3} that the number $X_{c}^{(n)}(p) = p$ for any $p\mid c$ in $\mathbb{Z}_{p}$. But now we note $\sum\limits_{3\leq p\leq c, \ p\mid c \text{ in } \mathbb{Z}_{p}} X_{c}^{(n)}(p) = \sum\limits_{3\leq p\leq c, \ p\mid c \text{ in } \mathbb{Z}_{p}}p =: \sigma_{1,p}(c)$ and  $\sum\limits_{3\leq p\leq c, \ p\mid c \text{ in } \mathbb{Z}_{p}} 1  = \omega(c)$, where $\sigma_{1}(\ell)$ (resp. $\omega(\ell)$) is by definition the number of divisors (resp. the number of distinct prime divisors) of any $\ell\in \mathbb{Z}_{\geq 1}$; and so we then obtain $\frac{\sum\limits_{3\leq p\leq c, \ p\mid c \text{ in } \mathbb{Z}_{p}} X_{c}^{(n)}(p)}{\sum\limits_{3\leq p\leq c, \ p\mid c \text{ in } \mathbb{Z}_{p}} 1} = \frac{\sigma_{1,p}(c)}{\omega(c)}$. But now applying a similar argument as in [\cite{BK222}, Proof of Corollary 5.1], we then conclude the average value Avg $X_{c = pt}^{(n)}(p) = \infty$. This then completes the whole proof, as needed. 
\end{proof} 
\begin{rem}\label{re6.2}
From arithmetic statistics to arithmetic dynamics, we note that Corollary \ref{7.1} shows that any $\varphi_{p^{\ell},c}$ iterated on $\mathbb{Z}_{p} / p\mathbb{Z}_{p}$ has on average zero or unbounded number of distinct $n$-periodic $p$-adic integral orbits as $c\to \infty$; a somewhat interesting averaging phenomenon coinciding with the phenomenon in [\cite{BK3}, Remark 7.4 (when $\ell \in \{1, p\}$)] on the average number of distinct fixed $p$-adic integral orbits of any $\varphi_{p^{\ell},c}$ iterated on $\mathbb{Z}_{p} / p \mathbb{Z}_{p}$.
\end{rem}

Similarly, we also restrict on $\mathbb{Z}\subset \mathbb{Z}_{p}$ and then determine: \say{\textit{What is the average value of $Y_{c}^{(n)}(p)$ as $c \to \infty$?}} The following corollary shows that the average value of the function $Y_{c}^{(n)}(p)$ is $1$ or $2$ or 0 as $c\to \infty$:

\begin{cor}\label{cor7.3}
Let $p\geq 5$ be any prime, and $n\in \mathbb{Z}_{\geq 2}$ be any fixed (period). Then the average value of $n$-periodic point-counting function $Y_{c}^{(n)}(p)$ exists and is equal to $1$ or $2$ or $0$ as $c\to\infty$. More precisely, we have 
\begin{myitemize}
    \item[\textnormal{(a)}] \textnormal{Avg} $Y_{c\pm1 = pt}^{(n)}(p) := \lim\limits_{c\to\infty} \Large{\frac{\sum\limits_{5\leq p\leq c, \ p\mid (c\pm1) \textnormal{ in } \mathbb{Z}_{p}}Y_{c}^{(n)}(p)}{\Large{\sum\limits_{5\leq p\leq c, \ p\mid (c\pm1) \textnormal{ in } \mathbb{Z}_{p}}1}}} = 1.$ 

    \item[\textnormal{(b)}] \textnormal{Avg} $Y_{c= pt}^{(n)}(p) := \lim\limits_{c\to\infty} \Large{\frac{\sum\limits_{5\leq p\leq c, \ p\mid c \textnormal{ in } \mathbb{Z}_{p}}Y_{c}^{(n)}(p)}{\Large{\sum\limits_{5\leq p\leq c, \ p\mid c \textnormal{ in } \mathbb{Z}_{p}}1}}} = 2.$
    
     \item[\textnormal{(c)}] \textnormal{Avg} $Y_{c\not \equiv\pm1, 0 \ (\textnormal{mod }p)}^{(n)}(p):= \lim\limits_{c \to\infty} \Large{\frac{\sum\limits_{5\leq p\leq c, \ c\not \equiv\pm1, 0 \ (\textnormal{mod } p)}Y_{c}^{(n)}(p)}{\Large{\sum\limits_{5\leq p\leq c, \ c\not \equiv\pm1, 0 \ (\textnormal{mod } p)}1}}} =  0$.
\end{myitemize}

\end{cor}
\begin{proof}
By applying a similar argument as in the Proof of Corollary \ref{7.1}, we then obtain the limits, as desired.
\end{proof} 
\begin{rem} \label{re6.4}
As before, we also note that from arithmetic statistics to arithmetic dynamics, Corollary \ref{cor7.3} shows that any $\varphi_{(p-1)^{\ell},c}$ iterated on $\mathbb{Z}_{p} \slash p\mathbb{Z}_{p}$ has on average one or two or no $n$-periodic orbits as $c\to \infty$; a somewhat interesting averaging phenomenon coinciding precisely with an averaging phenomenon in [\cite{BK3}, Remark 7.6] on the average number of distinct fixed $p$-adic integral orbits of any $\varphi_{(p-1)^{\ell},c}$ iterated on $\mathbb{Z}_{p} \slash p \mathbb{Z}_{p}$.
\end{rem}

\section{On Average Number of $n$-Periodic $\mathbb{F}_{p}[t]\slash (\pi)$-Points of any Family of $\varphi_{p^{\ell},c}$ and $\varphi_{(p-1)^{\ell},c}$}\label{sec7}

As in Section \ref{sec6}, we also wish to inspect the asymptotic behavior of the function $N_{c(t)}^{(n)}(\pi, p)$ as \text{deg}$(c)\to \infty$. More precisely, we wish to determine: \say{\textit{What is the average value of the function $N_{c(t)}^{(n)}(\pi, p)$ as \text{deg}(c)$\to \infty$?}} The following corollary shows that the average value of $N_{c(t)}^{(n)}(\pi, p)$ is either zero or unbounded as \text{deg}$(c)\to \infty$:
\begin{cor}\label{co6}
Let $p\geq 3$ be any prime, and deg(c) $=\kappa\geq 3$ be any integer. Then the average value of $n$-periodic point-counting function $N_{c(t)}^{(n)}(\pi, p)$ is equal to zero or unbounded as $\kappa\to\infty$. Specifically, we have
\begin{myitemize}
    \item[\textnormal{(a)}] \textnormal{Avg} $N_{c(t)\neq \pi t}^{(n)}(p):= \lim\limits_{\kappa \to\infty} \Large{\frac{\sum\limits_{3\leq p\leq \kappa, \ \pi\nmid c \textnormal{ in } \mathbb{F}_{p}[t]}N_{c(t)}^{(n)}(\pi, p)}{\Large{\sum\limits_{3\leq p\leq \kappa, \ \pi\nmid c \textnormal{ in } \mathbb{F}_{p}[t]}1}}} =  0$. 
    
    \item[\textnormal{(b)}] \textnormal{Avg} $N_{c(t)= \pi t}^{(n)}(p):= \lim\limits_{\kappa \to\infty} \Large{\frac{\sum\limits_{3\leq p\leq \kappa, \ \pi \mid c \textnormal{ in } \mathbb{F}_{p}[t]}N_{c(t)}^{(n)}(\pi, p)}{\Large{\sum\limits_{3\leq p\leq \kappa, \ \pi\mid c \textnormal{ in } \mathbb{F}_{p}[t]}1}}} =  \infty$.
\end{myitemize}

\end{cor}

\begin{proof}
Applying a similar argument as in [\cite{BK3}, Proof of Corollary 7.1], we then obtain the limits as desired. That is, since from Theorem \ref{4.3} we know $N_{c(t)}^{(n)}(\pi, p) = 0$ for any prime $\pi$ such that $\pi \nmid c$ in $\mathbb{F}_{p}[t]$, we then obtain $\lim\limits_{\kappa\to\infty} \Large{\frac{\sum\limits_{3\leq p \leq \kappa, \ \pi \nmid c\text{ in }\mathbb{F}_{p}[t]}N_{c(t)}^{(n)}(\pi, p)}{\Large{\sum\limits_{3\leq p \leq \kappa, \ \pi\nmid c\text{ in }\mathbb{F}_{p}[t]}1}}} = 0$; and so we then conclude Avg $N_{c(t) \neq \pi t}^{(n)}(\pi, p) = 0$. To see \textnormal{(b)}, we recall from Theorem \ref{4.3} that $N_{c(t)}^{(n)}(\pi, p) = p$ for any prime $\pi$ such that $\pi \mid c$ in $\mathbb{F}_{p}[t]$. But now observe $\sum\limits_{3\leq p \leq \kappa, \ \pi \mid c \text{ in } \mathbb{F}_{p}[t]}N_{c(t)}^{(n)}(\pi, p) = \sum\limits_{3\leq p \leq \kappa, \ \pi \mid c \text{ in } \mathbb{F}_{p}[t]}p = \sigma_{1,\pi}(c)$ and  $\sum\limits_{3\leq p \leq \kappa, \ \pi\mid c \text{ in } \mathbb{F}_{p}[t], \ell \in \{1, p\}}1  =: \omega_{\pi}(c)$, where recalling from function field number theory [\cite{Rose}, Page 15] that the divisor function $\sigma_{1}(f)$ is defined as $\sigma_{1}(f)=\sum\limits_{g\mid f}|g|$ where $|g|=\# \mathbb{F}_{p}[t]\slash (g)$ for any monic polynomials $g, f\in \mathbb{F}_{p}[t]$ and setting deg $\pi = 1$ (and so the size $|\pi|=\# \mathbb{F}_{p}[t]\slash (\pi) = p$) we then have $\sigma_{1, \pi}(c):=\sigma_{1}(c) =\sum\limits_{3\leq p \leq \kappa, \ \pi \mid c \text{ in } \mathbb{F}_{p}[t]}|\pi| = \sum\limits_{3\leq p \leq \kappa, \ \pi \mid c \text{ in } \mathbb{F}_{p}[t]}p$. So now, since we are varying deg$(c) = \kappa$ (and hence varying $c=c(t)$) and so defining $\sigma_{1, \pi}(\kappa):=\sigma_{1, \pi}(c)$ and also $\omega(\kappa):= \omega_{\pi}(c)$, we then obtain  $\lim\limits_{\kappa\to\infty}\frac{\sum\limits_{3\leq p \leq \kappa, \ \pi \mid c \text{ in } \mathbb{F}_{p}[t]}N_{c(t)}^{(n)}(\pi, p)}{\sum\limits_{3\leq p \leq \kappa, \ \pi\mid c \text{ in } \mathbb{F}_{p}[t]}1} = \lim\limits_{\kappa\to\infty}\frac{\sigma_{1,\pi}(c)}{\omega_{\pi}(c)} = \lim\limits_{\kappa\to\infty}\frac{\sigma_{1,\pi}(\kappa)}{\omega_{\pi}(\kappa)}$. Now since the partial sum $\sum\limits_{3\leq p \leq \kappa, \ \pi \mid c \text{ in } \mathbb{F}_{p}[t]}p=\sigma_{1, \pi}(\kappa)$ and $\sum\limits_{3\leq p \leq c, \ p \mid c \text{ in } \mathbb{Z}}p=\sigma_{1, p}(c)$ are summed over the same  divisibility condition and moreover have the same summand, we then obtain $\sigma_{1, \pi}(c)=\sigma_{1, p}(c)$ and $\omega_{\pi}(c)=\omega(c)$ for each $c$; and from which it then follows $\frac{\sigma_{1, \pi}(c)}{\omega_{\pi}(c)}= \frac{\sigma_{1,p}(c)}{\omega(c)}$ for each $c$. But then recall from the Proof of Corollary \ref{7.1} that the quotient $\frac{\sigma_{1,p}(c)}{\omega(c)}\to \infty$ as a positive integer $c\to \infty$; and hence $\frac{\sigma_{1,\pi}(c)}{\omega_{\pi}(c)}\to \infty$ as deg$(c)=\kappa\to \infty$, as desired.
\end{proof}

\begin{rem} 
Again, we also note that from arithmetic statistics to arithmetic dynamics, Corollary \ref{co6} shows any $\varphi_{p^{\ell},c}$ iterated on the space $\mathbb{F}_{p}[t] / (\pi)$ has on average zero or unbounded number of distinct $n$-periodic orbits as \text{deg}$(c)\to \infty$; a somewhat interesting averaging phenomenon coinciding with the phenomenon about in [\cite{BK3}, Rem ark 8.2(when $\ell \in \{1, p\}$)]  on the average number of distinct fixed orbits of any $\varphi_{p^{\ell},c}$ iterated on $\mathbb{F}_{p}[t] / (\pi)$.
\end{rem}

Similarly, we also wish to determine: \say{\textit{What is the average value of $M_{c(t)}^{(n)}(\pi, p)$ as \text{deg}(c)$\to \infty$?}} The following corollary shows that the average value of $M_{c(t)}^{(n)}(\pi, p)$ exists and is equal to $1$ or $2$ or zero as \text{deg}$(c)\to \infty$:
\begin{cor}\label{cor6.1}
Let $p\geq 5$ be any prime, and deg(c) $=\kappa\geq 5$ be any integer. Then the average value of $n$-periodic point-counting function $M_{c(t)}^{(n)}(\pi, p)$ is equal to $1$ or $2$ or $0$ as $\kappa\to\infty$. More precisely, we have 
\begin{myitemize}
    \item[\textnormal{(a)}] \textnormal{Avg} $M_{c(t)\pm1 = \pi t}^{(n)}(\pi, p) := \lim\limits_{\kappa\to\infty} \Large{\frac{\sum\limits_{5\leq p\leq \kappa, \ \pi\mid (c(t)\pm1)\textnormal{ in }\mathbb{F}_{p}[t]}M_{c(t)}^{(n)}(\pi, p)}{\Large{\sum\limits_{5\leq p\leq \kappa, \ \pi\mid (c(t)\pm1) \textnormal{ in }\mathbb{F}_{p}[t]}1}}} = 1.$ 

    \item[\textnormal{(b)}] \textnormal{Avg} $M_{c(t)= \pi t}^{(n)}(\pi, p) := \lim\limits_{\kappa\to\infty} \Large{\frac{\sum\limits_{5\leq p\leq \kappa, \ \pi\mid c(t) \textnormal{ in }\mathbb{F}_{p}[t]}M_{c(t)}^{(n)}(\pi, p)}{\Large{\sum\limits_{5\leq p\leq \kappa, \ \pi\mid c(t)\textnormal{ in }\mathbb{F}_{p}[t]}1}}} = 2.$
    
    \item[\textnormal{(c)}] \textnormal{Avg} $M_{c\not \equiv\pm1, 0 \ (\textnormal{mod }\pi)}^{(n)}(\pi, p):= \lim\limits_{\kappa \to\infty} \Large{\frac{\sum\limits_{5\leq p\leq \kappa, \ c\not \equiv\pm1, 0 \ (\textnormal{mod }\pi)}M_{c(t)}^{(n)}(\pi, p)}{\Large{\sum\limits_{5\leq p\leq \kappa, \ c\not \equiv\pm1, 0 \ (\textnormal{mod }\pi)}1}}} =  0$.    
\end{myitemize}

\end{cor}
\begin{proof}
Applying a similar argument as in the Proof of Corollary \ref{co6}, we then obtain the limits, as desired.
\end{proof} 
\begin{rem} \label{re8.4}
As before, we also note that from arithmetic statistics to arithmetic dynamics, Corollary \ref{cor6.1} shows that any $\varphi_{(p-1)^{\ell},c}$ iterated on the space $\mathbb{F}_{p}[t] / (\pi)$ has on average one or two or no $n$-periodic orbits as \text{deg}$(c)\to \infty$; a somewhat interesting averaging phenomenon coinciding precisely with an averaging phenomenon remarked in [\cite{BK3}, Remark 8.4] on the average number of distinct fixed orbits of any $\varphi_{p^{\ell},c}$ iterated on $\mathbb{F}_{p}[t] / (\pi)$.
\end{rem}

\section{On Density of $\varphi_{p^{\ell},c}(x)\in \mathbb{Z}[x]$ with $X_{c}^{(n)}(p) = p$ \& $\varphi_{(p-1)^{\ell},c}(x)\in \mathbb{Z}[x]$ with $Y_{c}^{(n)}(p) = 1$ or $2$}\label{sec8}
As in [\cite{BK3}, Section 9] we in this and the next section, wish to restrict our counting on a subring $\mathbb{Z}\subset \mathbb{Z}_{p}$ and then determine: \say{\textit{For any prime $p\geq 3$ and for any fixed $\ell \in \mathbb{Z}_{\geq 1}$ and period $n\in \mathbb{Z}_{\geq 2}$, what is the density of monic integer polynomials $\varphi_{p^{\ell},c}(x) = x^{p^{\ell}} + c\in \mathbb{Z}_{p}[x]$ with exactly $p$ distinct $n$-periodic integral points modulo $p$?}} The following corollary shows that for any fixed $\ell \in \mathbb{Z}_{\geq 1}$ and period $n\in \mathbb{Z}_{\geq 2}$, there are very few monic $p$-adic integer polynomials $\varphi_{p^{\ell},c}(x)\in \mathbb{Z}[x]\subset \mathbb{Z}_{p}[x]$ with exactly $p$ distinct $n$-periodic integral points modulo $p$:
\begin{cor}\label{8.1}
Let $p\geq 3$ be any prime, and $\ell \geq 1$ be any fixed integer. Then the density of integer polynomials $\varphi_{p^{\ell},c}(x) = x^{p^{\ell}} + c\in \mathbb{Z}_{p}[x]$ with $X_{c}^{(n)}(p) = p$ exists and is equal to $0 \%$ as $c\to \infty$. More precisely, we have 
\begin{center}
    $\lim\limits_{c\to\infty} \Large{\frac{\# \{\varphi_{p^{\ell},c}(x)\in \mathbb{Z}[x] \ : \ 3\leq p\leq c \ \textnormal{and} \ X_{c}^{(n)}(p) \ = \ p\}}{\Large{\# \{\varphi_{p^{\ell},c}(x) \in \mathbb{Z}[x] \ : \ 3\leq p\leq c \}}}} = \ 0.$
\end{center}
\end{cor}

\begin{proof}
Since the defining condition $X_{c}^{(n)}(p) = p$ is as we proved in Theorem \ref{2.3} determined whenever $c$ is divisible by $p$, we may then count the number $\# \{\varphi_{p^{\ell},c}(x) \in \mathbb{Z}[x] : 3\leq p\leq c \ \text{and} \ X_{c}^{(n)}(p) \ = \ p\}$ by counting the number $\# \{\varphi_{p^{\ell},c}(x)\in \mathbb{Z}[x] : 3\leq p\leq c \ \text{and} \ p\mid c \ \text{for \ any \ fixed} \ c \}$. In that case, we then write the quotient
\begin{center}
$\Large{\frac{\# \{\varphi_{p^{\ell},c}(x) \in \mathbb{Z}[x] \ : \ 3\leq p\leq c \ \text{and} \ X_{c}^{(n)}(p) \ = \ p\}}{\Large{\# \{\varphi_{p^{\ell},c}(x) \in \mathbb{Z}[x] \ : \ 3\leq p\leq c \}}}} = \Large{\frac{\# \{\varphi_{p^{\ell},c}(x)\in \mathbb{Z}[x] \ : \ 3\leq p\leq c \ \text{and} \ p\mid c \ \text{for any fixed} \ c \}}{\Large{\# \{\varphi_{p^{\ell},c}(x) \in \mathbb{Z}[x] \ : \ 3\leq p\leq c \}}}}$. 
\end{center}\indent Moreover, for any fixed integer $c\in \mathbb{Z}_{\geq 3}$, we may rewrite the numerator of the foregoing quotient to then obtain
\begin{center}
$\# \{\varphi_{p^{\ell},c}(x) \in \mathbb{Z}[x] : 3\leq p\leq c \ \text{and} \ p\mid c \} = \# \{p : 3\leq p\leq c \text{ and } p\mid c \} = \sum_{3\leq p\leq c, \ p\mid c}1 = \omega (c)$, 
\end{center}where $\omega(\ell)$ is by definition the number of distinct prime factors of $\ell$. Writing $\# \{\varphi_{p^{\ell},c}(x) \in \mathbb{Z}[x]  : 3\leq p\leq c \} = \sum_{3\leq p\leq c} 1 = \pi(c)$, where $\pi(\ell)$ is by definition the number of primes at most $\ell$, we then note that the quotient 
\begin{center}
$\Large{\frac{\# \{\varphi_{p^{\ell},c}(x)\in \mathbb{Z}[x] \ : \ 3\leq p\leq c \ \text{and} \ p\mid c \ \text{for any fixed} \ c \}}{\Large{\# \{\varphi_{p^{\ell},c}(x)\in \mathbb{Z}[x] \ : \ 3\leq p\leq c \}}}} = \frac{\omega(c)}{\pi(c)}$.
\end{center}But now applying a very similar argument as in [\cite{BK222}, Proof of Corollary 5.1], we then immediately obtain that the limit exits and moreover the limit is equal to $0$ as desired. This then completes the whole proof, required. 
\end{proof}

\noindent Note that one may interpret Corollary \ref{8.1} as saying that for any fixed $\ell \in \mathbb{Z}_{\geq 1}$ and fixed period $n\in \mathbb{Z}_{\geq 2}$, the probability of choosing randomly a monic $p$-adic integer polynomial $\varphi_{p^{\ell},c}(x) \in \mathbb{Z}[x]\subset \mathbb{Z}_{p}[x]$ having exactly $p$ distinct $n$-periodic $p$-adic integral points modulo $p$ is equal to zero; a somewhat interesting probabilistic phenomenon coinciding with a phenomenon remarked in [\cite{BK3}, Section 9] on the probability of choosing randomly a monic $p$-adic integer polynomial $\varphi_{p^{\ell},c}(x)\in \mathbb{Z}[x]$ with exactly $p$ distinct fixed $p$-adic integral points modulo $p$.

Similarly, we also wish to determine: \say{\textit{For any prime $p\geq 5$ and for any fixed $\ell \in \mathbb{Z}_{\geq 1}$ and period $n\in \mathbb{Z}_{\geq 2}$, what is the density of integer polynomials $\varphi_{(p-1)^{\ell},c}(x) = x^{(p-1)^{\ell}} + c\in \mathbb{Z}_{p}[x]$ with two distinct $n$-periodic integral points modulo $p$?}} The following corollary shows that for any fixed $\ell \in \mathbb{Z}_{\geq 1}$ and period $n\in \mathbb{Z}_{\geq 2}$, there are also very few $p$-adic integer polynomials $\varphi_{(p-1)^{\ell},c}(x)\in \mathbb{Z}[x]$ having two distinct $n$-periodic integral points modulo $p$:

\begin{cor}\label{8.2}
Let $p\geq 5$ be any prime, and $\ell \geq 1$ be any fixed integer. The density of integer polynomials $\varphi_{(p-1)^{\ell},c}(x) = x^{(p-1)^{\ell}} + c\in \mathbb{Z}_{p}[x]$ with $Y_{c}^{(n)}(p) = 2$ exists and is equal to $0 \%$ as $c\to \infty$. Specifically, we have 
\begin{center}
    $\lim\limits_{c\to\infty} \Large{\frac{\# \{\varphi_{(p-1)^{\ell},c}(x) \in \mathbb{Z}[x]\ : \ 5\leq p\leq c \ \textnormal{and} \ Y_{c}^{(n)}(p) \ = \ 2\}}{\Large{\# \{\varphi_{(p-1)^{\ell},c}(x) \in \mathbb{Z}[x]\ : \ 5\leq p\leq c \}}}} = \ 0.$
\end{center}
\end{cor}
\begin{proof}
Since the defining condition $Y_{c}^{(n)}(p) = 2$ is as we proved in Theorem \ref{3.3} determined whenever $c$ is divisible by $p$, we may then count the number $\# \{\varphi_{(p-1)^{\ell},c}(x) \in \mathbb{Z}[x] : 5\leq p\leq c \ \text{and} \ Y_{c}^{(n)}(p) \ = \ 2\}$ by simply counting the number $\# \{\varphi_{(p-1)^{\ell},c}(x)\in \mathbb{Z}[x] : 5\leq p\leq c \ \text{and} \ p\mid c \ \text{for \ any \ fixed} \ c \}$. But now applying a similar argument as in the Proof of Corollary \ref{8.1}, we then obtain that the limit exits and is equal to $0$, as desired.
\end{proof}

\noindent As before, we may also interpret Corollary \ref{8.2} as saying that for any fixed $\ell \in \mathbb{Z}_{\geq 1}$ and fixed period $n\in \mathbb{Z}_{\geq 2}$, the probability of choosing randomly a monic $p$-adic integer polynomial $\varphi_{(p-1)^{\ell},c}(x) \in \mathbb{Z}[x]\subset \mathbb{Z}_{p}[x]$ having two distinct $n$-periodic $p$-adic integral points modulo $p$ is zero; a somewhat interesting probabilistic phenomenon coinciding with a phenomenon remarked in [\cite{BK3}, Section 10] on the probability of choosing randomly a monic $p$-adic integer polynomial $\varphi_{(p-1)^{\ell},c}(x)\in \mathbb{Z}[x]$ having exactly two distinct fixed $p$-adic integral points modulo $p$.

The following corollary shows that for any fixed $\ell \in \mathbb{Z}_{\geq 1}$ and period $n\in \mathbb{Z}_{\geq 2}$, the probability of choosing randomly a monic $p$-adic integer polynomial $\varphi_{(p-1)^{\ell},c}(x)\in \mathbb{Z}[x]$ with one $n$-periodic point modulo $p$ is also zero.

\begin{cor}\label{8.3}
Let $p\geq 5$ be any prime, and $\ell \geq 1$ be any fixed integer. The density of integer polynomials $\varphi_{(p-1)^{\ell},c}(x) = x^{(p-1)^{\ell}} + c\in \mathbb{Z}_{p}[x]$ with $Y_{c}^{(n)}(p) = 1$ exists and is equal to $0 \%$ as $c\to \infty$. That is, we have 
\begin{center}
    $\lim\limits_{c\to\infty} \Large{\frac{\# \{\varphi_{(p-1)^{\ell},c}(x) \in \mathbb{Z}[x]\ : \ 5\leq p\leq c \ \textnormal{and} \ Y_{c}^{(n)}(p) \ = \ 1\}}{\Large{\# \{\varphi_{(p-1)^{\ell},c}(x) \in \mathbb{Z}[x]\ : \ 5\leq p\leq c \}}}} = \ 0.$
\end{center}
\end{cor}
\begin{proof}
As before, since the defining condition $Y_{c}^{(n)}(p) = 1$ is as we proved in Theorem \ref{3.3} determined whenever $c$ is such that $c\pm1$ is divisible by $p$, hence, we may count $\# \{\varphi_{(p-1)^{\ell},c}(x) \in \mathbb{Z}[x] : 5\leq p\leq c \ \text{and} \ Y_{c}^{(n)}(p) \ = \ 1\}$ by counting $\# \{\varphi_{(p-1)^{\ell},c}(x)\in \mathbb{Z}[x] : 5\leq p\leq c \ \text{and} \ p\mid (c\pm1) \ \text{for \ any \ fixed} \ c \}$. But then applying a similar argument as in [\cite{BK111}, Proof of Corollary 6.2], we then obtain that the limit exists and is equal to $0$, as desired.
\end{proof}

\section{On the Density of $\varphi_{p^{\ell},c}(x)\in \mathbb{Z}[x]$ with $X_{c}^{(n)}(p) = 0$ \& $\varphi_{(p-1)^{\ell},c}(x)\in \mathbb{Z}[x]$ with $Y_{c}^{(n)}(p) = 0$}\label{sec9}

\noindent Recall in Corollary \ref{8.1} that a density of $0\%$ of monic $p$-adic integer polynomials $\varphi_{p^{\ell},c}(x)\in\mathbb{Z}[x]\subset \mathbb{Z}_{p}[x]$ have $X_{c}^{(n)}(p) = p$; and so the density of monic integer polynomials $\varphi_{p^{\ell},c}^{n}(x)-x\in \mathbb{Z}_{p}[x]$ that are reducible modulo $p$ is $0\%$. So now, we also wish to determine: \say{\textit{For any prime $p\geq 3$ and for any fixed $\ell \in \mathbb{Z}_{\geq 1}$ and $n\in \mathbb{Z}_{\geq 2}$, what is the density of integer polynomials $\varphi_{p^{\ell},c}(x)\in \mathbb{Z}_{p}[x]$ with no $n$-periodic integral points modulo $p$?}} The following corollary shows that for any fixed $\ell \in \mathbb{Z}_{\geq 1}$ and $n\in \mathbb{Z}_{\geq 2}$, the probability of choosing randomly a monic $p$-adic integer polynomial $\varphi_{p^{\ell},c}(x) \in \mathbb{Z}[x]\subset \mathbb{Z}_{p}[x]$ so that $\mathbb{Q}[x]\slash (\varphi_{p^{\ell}, c}^n(x)-x)$ is a number field of odd degree $p^{n\ell}$ is $1$: 
\begin{cor}\label{9.1}
Let $p\geq 3$ be any prime, and $\ell \geq 1$ be any fixed integer. The density of integer polynomials $\varphi_{p^{\ell},c}(x) = x^{p^{\ell}} + c\in \mathbb{Z}_{p}[x]$ with $X_{c}^{(n)}(p) = 0$ exists and is equal to $100 \%$ as $c\to \infty$. More precisely, we have 
\begin{center}
    $\lim\limits_{c\to\infty} \Large{\frac{\# \{\varphi_{p^{\ell},c}(x)\in \mathbb{Z}[x] \ : \ 3\leq p\leq c \ \textnormal{and} \ X_{c}^{(n)}(p) \ = \ 0 \}}{\Large{\# \{\varphi_{p^{\ell},c}(x) \in \mathbb{Z}[x] \ : \ 3\leq p\leq c \}}}} = \ 1.$
\end{center}
\end{cor}
\begin{proof}
Since the number $X_{c}^{(n)}(p) = p$  or $0$ for any prime $p\geq 3$ and any $\ell \in \mathbb{Z}_{\geq 1}$, and since we also proved the density in Cor. \ref{8.1}, we then obtain the desired density (i.e., we get that the limit exists and is equal to 1). 
\end{proof}

\noindent Note that the foregoing corollary also shows that for any fixed $\ell \in \mathbb{Z}_{\geq 1}$ and $n\in \mathbb{Z}_{\geq 2}$, there are infinitely many polynomials $\varphi_{p^{\ell},c}(x)\in \mathbb{Z}[x]\subset \mathbb{Q}[x]$ such that for $f(x) = \varphi_{p^{\ell},c}^n(x)-x$, the induced quotient $K_{f} = \mathbb{Q}[x]\slash (f(x))$ is a number field of degree $\kappa=p^{n\ell}$. Comparing the densities in Cor. \ref{8.1} and \ref{9.1}, one may then observe that in the whole family of polynomials $\varphi_{p^{\ell},c}(x)\in \mathbb{Z}[x]$, almost all such monics have no integral $n$-periodic points modulo $p$ (i.e., have no rational roots); and so almost all such monic polynomials $f$ are irreducible over $\mathbb{Q}$. This may also imply that the average value of $X_{c}^{(n)}(p)$ in the whole family of polynomials $\varphi_{p^{\ell},c}(x)\in \mathbb{Z}[x]$ is zero.

Similarly, we may also recall in Corollary \ref{8.2} or \ref{8.3} that a density of $0\%$ of monic $p$-adic integer polynomials $\varphi_{(p-1)^{\ell},c}(x)\in \mathbb{Z}[x]\subset \mathbb{Z}_{p}[x]$ have $Y_{c}^{(n)}(p) = 2$ or $1$, resp.; and so the density of monic integer polynomials $\varphi_{(p-1)^{\ell}, c}^{n}(x)-x\in \mathbb{Z}[x]$ that are reducible modulo $p$ is $0\%$. So now, as before we also wish to determine: \say{\textit{For any prime $p\geq 5$ and for any fixed $\ell \in \mathbb{Z}_{\geq 1}$ and $n\in \mathbb{Z}_{\geq 2}$, what is the density of monic integer polynomials $\varphi_{(p-1)^{\ell},c}(x)\in \mathbb{Z}_{p}[x]$ with no $n$-periodic integral points modulo $p$?}} To that end, we then also have the following corollary showing that for any fixed $\ell \in \mathbb{Z}_{\geq 1}$ and $n\in \mathbb{Z}_{\geq 2}$, the probability of choosing randomly $p$-adic integer polynomial $\varphi_{(p-1)^{\ell},c}(x)\in \mathbb{Z}[x]$ such that $\mathbb{Q}[x]\slash (\varphi_{(p-1)^{\ell}, c}^n(x)-x)$ is a number field of  degree $(p-1)^{n\ell}$ is also 1:

\begin{cor} \label{10.2}
Let $p\geq 5$ be any prime, and $\ell \geq 1$ be any fixed integer. The density of integer polynomials $\varphi_{(p-1)^{\ell}, c}(x) = x^{(p-1)^{\ell}}+c\in \mathbb{Z}_{p}[x]$ with $Y_{c}^{(n)}(p) = 0$ exists and is equal to $100 \%$ as $c\to \infty$. Specifically, we have 
\begin{center}
    $\lim\limits_{c\to\infty} \Large{\frac{\# \{\varphi_{(p-1)^{\ell}, c}(x)\in \mathbb{Z}[x] \ : \ 5\leq p\leq c \ \textnormal{and} \ Y_{c}^{(n)}(p) \ = \ 0 \}}{\Large{\# \{\varphi_{(p-1)^{\ell},c}(x) \in \mathbb{Z}[x] \ : \ 5\leq p\leq c \}}}} = \ 1.$
\end{center}
\end{cor}
\begin{proof}
Recall that $Y_{c}^{(n)}(p) = 1, 2$ or $0$ for any given prime $p\geq 5$ and since we also proved the densities in Corollary \ref{8.2} and \ref{8.3}, we now obtain the desired density (i.e., we get that the limit exists and is equal to 1).
\end{proof}

\noindent As before, Corollary \ref{10.2} also shows that for any fixed $\ell \in \mathbb{Z}_{\geq 1}$ and fixed $n\in \mathbb{Z}_{\geq 2}$, there are infinitely many monic polynomials $\varphi_{(p-1)^{\ell},c}(x)\in \mathbb{Z}[x]\subset \mathbb{Q}[x]$ such that for $g(x) = \varphi_{(p-1)^{\ell},c}^n(x)-x$, the induced quotient $L_{g} = \mathbb{Q}[x]\slash (g(x))$ is a degree-$r=(p-1)^{n\ell}$ number field. Again, comparing densities in Cor. \ref{8.2}, \ref{8.3} and \ref{10.2}, it also follows that in the whole family of monics $\varphi_{(p-1)^{\ell},c}(x)\in \mathbb{Z}[x]$, almost all such monics have no $n$-periodic integral points modulo $p$ (i.e., have no $\mathbb{Q}$-roots); and so almost all monic polynomials $g$ are irreducible over $\mathbb{Q}$. But this may also mean that the average value of $Y_{c}^{(n)}(p)$ in the whole family of $\varphi_{(p-1)^{\ell},c}(x)\in \mathbb{Z}[x]$ is also zero.

Recall more generally that any number field $K$ is always naturally equipped with a ring $\mathcal{O}_{K}$ of integers in $K$; and which is classically known to describe the arithmetic of $K$, however, usually difficult to compute in practice. So now, every field $K_{f}=\mathbb{Q}[x]\slash (f(x))$ has a ring of integers $\mathcal{O}_{K_{f}}$ and moreover (as in \cite{BK222}) applying [\cite{sch1}, Theorem 1.2], we then also obtain the following corollary which shows that the probability of choosing randomly $p$-adic integer polynomial $f\in \mathbb{Z}[x]\subset \mathbb{Z}_{p}[x]$ arising from a polynomial discrete dynamical system in Section \ref{sec2} (and ascertained by Corollary \ref{9.1}), such that $\mathbb{Z}[x]\slash (f(x))$ is the ring of integers of $K_{f}$, is $\approx 60.7927\%$:
\begin{cor} \label{9.3}
Assume Corollary \ref{9.1}. When monic integer polynomials $f\in \mathbb{Z}[x]$ are ordered by height $H(f)$ as defined in \textnormal{\cite{sch1}}, the density of such polynomials $f$ such that $\mathbb{Z}[x]\slash (f(x))$ is the ring of integers of $K_{f}$ is $\zeta(2)^{-1}$. 
\end{cor}

\begin{proof}
Since from Corollary \ref{10.1} we know that there are infinitely many irreducible monic integer polynomials $f(x)=\varphi_{p^{\ell},c}^n(x)-x$ such that $K_{f} = \mathbb{Q}[x]\slash (f(x))$ is a number field of degree $\kappa = p^{n\ell}$; and moreover associated to $K_{f}$ is the ring of integers $\mathcal{O}_{K_{f}}$. This then also means that the family of irreducible monic polynomials $f\in \mathbb{Z}[x]$ such that $K_{f}$ is a number field of odd degree $\kappa$ is not empty. Now applying  [\cite{sch1}, Theorem 1.2] on the underlying family of monic polynomials $f\in \mathbb{Z}[x]$ ordered by height $H(f)$ as defined in \cite{sch1} such that $\mathcal{O}_{K_{f}} = \mathbb{Z}[x]\slash (f(x))$, it then follows that the density of such monic polynomials $f\in \mathbb{Z}[x]$ is equal to $\zeta(2)^{-1} \approx 60.7927\%$, as needed.
\end{proof}

Similarly, every number field $L_{g}=\mathbb{Q}[x]\slash (g(x))$ induced by $g$, is naturally equipped with the ring of integers $\mathcal{O}_{L_{g}}$, and which may also be difficult to compute in practice. So now, by again taking great advantage of [\cite{sch1}, Theorem 1.2], we then also obtain the following corollary which shows that the probability of choosing randomly $p$-adic integer polynomial $g\in \mathbb{Z}[x]\subset \mathbb{Z}_{p}[x]$ arising from a polynomial discrete dynamical system in Sect.\ref{sec3} (and ascertained by Corollary \ref{10.2}), such that $\mathbb{Z}[x]\slash (g(x))$ is the ring of integers of $L_{g}$ is also $\approx 60.7927\%$:

\begin{cor}
Assume Corollary \ref{10.2}. When monic integer polynomials $g\in \mathbb{Z}[x]$ are ordered by height $H(g)$ as defined in \textnormal{\cite{sch1}}, the density of such polynomials $g$ such that $\mathbb{Z}[x]\slash (g(x))$ is the ring of integers of $L_{g}$ is $\zeta(2)^{-1}$. 
\end{cor}

\begin{proof}
By applying a similar argument as in Proof of Corollary \ref{9.3}, we then obtain the density, as needed.
\end{proof}

\section{On Local Densities of $f, g\in \mathbb{Z}_{p}[x]$ inducing Maximal orders in Corresponding Fields}

As in [\cite{BK3}, Section 11] we recall that an \say{\textit{order}} in an algebraic number field $K$ is any subring $R\subset K$ that is free of rank $n = [K :\mathbb{Q}]$ over $\mathbb{Z}$. It is a well-known fact that the ring of integers $\mathcal{O}_{K}$ in any number field $K$ is the union of all orders in $K$; and moreover $\mathcal{O}_{K}$ is not only an order in $K$ but is also the maximal order in $K$. But as  mentioned in Section \ref{sec9} that the ring of integers $\mathcal{O}_{K}$ (and so this maximal order in $K$) of any arbitrary number field $K$ is undoubtedly very difficult to compute in practice; and which consequently may then prompt one to work with orders that are possibly smaller and computationally accessible than the maximal order $\mathcal{O}_{K}$. 

So now, recall from Corollary \ref{9.1} the existence of infinitely many monic irreducible polynomials $f\in \mathbb{Z}[x]\subset \mathbb{Z}_{p}[x]$ such that the quotient $K_{p(f)} = \mathbb{Q}_{p}[x]\slash (f(x))$ is a degree-$\kappa=p^{n\ell}$ field extension of $\mathbb{Q}_{p}$ (i.e., $K_{p(f)}\slash \mathbb{Q}_{p}$ is an algebraic $p$-adic number field and so has ring of integers $\mathcal{O}_{K_{p(f)}}$). Meanwhile, recall also that the second part of Theorem \ref{2.3} (i.e., the part in which $X_{c}(p) = 0$ for every $c\not \in p\mathbb{Z}_{p}$) implies $f(x) = \varphi_{p^{\ell},c}^n(x)-x \in \mathbb{Z}_{p}[x]$ is irreducible modulo $p\mathbb{Z}_{p}$; and so to every such irreducible monic polynomial $f\in \mathbb{Z}_{p}[x]$ corresponds a field, say, $K_{p(f)}$. So now inspired (as in \cite{BK3}) by (\cite{gav,as, sch1}), we may also ask for the density of irreducible $p$-adic integer polynomials $f\in \mathbb{Z}_{p}[x]$ arising from a polynomial discrete dynamical system in Section \ref{sec2} (and/or ascertained by Corollary \ref{9.1}), such that $\mathbb{Z}_{p}[x]\slash (f(x))$ is the maximal order in $K_{p(f)}$. In doing so, and then applying (as in \cite{BK3}) $p$-adic density result due to Hendrik Lenstra \cite{as} on irreducible monic $p$-adic integer polynomials $f\in \mathbb{Z}_{p}[x]$, we then obtain here the following corollary showing the probability of choosing randomly $p$-adic integer polynomial $f\in \mathbb{Z}_{p}[x]$ such that $\mathbb{Z}_{p}[x]\slash (f(x))$ is the maximal order in $K_{p(f)}$; moreover this probability tends to $1$ as $p\to \infty$: 
\begin{cor}\label{10.6}
Assume Corollary \ref{9.1} or second part of Theorem \ref{2.3}. Then the density of monic $p$-adic integer polynomials $f$ over $\mathbb{Z}_{p}$ ordered by height $H(f)$ as defined in \textnormal{\cite{as}} such that $\mathbb{Z}_{p(f)} = \mathbb{Z}_{p}[x]\slash (f(x))$ is the maximal order in $K_{p(f)}$ exists and is equal to $\rho_{\textnormal{deg}(f)}(p):= 1 - p^{-2}$. Moreover, this density tends to $1$ as $p\to \infty$.   
\end{cor}
\begin{proof}
To see the density, we recall from Corollary \ref{9.1} the existence of infinitely many $p$-adic integer polynomials $f\in \mathbb{Z}[x]\subset \mathbb{Z}_{p}[x]\subset \mathbb{Q}_{p}[x]$ such that $K_{p(f)}\slash \mathbb{Q}_{p}$ is a number field of degree $\kappa=p^{n\ell}$, or recall that the second part of Theorem \ref{2.3} implies that the polynomial $f(x) = \varphi_{p^{\ell},c}^n(x)-x \in \mathbb{Z}_{p}[x]\subset \mathbb{Q}_{p}[x]$ is irreducible modulo any fixed $p\mathbb{Z}_{p}$ for every coefficient $c\not \in \mathbb{Z}_{p}$, and so induces a degree-$\kappa$ number field $K_{p(f)}\slash \mathbb{Q}_{p}$. This then means that the set of fields $K_{p(f)}$ is not empty. So now, as pointed out in the work of Bhargava-Shankar-Wang [\cite{sch1}, Page 2], we may then apply [\cite{as}, Prop. 3.5] on the family of irreducible polynomials $f\in \mathbb{Z}_{p}[x]$ resulting from Corollary \ref{9.1} or from the second part of Theorem \ref{2.3} when we've ordered them by height $H(f)$ as defined in \cite{as}, to then obtain the first part. Note that letting $p\to \infty$, we then obtain that $\rho_{\text{deg($f$)}}(p)\to 1$ as indeed needed. 
\end{proof}

Similarly, recall also that from Corollary \ref{10.2} that there are infinitely many monic irreducible polynomials $g\in \mathbb{Z}[x]\subset \mathbb{Z}_{p}[x]$ such that the quotient $L_{p(g)} = \mathbb{Q}_{p}[x]\slash (g(x))$ is a degree-$(p-1)^{n\ell}$ field extension of $\mathbb{Q}_{p}$ (and again $L_{p(g)}\slash \mathbb{Q}_{p}$ is an algebraic $p$-adic number field and so has ring of integers $\mathcal{O}_{L_{p(g)}}$)). Moreover, we may also recall that the second part of Theorem \ref{3.3} (i.e., the part in which $Y_{c}(p) = 0$ for every $c\not \equiv \pm1, 0\ (\text{mod} \ p\mathbb{Z}_{p})$) implies $g(x) = \varphi_{(p-1)^{\ell},c}^n(x)-x \in \mathbb{Z}_{p}[x]$ is irreducible modulo $p\mathbb{Z}_{p}$; and hence to every such irreducible monic polynomial $g\in \mathbb{Z}_{p}[x]$ also corresponds a field, say, $L_{p(g)}$. So now, as before we may also wish to determine the density of irreducible polynomials $g\in \mathbb{Z}_{p}[x]$ arising from a polynomial discrete dynamical system in Section \ref{sec3} (and/or ascertained by Corollary \ref{10.2}), such that $\mathbb{Z}_{p}[x]\slash (g(x))$ is the maximal order in $L_{p(g)}$. With that in mind, we note that by again taking great advantage of that same $p$-adic density result due to Lenstra \cite{as}, we then also obtain the following corollary showing the probability of choosing randomly a polynomial $g\in \mathbb{Z}_{p}[x]$ such that $\mathbb{Z}_{p}[x]\slash (g(x))$ is the maximal order in $L_{p(g)}$; and moreover this probability also tends to $1$ as $p\to \infty$: 

\begin{cor}
Assume Corollary \ref{10.2} or second part of Theorem \ref{3.3}. Then the density of monic $p$-adic integer polynomials $g$ over $\mathbb{Z}_{p}$ ordered by height $H(g)$ as defined in \textnormal{\cite{as}} such that $\mathbb{Z}_{p(g)} = \mathbb{Z}_{p}[x]\slash (g(x))$ is the maximal order in $L_{p(g)}$ exists, and equal to $\rho_{\textnormal{deg}(g)}(p):= 1 - p^{-2}$. Moreover, this density tends to $1$ as $p\to \infty$.   
\end{cor}
\begin{proof}
By applying a similar argument as in Proof of Corollary \ref{10.6}, we then obtain the density, as needed. 
\end{proof}

\section{On the Artin-Mazur Zeta Functions induced by Polynomial Maps $\overline{\varphi_{p^{\ell},c}}$ and $\overline{\varphi_{(p-1)^{\ell},c}}$}

As in [\cite{BK222}, Section 9] recall from Artin-Mazur [\cite{AM}, Page 84] that for any given topological space $X$ and for any given $f:X\to X$, we can attach to $f$ a zeta function $\zeta_{f}$ defined in the following way, that for any $s\in \mathbb{C}$, we set  
\begin{equation}\label{eqAM}
    \zeta_{f}(s)=\textnormal{exp}\biggl(\sum_{n=1}^{\infty} \frac{N_{n}(f)\cdot s^n}{n}\biggr)
\end{equation} where $N_{n}(f)$ is the number of isolated periodic points of $f$, of period $n$. So now, inspired (as in [\cite{BK222}, Sect.9]) by the definition and work of Artin-Mazur \cite{AM} on their zeta function $\zeta_{f}(s)$, we in this section also wish to define (in a similar way) and then determine the Artin-Mazur zeta functions arising from a polynomial discrete dynamical system in Section \ref{sec2} and \ref{sec3}. To that end, we note that since $\mathbb{Z}_{p}\slash p\mathbb{Z}_{p}$ is a finite set of $p$ elements and so $\mathbb{Z}_{p}\slash p\mathbb{Z}_{p}$ can be equipped with a topology, then denoting the reduced polynomial map $\varphi_{p^{\ell},c}$ modulo $p\mathbb{Z}_{p}$ by $\overline{\varphi_{p^{\ell},c}}$ and then also replacing  $N_{n}(f)$ with $X_{c}^{(n)}(p)$ in Equation (\ref{eqAM}), we then let $\zeta_{\overline{\varphi_{p^{\ell},c}}}$ be the Artin-Mazur zeta function induced by the map $\overline{\varphi_{p^{\ell},c}}: \mathbb{Z}_{p}\slash p\mathbb{Z}_{p} \to \mathbb{Z}_{p}\slash p\mathbb{Z}_{p}$. Inspired by \cite{AM} on their zeta function $\zeta_{f}$, we in our case have the following corollary showing that the Artin-Mazur zeta functions $\zeta_{\overline{\varphi_{p^{\ell},c}}}$ associated with infinitely many polynomial maps $\overline{\varphi_{p^{\ell},c}}$ are constant functions on the whole complex plane $\mathbb{C}$ and hence algebraic functions of $s$:  

\begin{cor}\label{cAM}
Assume Corollary \ref{9.1} or second part of Theorem \ref{2.3}, and denote the polynomial map $\varphi_{p^{\ell},c}$ modulo prime $p\mathbb{Z}_{p}$ by $\overline{\varphi_{p^{\ell},c}}$. Then the Artin-Mazur zeta function $\zeta_{\overline{\varphi_{p^{\ell},c}}}(s)=1$ for every complex number $s\in \mathbb{C}$.
\end{cor}

\begin{proof}
To see this, we note that from Corollary \ref{9.1} the existence of an infinite family of monic polynomials $\varphi_{p^{\ell},c}(x)\in (\mathbb{Z}\slash p\mathbb{Z})[x]$ having $X_{c}^{(n)}(p)=0$, for every fixed period $n\in \mathbb{Z}_{\geq 2}$; or note that from the second part of Theorem \ref{2.3} that the monic polynomial $\varphi_{p^{\ell},c}(x)\in (\mathbb{Z}_{p}\slash p\mathbb{Z}_{p})[x]$ has $X_{c}^{(n)}(p)=0$ for every coefficient $c\not \equiv 0\ (\text{mod} \ p\mathbb{Z}_{p})$ and fixed period $n\in \mathbb{Z}_{\geq 2}$. But now since $X_{c}^{(n)}(p)=0$ for every fixed $n\in \mathbb{Z}_{\geq 2}$, we then note  
\begin{equation}
    \zeta_{\overline{\varphi_{p^{\ell},c}}}(s)=\textnormal{exp}\biggl(\sum_{n=2}^{\infty} \frac{X_{c}^{(n)}(p)\cdot s^n}{n}\biggr)=\textnormal{exp}(0)=1,
\end{equation} and thus $\zeta_{\overline{\varphi_{p^{\ell},c}}}(s)=1$ for every complex number $s\in \mathbb{C}$. This then completes the whole proof, as desired.
\end{proof}

Similarly, denoting the reduced polynomial map $\varphi_{(p-1)^{\ell},c}$ modulo $p\mathbb{Z}_{p}$ by $\overline{\varphi_{(p-1)^{\ell},c}}$ and then also replacing $N_{n}(f)$ with $Y_{c}^{(n)}(p)$ in Equation (\ref{eqAM}), we then also let $\zeta_{\overline{\varphi_{(p-1)^{\ell},c}}}$ be the Artin-Mazur zeta function induced by the polynomial map $\overline{\varphi_{(p-1)^{\ell},c}}: \mathbb{Z}_{p}\slash p\mathbb{Z}_{p} \to \mathbb{Z}_{p}\slash p\mathbb{Z}_{p}$. So now, inspired again by \cite{AM}, we then also have the following corollary showing that the Artin-Mazur zeta functions $\zeta_{\overline{\varphi_{(p-1)^{\ell},c}}}$ associated with infinitely many polynomial maps $\overline{\varphi_{(p-1)^{\ell},c}}$ are also constant functions on the whole complex plane $\mathbb{C}$ and thus also algebraic functions of $s$:

\begin{cor}
Assume Corollary \ref{10.2} or second part of Theorem \ref{3.3}, and denote the polynomial map $\varphi_{(p-1)^{\ell},c}$ modulo prime $p\mathbb{Z}_{p}$ by $\overline{\varphi_{(p-1)^{\ell},c}}$. Then the Artin-Mazur zeta function $\zeta_{\overline{\varphi_{(p-1)^{\ell},c}}}(s)=1$ for every complex $s\in \mathbb{C}$.
\end{cor}

\begin{proof}
To see this, we note that from Corollary \ref{10.2} the existence of an infinite family of monic polynomials $\varphi_{(p-1)^{\ell},c}(x)\in (\mathbb{Z}\slash p\mathbb{Z})[x]$ having $Y_{c}^{(n)}(p)=0$, for every fixed period $n\in \mathbb{Z}_{\geq 2}$; or note that from the second part of Theorem \ref{2.3} that the monic polynomial $\varphi_{(p-1)^{\ell},c}(x)\in (\mathbb{Z}_{p}\slash p\mathbb{Z}_{p})[x]$ has $X_{c}^{(n)}(p)=0$ for every coefficient $c\not \equiv \pm1, 0\ (\text{mod} \ p\mathbb{Z}_{p})$ and fixed period $n\in \mathbb{Z}_{\geq 2}$. Now since $Y_{c}^{(n)}(p)=0$ for every fixed $n\in \mathbb{Z}_{\geq 2}$, we then note  
\begin{equation}
    \zeta_{\overline{\varphi_{(p-1)^{\ell},c}}}(s)=\textnormal{exp}\biggl(\sum_{n=2}^{\infty} \frac{Y_{c}^{(n)}(p)\cdot s^n}{n}\biggr)=\textnormal{exp}(0)=1,
\end{equation} and thus $\zeta_{\overline{\varphi_{(p-1)^{\ell},c}}}(s)=1$ for every complex number $s\in \mathbb{C}$. This then completes the whole proof, as desired.
\end{proof}

As before, we also note that since $\mathbb{F}_{p}[t]\slash (\pi)$ is a finite set of $p^{\text{deg}(\pi)}=p^m$ elements and so $\mathbb{F}_{p}[t]\slash (\pi)$ can be equipped with a topology, then denoting the reduced polynomial map $\varphi_{p^{\ell},c(t)}$ modulo $\pi$ by $\overline{\varphi_{p^{\ell},c(t)}}$ and then also replacing  $N_{n}(f)$ with $N_{c(t)}^{(n)}(\pi, p)$ in Equation (\ref{eqAM}), we then let $\zeta_{\overline{\varphi_{p^{\ell},c(t)}}}$ be the Artin-Mazur zeta function induced by the map $\overline{\varphi_{p^{\ell},c(t)}}: \mathbb{F}_{p}[t]\slash (\pi) \to \mathbb{F}_{p}[t]\slash (\pi)$. So now, as before inspired by \cite{AM} on their zeta function $\zeta_{f}$, we in our case have the following corollary showing that the Artin-Mazur zeta function $\zeta_{\overline{\varphi_{p^{\ell},c(t)}}}$ associated with a polynomial map $\overline{\varphi_{p^{\ell},c(t)}}$ is constant functions on the whole complex plane $\mathbb{C}$ and so an algebraic function of $s$:

\begin{cor}\label{cAM1}
Assume the second part of Theorem \ref{4.3}, and denote the reduced polynomial map $\varphi_{p^{\ell},c(t)}$ modulo prime $\pi$ by $\overline{\varphi_{p^{\ell},c(t)}}$. Then the Artin-Mazur zeta function $\zeta_{\overline{\varphi_{p^{\ell},c(t)}}}(s)=1$ for any complex number $s\in \mathbb{C}$.
\end{cor}

\begin{proof}
Recall from the second part of Theorem \ref{4.3} that $\varphi_{p^{\ell},c}(x)\in (\mathbb{F}_{p}[t]\slash (\pi))[x]$ has $N_{c(t)}^{(n)}(\pi, p)=0$ for every $c\not \equiv 0\ (\text{mod} \ \pi)$ and every fixed $n\in \mathbb{Z}_{\geq 2}$. But now since $N_{c(t)}^{(n)}(\pi, p)=0$ for every fixed $n\in \mathbb{Z}_{\geq 2}$, we then note 
\begin{equation}
    \zeta_{\overline{\varphi_{p^{\ell},c(t)}}}(s)=\textnormal{exp}\biggl(\sum_{n=2}^{\infty} \frac{N_{c(t)}^{(n)}(\pi, p)\cdot s^n}{n}\biggr)=\textnormal{exp}(0)=1,
\end{equation} and thus $\zeta_{\overline{\varphi_{p^{\ell},c(t)}}}(s)=1$ for every complex number $s\in \mathbb{C}$. This then completes the whole proof, as desired.
\end{proof}

Similarly, denoting the reduced polynomial map $\varphi_{(p-1)^{\ell},c}$ modulo $\pi$ by $\overline{\varphi_{(p-1)^{\ell},c(t)}}$ and then also replacing $N_{n}(f)$ with $M_{c(t)}^{(n)}(\pi, p)$ in Equation (\ref{eqAM}), we then also let $\zeta_{\overline{\varphi_{(p-1)^{\ell},c(t)}}}$ be the Artin-Mazur zeta function induced by the polynomial map $\overline{\varphi_{(p-1)^{\ell},c(t)}}: \mathbb{F}_{p}[t]\slash (\pi) \to \mathbb{F}_{p}[t]\slash (\pi)$. So now, inspired again by \cite{AM}, we then also have the following corollary showing that the Artin-Mazur zeta functions $\zeta_{\overline{\varphi_{(p-1)^{\ell},c(t)}}}$ associated with a polynomial map $\overline{\varphi_{(p-1)^{\ell},c(t)}}$ is also constant functions on the whole complex plane $\mathbb{C}$ and hence also an algebraic function of $s$:

\begin{cor}
Assume the second part of Theorem \ref{5.3}, and denote the reduced polynomial map $\varphi_{(p-1)^{\ell},c(t)}$ modulo prime $\pi$ by $\overline{\varphi_{(p-1)^{\ell},c(t)}}$. Then the Artin-Mazur zeta function $\zeta_{\overline{\varphi_{(p-1)^{\ell},c(t)}}}(s)=1$ for any complex $s\in \mathbb{C}$.
\end{cor}

\begin{proof}
Recall from the second part of Theorem \ref{5.3} that $\varphi_{(p-1)^{\ell},c}(x)\in (\mathbb{F}_{p}[t]\slash (\pi))[x]$ has $M_{c(t)}^{(n)}(\pi, p)=0$ for every $c\not \equiv \pm1, 0\ (\text{mod} \ \pi)$ and every fixed $n\in \mathbb{Z}_{\geq 2}$. Now since $M_{c(t)}^{(n)}(\pi, p)=0$ for every $n\in \mathbb{Z}_{\geq 2}$, we then note 
\begin{equation}
    \zeta_{\overline{\varphi_{(p-1)^{\ell},c(t)}}}(s)=\textnormal{exp}\biggl(\sum_{n=2}^{\infty} \frac{M_{c(t)}^{(n)}(\pi, p)\cdot s^n}{n}\biggr)=\textnormal{exp}(0)=1,
\end{equation} and so $\zeta_{\overline{\varphi_{(p-1)^{\ell},c(t)}}}(s)=1$ for every complex number $s\in \mathbb{C}$. This then completes the whole proof, as desired.
\end{proof}

\section{On Number of Monic Polynomials $f\in \mathbb{Z}[x]$ and $g\in \mathbb{Z}[x]$ with Primitive Galois groups}\label{sec13}

Recall from Corollary \ref{9.1} that there is an infinite family of irreducible monic $p$-adic integer polynomials $f(x) = \varphi_{p^{\ell},c}^n(x)-x\in \mathbb{Z}[x]$ such that $K_{f}=\mathbb{Q}[x]\slash (f(x))$ induced by $f$ is a number field of degree $\kappa=p^{n\ell}$. Moreover, to each such irreducible monic integer polynomial $f\in \mathbb{Q}[x]$, let $G_{f}$ be the Galois group of $f$ over $\mathbb{Q}$. 

So now, inspired (as in [\cite{BK222}, Sect.10] by work of Bhargava \cite{gav1} on van der Waerden's Conjecture, we then also wish to determine the number of irreducible monic polynomials  $f\in \mathbb{Z}[x]$ arising from a polynomial discrete dynamical system in Section \ref{sec2}, of bounded height and such that $G_{f}$ is a primitive Galois group not equal to the full symmetric group $S_{\kappa}$. To that end, we (assuming Corollary \ref{9.1}) wish to first adhere to the setup and definition of coefficient height $H(f)$ in \cite{gav1}). That is, for any fixed $\kappa=p^{n\ell}$, we let $E_{\kappa}(H)$ be the number of monic  integer polynomials $f(x) = \varphi_{p^{\ell},c}^n(x)-x$ of degree $\kappa$ with $h(f)\leq H$ and such that $G_{f}\neq S_{\kappa}$. But now by again taking great advantage of [\cite{gav1}, Theorem 1] of Bhargava, we then obtain the following corollary on  $E_{\kappa}(H)$:

\begin{cor}\label{c9.1}
Assume Corollary \ref{9.1}, and let $E_{\kappa}(H)$ be as before. Then we have that $E_{\kappa}(H)=O(H^{\kappa-1})$.
\end{cor}

\begin{proof}
From Corollary \ref{9.1}, there are infinitely many irreducible monic polynomials $f(x)=\varphi_{p^{\ell},c}^n(x)-x\in \mathbb{Z}[x]$ of degree $\kappa=p^{n\ell}$. Now for every polynomial $f\in \mathbb{Z}[x]\subset \mathbb{Q}[x]$ of fixed degree $\kappa=p^{n\ell}$, let $G_{f}$ be the Galois group of $f$ over $\mathbb{Q}$. But now applying [\cite{gav1}, Theorem 1] on the set of monic polynomials $f\in \mathbb{Z}[x]$ of degree $\kappa$ with $h(f)\leq H$ and such that $G_{f}$ is primitive and $G_{f}\neq S_{\kappa}$, we then immediately obtain the count, as required.
\end{proof}

As before, we may also recall from Corollary \ref{10.2} that there is an infinite family of irreducible monic $p$-adic integer polynomials $g(x) = \varphi_{(p-1)^{\ell},c}^n(x)-x$ such that $L_{g}=\mathbb{Q}[x]\slash (g(x))$ induced by $g$ is a number field of degree $r=(p-1)^{n\ell}$. And moreover, to each such irreducible integer polynomial $g\in \mathbb{Q}[x]$, let $G_{g}$ be the Galois group of $g$ over $\mathbb{Q}$. So now, inspired again work of Bhargava \cite{gav1} on van der Waerden's Conjecture, we then also wish to determine the number of irreducible monic polynomials  $g\in \mathbb{Z}[x]$ arising from a polynomial discrete dynamical system in Section \ref{sec3}, of bounded height and such that $G_{g}$ is a primitive Galois group not equal to the full symmetric group $S_{r}$. To that end, we (assuming Corollary \ref{10.2}) again first import the setup and definition of coefficient height $H(g)$ in \cite{gav1}). That is, for any fixed $r=(p-1)^{n\ell}$, we let $E_{r}(H)$ be the number of monic integer degree-$\upsilon$ polynomials $g(x) = \varphi_{(p-1)^{\ell},c}^{n\ell}(x)-x$ with $h(g)\leq H$ and such that $G_{g}\neq S_{r}$. Now as before, by again taking great advantage of Bhargava's theorem [\cite{gav}, Thm. 1], we then also obtain the following corollary:

\begin{cor}
Assume Corollary \ref{10.2}, and let $E_{r}(H)$ be defined as before. Then we have $E_{r}(H)=O(H^{r-1})$.
\end{cor}

\begin{proof}
Applying a similar argument as in the Proof of Corollary \ref{c9.1}, we then obtain the count, as required.
\end{proof}

\section{On the Number of Number fields $K_{f}$ and $L_{g}$ having Bounded Absolute Discriminant}\label{sec14}
Recall from Corollary \ref{9.1} that there is an infinite family of irreducible monic $p$-adic integer polynomials $f(x) = \varphi_{p^{\ell},c}^n(x)-x\in \mathbb{Z}[x]$ for every fixed $\ell\in  \mathbb{Z}_{\geq 1}$ and $n\in \mathbb{Z}_{\geq 2}$, such that the quotient $K_{f}=\mathbb{Q}[x]\slash (f(x))$ induced by $f$ is a number field of degree $\kappa=p^{n\ell}$. Similarly, recall from Corollary \ref{10.2} that one can always find an infinite family of irreducible monic $p$-adic integer polynomials $g(x) = \varphi_{(p-1)^{\ell},c}^n(x)-x\in \mathbb{Z}[x]$ for every fixed $\ell\in  \mathbb{Z}_{\geq 1}$ and $n\in \mathbb{Z}_{\geq 2}$, such that the quotient $L_{g}=\mathbb{Q}[x]\slash (g(x))$ induced by $g$ is a number field of even degree $r=(p-1)^{n\ell}$. Moreover, we may associate to $K_{f}$ (resp., $L_{g}$) an integer Disc$(K_{f})$ (resp., Disc$(L_{g})$) called the discriminant. So now, inspired (as in \cite{BK3,BK222}) by number field-counting advances in arithmetic statistics, we then also wish to count algebraic number fields $K_{f}$ and $L_{g}$ induced by irreducible polynomials $f$ and $g$ arising from polynomial discrete dynamical systems in Section \ref{sec2} and \ref{sec3} (and ascertained by Corollary \ref{9.1} and \ref{10.2}). To do so, we (as in \cite{BK3,BK222}) define and then determine the asymptotic behavior of the counting functions  
\begin{equation}\label{N_{m}}
N_{\kappa}(X) := \# \Bigl\{K_{f}\slash \mathbb{Q} : [K_{f} : \mathbb{Q}] = \kappa \textnormal{ and } |\text{Disc}(K_{f})|\leq X \Bigr\}
\end{equation} 
\begin{equation}\label{M_{r}}
M_{r}(X) := \# \Bigl\{L_{g}\slash \mathbb{Q} : [L_{g} : \mathbb{Q}] = r \textnormal{ and} \ |\text{Disc}(L_{g})|\leq X \Bigr\}
\end{equation} as a positive real number $X\to \infty$. To that end, motivated (as in \cite{BK222}) by work of Lemke Oliver-Thorne \cite{lem} on counting number fields and then applying [\cite{lem}, Theorem 1.2 (1)] on the function $N_{\kappa}(X)$, we then also obtain:

\begin{cor} \label{10.1} Assume Corollary \ref{9.1}, and let $N_{\kappa}(X)$ be the number defined as in \textnormal{(\ref{N_{m}})}. Then we have 
\begin{equation}\label{N_{m}(x)} 
N_{\kappa}(X)\ll_{\kappa}X^{2d - \frac{d(d-1)(d+4)}{6\kappa}}\ll X^{\frac{8\sqrt{\kappa}}{3}}, \text{where d is the least integer for which } \binom{d+2}{2}\geq 2\kappa + 1.
\end{equation}
\end{cor}

\begin{proof}
To see the inequality \textnormal{(\ref{N_{m}(x)})}, we first recall from Corollary \ref{9.1} the existence of infinitely many monic integer polynomials $f\in \mathbb{Q}[x]$ such that $K_{f}\slash \mathbb{Q}$ is an algebraic number field of odd degree $\kappa=p^{n\ell}$. This then also means that the set of degree-$\kappa$ number fields $K_{f}\slash \mathbb{Q}$ is not empty. But now applying [\cite{lem}, Theorem 1.2 (1)] on the number $N_{\kappa}(X)$, we then obtain inequality \textnormal{(\ref{N_{m}(x)})}; and which then completes the whole proof, as needed.
\end{proof}

Motivated again by that same work of Lemke Oliver-Thorne \cite{lem}, we again take great advantage of the first part of [\cite{lem}, Theorem 1.2] by applying it on $M_{r}(X)$. In doing so, we then also obtain the following corollary:

\begin{cor}Assume Corollary \ref{10.2}, and let $M_{r}(X)$ be the number defined as in \textnormal{(\ref{M_{r}})}. Then we have 
\begin{equation}\label{M_{r}(x)}
M_{r}(X)\ll_{r}X^{2d - \frac{d(d-1)(d+4)}{6r}}\ll X^{\frac{8\sqrt{r}}{3}}, \text{where d is the least integer for which } \binom{d+2}{2}\geq 2r + 1.
\end{equation}
\end{cor}

\begin{proof}
Applying a similar argument as in Proof of Corollary \ref{10.1}, we then obtain inequality \textnormal{(\ref{M_{r}(x)})}, as needed.
\end{proof}

We recall that a number field $K$ is  \say{\textit{monogenic}} if there exists an algebraic number $\alpha \in K$ such that the ring of integers $\mathcal{O}_{K}$ is the subring $\mathbb{Z}[\alpha]$ generated by $\alpha$ over $\mathbb{Z}$, i.e., $\mathcal{O}_{K}= \mathbb{Z}[\alpha]$. Now inspired (as in \cite{BK222}), we also wish to count the number of fields $K_{f}$ induced by irreducible monic $p$-adic integer polynomials $f\in \mathbb{Z}[x]$ arising from a polynomial discrete dynamical system in Section \ref{sec2} (and ascertained by Corollary \ref{9.1}), that are monogenic with absolute discriminant $|\Delta(K_{f})| < X$ and with Galois group Gal$(K_{f}\slash \mathbb{Q})$ equal to symmetric group $S_{p^{n\ell}}$. With that in mind, we (as in [\cite{BK222}, Sect.11]) take great advantage of [\cite{sch1}, Cor.1.3] and then  obtain:

\begin{cor}\label{10.3}
Assume Corollary \ref{9.1}. The number of isomorphism classes of algebraic number fields $K_{f}$ of odd degree $\kappa=p^{n\ell}$ and with $|\Delta(K_{f})| < X$ that are monogenic and have associated Galois group $S_{\kappa}$ is $\gg X^{\frac{1}{2} + \frac{1}{\kappa}}$.
\end{cor}

\begin{proof}
To see this, recall from Corollary \ref{9.1} the existence of infinitely many monic $p$-adic integer polynomials $f\in \mathbb{Z}[x]\subset \mathbb{Q}[x]$ such that $K_{f}=\mathbb{Q}[x]\slash (f(x))$ is an algebraic number field of odd degree $\kappa=p^{n\ell}$, for every fixed integers $\ell\geq 1$ and $n\geq 2$. This then means that the set of fields $K_{f}$ is not empty. But now applying [\cite{sch1}, Corollary 1.3] on the underlying fields $K_{f}$ with $|\Delta(K_{f})| < X$ that are monogenic and have associated Galois group $S_{\kappa}$, it follows that the number of isomorphism classes of such algebraic number fields $K_{f}$ is $\gg X^{\frac{1}{2} + \frac{1}{\kappa}}$.
\end{proof}

Similarly, we take great advantage of [\cite{sch1}, Corollary 1.3] to then also count in the following corollary the number of number fields $L_{g}=\mathbb{Q}[x]\slash (g(x))$ induced by irreducible monic $p$-adic integer even degree-$r$ polynomials $g\in \mathbb{Z}[x]$ arising from a polynomial discrete dynamical system in Section \ref{sec3} (and ascertained by Corollary \ref{10.2}), that are monogenic with $|\Delta(L_{g})| < X$ and such that  associated Galois group Gal$(L_{g}\slash \mathbb{Q})$ is the group $S_{(p-1)^{n\ell}}$:

\begin{cor}
Assume Corollary \ref{10.2}. Then the number of isomorphism classes of algebraic number fields $L_{g}$ of degree $r=(p-1)^{n\ell}$ and $|\Delta(L_{g})| < X$ that are monogenic and have associated Galois group $S_{r}$ is $\gg X^{\frac{1}{2} + \frac{1}{r}}$.
\end{cor}

\begin{proof}
By applying a similar argument as in Proof of Corollary \ref{10.3}, we then obtain the count, as required.
\end{proof}

\section{On the Number of Algebraic Number fields $K_{f}$ \& $L_{g}$ with Prescribed Class Number}\label{sec15}

As in [\cite{BK222}, Section 12] recall that for any number field $K$ with ring of integers $\mathcal{O}_{K}$, we then have an abelian group called \say{\textit{ideal class group}} $\textnormal{Cl}(K)$ corresponding to $K$, which provides a way of measuring how far $\mathcal{O}_{K}$ is from being a unique factorization domain. Now even though the order (also called the \say{\textit{class number}} of $K$ (denoted as $h_{K}$)) of $\textnormal{Cl}(K)$ is finite, it is well known in algebraic and analytic number theory and even more so in arithmetic statistics, that computing $\textnormal{Cl}(K)$ in practice let alone determine precisely $h_{K}$, is a hard problem. 

So now, recall from Corollary \ref{9.1} that there is an infinite family of irreducible monic $p$-adic integer polynomials $f(x) = \varphi_{p^{\ell},c}^n(x)-x\in \mathbb{Z}[x]$ for every fixed $\ell\in  \mathbb{Z}_{\geq 1}$ and $n\in \mathbb{Z}_{\geq 2}$, such that $K_{f}=\mathbb{Q}[x]\slash (f(x))$ is a number field of degree $p^{n\ell}$. Moreover, to every such $K_{f}$ we also have $\textnormal{Cl}(K_{f})$ with finite class number $h_{K_{f}}$. Now inspired (as in \cite{BK3, BK33}) by work of Ho-Shankar-Varma \cite{ho} on odd degree number fields with odd class number, we then also wish to count the number of fields $K_{f}$ induced by irreducible polynomials $f\in \mathbb{Z}[x]$ arising from a polynomial discrete dynamical system in Section \ref{sec2} (and ascertained by Corollary \ref{9.1}), with associated Galois group $S_{p^{n\ell}}$ and with prescribed $h_{K_{f}}$. With that in mind, we take great advantage of [\cite{ho}, Theorem 4] and then obtain the following corollary on the existence of infinitely many $S_{p^{n\ell}}$-number fields $K_{f}$ with odd class number:

\begin{cor}
Assume Corollary \ref{9.1}, and let $\kappa=p^{n\ell}$ be any fixed odd integer. Then there exist infinitely many $S_{\kappa}$-algebraic number fields $K_{f}$ of odd degree $\kappa$  having odd class number.  More precisely, we have 
\begin{center}
$\# \Bigl\{ K_{f} : |\Delta(K_{f})| < X \textnormal{ and } 2\nmid \textnormal{Cl}(K_{f})|\Bigr\}\gg X^{\frac{\kappa + 1}{2\kappa -2}}$,
\end{center} where the implied constants depend on degree $\kappa$ and on an arbitrary finite set $S$ of primes given as in \textnormal{\cite{ho}}.
\end{cor}

\begin{proof}
From Corollary \ref{9.1}, it follows that the family of number fields $K_{f}$ of degree $\kappa = p^{n\ell}$ is not empty. But now since $\kappa$ is an odd integer, then the claim follows from [\cite{ho}, Theorem 4(a)] by setting $K_{f}=K$ as needed.
\end{proof}

Similarly, we may also recall from Corollary \ref{10.2} the existence of an infinite family of irreducible monic $p$-adic integer polynomials $g(x)=\varphi_{(p-1)^{\ell},c}^n(x)-x\in \mathbb{Z}[x]$ for every fixed $\ell\in  \mathbb{Z}_{\geq 1}$ and $n\in \mathbb{Z}_{\geq 2}$, such that the quotient $L_{g} = \mathbb{Q}[x]\slash (g(x))$ induced by $g$ is a number field of degree $(p-1)^{n\ell}$. Moreover, to every such $L_{g}$, we also have a class group $\textnormal{Cl}(L_{g})$ with finite class number $h_{L_{g}}$. So now, by taking again great advantage of work on class groups of number fields in arithmetic statistics and in particular the work of Siad \cite{Sia} on $S_{n}$-number fields $K$ of any even degree $n\geq 4$ and signature $(r_{1}, r_{2})$ where $r_{1}$ are the real embeddings of $K$  and $r_{2}$ are the pairs of conjugate complex embeddings of $K$, we then also obtain the following corollary on the number of fields $L_{g}\slash \mathbb{Q}$ induced by irreducible polynomials $g\in \mathbb{Z}[x]$ arising from a polynomial discrete dynamical system in Sect.\ref{sec3} (and ascertained by Corollary \ref{10.2}), with associated Galois group $S_{(p-1)^{n\ell}}$ and having odd class number: 

\begin{cor}
Assume Cor.  \ref{10.2}, and let $r=(p-1)^{n\ell}$ be any even integer. Then there are infinitely many degree-$r$ monogenic number fields $L_{g}$ of any signature and associated Galois group $S_{r}$ having odd class number. 
\end{cor}
\begin{proof}
To see this, we note that by Cor. \ref{10.2}, it follows that the family of number fields $L_{g}$ of degree $r = (p-1)^{n\ell}$ is not empty. Now since $r$ is even, we then note that the claim follows from [\cite{Sia}, Cor. 10], as indeed needed.
\end{proof}

\section{On the Equidistribution of Families of Artin $L$-Functions induced by Fields $K_{f}$ \& $L_{g}$}\label{sec16}

As in [\cite{BK222}, Sect.13] recall that for any degree-$n$ number field $K$ with ring of integers $\mathcal{O}_{K}$, we have a Dedekind zeta function $\zeta_{K}$ associated with $K$; and which for any $s\in \mathbb{C}$ with $\mathfrak{R}(s)>1$, this zeta function $\zeta_{K}$ is defined by  
\begin{equation}\label{Eqn9}
    \zeta_{K}(s) = \sum_{I\subset \mathcal{O}_{K}}\frac{1}{|\mathcal{O}_{K}\slash I|^s}=\prod_{\mathfrak{p}\subset \mathcal{O}_{K}}\frac{1}{1-|\mathcal{O}_{K}\slash \mathfrak{p}|^{-s}}   
\end{equation}where the above sum (resp., the above product) is taken over all the nonzero ideals $I\subset \mathcal{O}_{K}$ (resp., over all the nonzero prime ideals $\mathfrak{p}$), and $|\mathcal{O}_{K}\slash I|$ (resp. $|\mathcal{O}_{K}\slash \mathfrak{p}|$) is the absolute norm of $I$ (resp. the absolute norm of $\mathfrak{p}$). As a generalization of the Riemann zeta function $\zeta_{\mathbb{Q}}(s)$ (whose vanishing on the line  $\mathfrak{R}(s) = \frac{1}{2}$ is intimately related to the distribution of primes $p \in \mathbb{Z}$ (as a consequence of the Riemann Hypothesis)), it is a classical theme in number theory to understand the vanishing of $\zeta_{K}(s)$ especially on the line $\mathfrak{R}(s) = \frac{1}{2}$, since such vanishing of the zeta function $\zeta_{K}(s)$ is also expected of revealing precise information about the distribution of prime ideals $\mathfrak{p}$ in $K$ (as also a consequence of the number field version of the Riemann Hypothesis). Note that from [\cite{Nico}, Page 10] the zeta function $\zeta_{K}(s)$ factors as $\zeta_{K}(s)=\zeta_{\mathbb{Q}}(s)L(s, \rho_{K}$), where $L(s, \rho_{K})$ is the Artin $L$-function corresponding to an Artin representation $\rho_{K}: \text{Gal}(\mathbb{Q})\to \text{Gal}(M\slash \mathbb{Q}) \hookrightarrow S_{n}\to \text{GL}_{n-1}(\mathbb{C})$, and $M$ is the normal closure of $K$. 

So now, for every degree-$\kappa$ number field $K_{f}$ obtained from a polynomial discrete dynamical system in Section \ref{sec2} and ascertained by Corollary \ref{9.1}, we then also have a Dedekind zeta function $\zeta_{K_{f}}$ corresponding to $K_{f}$. Moreover, as noted earlier from work of Shankar-S\"{o}dergren-Templier [\cite{Nico}, Page 2] that this zeta function $\zeta_{K_{f}}(s)$ factors as $\zeta_{K_{f}}(s)=\zeta(s)L(s, \rho_{K_{f}}$), where $\zeta(s)$ is the Riemann zeta function, $L(s, \rho_{K_{f}})$ is the Artin $L$-function, $\rho_{K_{f}}: \text{Gal}(M_{f}\slash \mathbb{Q}) \hookrightarrow S_{\kappa}\to \text{GL}_{\kappa-1}(\mathbb{C})$ is a representation, and $M_{f}$ being the normal closure of $K_{f}$. 

Now inspired (as in \cite{BK222}) by remarkable work of Shankar-S\"{o}dergren-Templier \cite{Nico} on equidistribution of Artin $L$-functions arising from number fields induced by irreducible monic integer polynomials, we in the same spirit as in \cite{Nico} also wish to study the distribution of Artin $L$-functions $L(s, \rho_{K_{f}})$ arising from number fields $K_{f}$ induced by irreducible monic polynomials $f\in \mathbb{Z}[x]$ obtained from a polynomial discrete dynamical system in Section \ref{sec2}. To do so, we (assuming Corollary \ref{9.1}) wish to first adhere to the setup and notation in \cite{Nico}. That is, let $V(\mathbb{Z})^{\text{irr}}$ be the space consisting of irreducible monic integer polynomials  $f(x)=\varphi_{p^{\ell},c}^n(x)-x$ of fixed degree $\kappa=p^{n\ell}$,  and let $V(\mathbb{Z})^{\text{max}}\subset V(\mathbb{Z})^{\text{irr}}$ be a subset consisting of irreducible monic integer polynomials $f$ such that $R_{f}=\mathbb{Z}[x]\slash (f(x))$ is a maximal order in $K_{f}=\mathbb{Q}[x]\slash (f(x))$. Following \cite{Nico}, it also follows here that the additive group $G_{a}(\mathbb{Z})=\mathbb{Z}$ necessarily acts naturally on our space $V(\mathbb{Z})^{\text{irr}}$ via translation, namely, $(b \cdot f)(x):= f(x+b)$ for every element $b\in \mathbb{Z}$ and for every $f\in V(\mathbb{Z})^{\text{irr}}$; and moreover, this action of $G_{a}(\mathbb{Z})=\mathbb{Z}$ by translation also necessarily preserves each of the sets $V(\mathbb{Z})^{\text{irr}}$ and $V(\mathbb{Z})^{\text{max}}$. Now let $\mathfrak{F}_{1}$ be a family consisting of the $\mathbb{Z}$-orbits on $V(\mathbb{Z})^{\text{max}}$. It then follows (from \cite{Nico}) that the family $\mathfrak{F}_{1}$ necessarily parametrizes degree-$\kappa$ monogenized number fields $(K_{f}, \alpha)$ over $\mathbb{Q}$ up to isomorphism. Note that (by [\cite{Nico}, Subsection 2.3]) this same family $\mathfrak{F}_{1}$ parametrizing  degree-$\kappa$ monogenized fields $(K_{f}, \alpha)$ is also treated to be the same family of corresponding $L$-functions $L(s, \rho_{K_{f}})$.

So now, by taking great advantage of a nice theorem of Shankar-S\"{o}dergren-Templier[\cite{Nico}, Theorem 1.1], we also then obtain the following corollary on the family $\mathfrak{F}_{1}$ parametrizing degree-$\kappa$ monogenized fields $(K_{f}, \alpha)$:

\begin{cor}\label{11.1}
Assume Corollary \ref{9.1}, and let $\mathfrak{F}_{1}$ be as before. Then $\mathfrak{F}_{1}$ parametrizing monogenized degree-$\kappa$ fields ordered by height $h(f)$ as defined in \textnormal{\cite{Nico}} satisfies Sato-Tate equidistribution in the sense of \textnormal{[\cite{Sar}, Conj.1]}. 
\end{cor}

\begin{proof}
Since we know from Corollary \ref{9.1} that there are infinitely many irreducible monic integer polynomials $f$ such that $K_{f}$ is a number field of degree $\kappa=p^{n\ell}$, then this also means that the family of degree-$\kappa$ number fields $K_{f}\slash \mathbb{Q}$ is not empty. Now letting $\alpha$ be the image of $x$ in $R_{f}=\mathbb{Z}[x]\slash (f(x))$ and so (by \cite{Nico}) the pair $(K_{f}, \alpha)$ is a degree-$\kappa$ monogenized field, it then follows that the family of monogenized degree-$\kappa$ fields $(K_{f}, \alpha)$ is not empty; which also means that the family $\mathfrak{F}_{1}$ parametrizing  degree-$\kappa$ monogenized fields $(K_{f}, \alpha)$ is not empty. But now applying [\cite{Nico}, Thm. 1.1] to the underlying family $\mathfrak{F}_{1}$ ordered by height $h(f)$ as defined in [\cite{Nico}, Page 3], it then follows that $\mathfrak{F}_{1}$ satisfies Sato-Tate equidistribution in the sense of \textnormal{[\cite{Sar}, Conjecture 1]}, as needed.
\end{proof}

Similarly, for every degree-$r$ field $L_{g}$ obtained from a polynomial discrete dynamical system in Section \ref{sec3} and ascertained by Corollary \ref{10.2}, we also have a Dedekind zeta function $\zeta_{L_{g}}$ corresponding to $L_{g}$. Moreover, it again follows from \cite{Nico} that the Dedekind zeta function $\zeta_{L_{g}}(s)=\zeta(s)L(s, \rho_{L_{g}}$), where $L(s, \rho_{\mathbb{Q}_{g}})$ is the Artin $L$-function,  $\rho_{L_{g}}: \text{Gal}(M_{g}\slash \mathbb{Q}) \hookrightarrow S_{r}\to \text{GL}_{r-1}(\mathbb{C})$ is an Artin representation, and $M_{g}$ the normal closure of $L_{g}$. 

So now, in again the same spirit as in \cite{Nico}, we also wish to study the distribution of Artin $L$-functions $L(s, \rho_{L_{g}})$ arising from fields $\mathbb{Q}_{g}$ induced by irreducible polynomials $g\in \mathbb{Z}[x]$ obtained from a polynomial discrete dynamical system in Section \ref{sec3}. To this end, we (also assuming Corollary \ref{10.2}) adhere again to the setup and notation in \cite{Nico}. That is, we again let $W(\mathbb{Z})^{\text{irr}}$ be the space consisting of irreducible monic integer polynomials  $g(x)=\varphi_{(p-1)^{\ell},c}^n(x)-x$ of fixed degree $r=(p-1)^{n\ell}$,  and let $W(\mathbb{Z})^{\text{max}}\subset W(\mathbb{Z})^{\text{irr}}$ be a subset consisting of irreducible polynomials $g$ such that $R_{g}=\mathbb{Z}[x]\slash (g(x))$ is a maximal order in $L_{g}=\mathbb{Q}[x]\slash (g(x))$. Following again \cite{Nico}, it also follows here that $G_{a}(\mathbb{Z})=\mathbb{Z}$ necessarily acts naturally on $W(\mathbb{Z})^{\text{irr}}$ via translation, namely, $(b \cdot g)(x):= g(x+b)$ for every $b\in \mathbb{Z}$ and for every $g\in W(\mathbb{Z})^{\text{irr}}$; and moreover, this action of $G_{a}(\mathbb{Z})=\mathbb{Z}$ by translation also necessarily preserves each of $W(\mathbb{Z})^{\text{irr}}$ and $W(\mathbb{Z})^{\text{max}}$. Now let $\mathfrak{F}_{2}$ be a family consisting of the $\mathbb{Z}$-orbits on $W(\mathbb{Z})^{\text{max}}$. It then follows (from \cite{Nico}) that the family $\mathfrak{F}_{2}$ necessarily parametrizes degree-$r$ monogenized fields $(L_{g}, \beta)$ up to isomorphism. As before, we also note that (from [\cite{Nico}, Subsect.2.3]) this same family $\mathfrak{F}_{2}$ parametrizing  degree-$r$ monogenized fields $(L_{g}, \beta)$ is also the family of associated $L$-functions $L(s, \rho_{L_{g}})$. By again, taking great advantage of [\cite{Nico}, Theorem 1.1], we then obtain the following corollary on the family $\mathfrak{F}_{2}$:

\begin{cor}
Assume Corollary \ref{10.2}, and let $\mathfrak{F}_{2}$ be as before. Then $\mathfrak{F}_{2}$ parametrizing monogenized degree-$r$ fields  ordered by height $h(g)$ as defined in \textnormal{\cite{Nico}} satisfies Sato-Tate equidistribution in the sense of \textnormal{[\cite{Sar}, Conj.1]}. 
\end{cor}

\begin{proof}
By applying a similar argument as in the Proof of Corollary \ref{11.1}, it then also follows immediately that the family $\mathfrak{F}_{2}$ satisfies  Sato-Tate equidistribution in the sense of \textnormal{[\cite{Sar}, Conjecture 1]}, as also indeed needed.
\end{proof} 

\section{On the Densities of Square-free values of Polynomials $f_{c(t)}\in \mathbb{F}_{p}[t][x]$ and $g_{c(t)}\in \mathbb{F}_{p}[t][x]$}\label{sec17}
Recall that the second part of Theorem \ref{4.3} (i.e., the part in which we proved $N_{c(t)}^{(n)}(\pi, p) = 0$ for every coefficient $c\not \equiv 0\ (\text{mod} \ \pi)$) implies that the monic polynomial $f_{c(t)}(x) = \varphi_{p^{\ell},c}^n(x)-x\in \mathbb{F}_{p}[t][x]$ is irreducible modulo prime $\pi$; and so it then follows (from standard theory about irreducibility in $\mathbb{F}_{p}[t][x]$) that the monic polynomial $f_{c(t)}(x)$ is irreducible in $\mathbb{F}_{p}[t][x]$. Similarly, we also note that the second part of Theorem \ref{5.3} (i.e., the part in which we proved $M_{c(t)}^{(n)}(\pi, p) = 0$ for every coefficient $c\not \equiv \pm1, 0\ (\text{mod} \ \pi)$) also implies that the monic polynomial $g_{c(t)}(x) = \varphi_{(p-1)^{\ell},c}^n(x)-x\in \mathbb{F}_{p}[t][x]$ is irreducible modulo prime $\pi$; and thus also follows that the monic polynomial $g_{c(t)}(x)$ is irreducible in $\mathbb{F}_{p}[t][x]$. So now, motivated greatly by work of Ramsay \cite{Ram}, Poonen \cite{Bjo} and Rudnick \cite{Zeev} on the density of square-free polynomials in $\mathbb{F}_{q}[t]$ over finite field $\mathbb{F}_{q}$ by polynomials in $\mathbb{F}_{q}[t][x]$, we in this section wish to study the density of square-free polynomials in $\mathbb{F}_{p}[t]$ via monic irreducible polynomials $f_{c(t)}\in \mathbb{F}_{p}[t][x]$ (resp. $g_{c(t)}\in \mathbb{F}_{p}[t][x]$) arising from a polynomial discrete dynamical system in Section \ref{sec4}
(resp. in Section \ref{sec5}). To do so, we let $p\geq 3$ or $p\geq 5$ be any (fixed) prime and $m\geq 1$ be any integer, and then define 
\begin{equation}\label{d1}
\mathcal{M}_{m}(\mathbb{F}_{p}) = \Bigl\{a\in  \mathbb{F}_{p}[t] :  a\text{ is monic of } \text{deg} \ a = m  \Bigr\}
\end{equation} and so (by a standard fact in function field number theory about the number of monic polynomials in $\mathbb{F}_{p}[t]$ of any given degree) then $\#\mathcal{M}_{m}(\mathbb{F}_{p}) = p^m$. Assuming the second part of Theorem \ref{4.3}, we then wish to define 
\begin{equation}\label{d2}
\mathcal{S}_{f_{c(t)}}(m)(\mathbb{F}_{p}) = \Bigl\{a\in  \mathcal{M}_{m}(\mathbb{F}_{p}) :  f_{c(t)}(a)\text{ is square-free}  \Bigr\}.
\end{equation}

\indent Similarly, assuming the second part of Theorem \ref{5.3}, we then also wish to define the following set
\begin{equation}\label{d3}
\mathcal{S}_{g_{c(t)}}(m)(\mathbb{F}_{p}) = \Bigl\{a\in  \mathcal{M}_{m}(\mathbb{F}_{p}) :  g_{c(t)}(a)\text{ is square-free}  \Bigr\}.
\end{equation}

By taking great advantage of a theorem (restated again in [\cite{Zeev}, pp. 61-62, Equations (1.4) and (1.5)]) first proved by Ramsay [\cite{Ram}, Theorem 1] and then also proved differently by Poonen [\cite{Bjo}, Theorem 3.4 (with $n=1$)], we (assuming that the irreducible monic polynomial $f_{c(t)}(x)$ has no repeated roots in any extension of $\mathbb{F}_{p}(t)$) then obtain the following corollary about the density of $a\in \mathbb{F}_{p}[t]$ such that $f_{c(t)}(a)$ is square-free in $\mathbb{F}_{p}[t]$:
\begin{cor}\label{co15.1}
Assume second part of Theorem \ref{4.3}, and suppose $f_{c(t)}(x) = \varphi_{p^{\ell},c}^n(x)-x\in \mathbb{F}_{p}[t][x]$ is a separable polynomial. Let $\mathcal{M}_{m}(\mathbb{F}_{p})$ \textnormal{(}resp. $\mathcal{S}_{f_{c(t)}}(m)(\mathbb{F}_{p})$\textnormal{)} be defined as in \textnormal{(\ref{d1}) (}resp. as in \textnormal{(\ref{d2}))}. Then we have
\begin{equation}\label{co17.1}
\frac{\#\mathcal{S}_{f_{c(t)}}(m)(\mathbb{F}_{p})}{\#\mathcal{M}_{m}(\mathbb{F}_{p})}=c_{f_{c(t)}} + O_{f_{c(t)},p \ }\biggl(\frac{1}{m}\biggl), \textnormal{ as } m\to \infty
\end{equation}where
\begin{equation*}
    c_{f_{c(t)}} = \prod_{\textnormal{primes }\pi\in \mathbb{F}_{p}[t]}\biggl(1-\frac{\rho_{f_{c(t)}(\pi^2)}}{|\pi|^2}\biggl) 
\end{equation*}and for any polynomial $D\in \mathbb{F}_{p}[t]$, the number $\rho_{f_{c(t)}}(D) = \# \{C \textnormal{ mod } D\ : \ f_{c(t)}(C)\ = \ 0 \textnormal{ mod } D \}$.
\end{cor}

\begin{proof}
Since we know from earlier discussion in this section that the second part of Theorem \ref{4.3} induces irreducible monic degree-$p^{n\ell}$ polynomials $f_{c(t)}(x)=\varphi_{p^{\ell},c}^n(x)-x\in \mathbb{F}_{p}[t][x]$, we then note that the set of monic polynomials $f_{c(t)}(x)\in \mathbb{F}_{p}[t][x]$ of odd degree $\kappa=p^{n\ell}$ in $x$ is not empty. But now since we know by assumption that every polynomial $f_{c(t)}\in \mathbb{F}_{p}[t][x]$ is separable, then applying [\cite{Ram}, Theorem 1] (also restated again in [\cite{Zeev}, pp. 61-62, Eqns (1.4) and (1.5)]) on $\frac{\#\mathcal{S}_{f_{c(t)}}(m)(\mathbb{F}_{p})}{\#\mathcal{M}_{m}(\mathbb{F}_{p})}$, we then immediately obtain the equality (\ref{co17.1}), as needed.   
\end{proof}

By again taking great advantage of a theorem (restated again in [\cite{Zeev}, pp. 61-62, Eqns (1.4) and (1.5)]) of \cite{Ram, Bjo}, we (assuming that the irreducible monic polynomial $g_{c(t)}(x)$ has no repeated roots in any extension of $\mathbb{F}_{p}(t)$) also obtain the following corollary on the density of $a\in \mathbb{F}_{p}[t]$ such that $g_{c(t)}(a)$ is square-free in $\mathbb{F}_{p}[t]$: 

\begin{cor}\label{co15.2}
Assume second part of Theorem \ref{5.3}, and suppose $g_{c(t)}(x) = \varphi_{(p-1)^{\ell},c}^n(x)-x\in \mathbb{F}_{p}[t][x]$ is a separable polynomial. Let $\mathcal{M}_{m}(\mathbb{F}_{p})$ \textnormal{(}resp. $\mathcal{S}_{g_{c(t)}}(m)(\mathbb{F}_{p})$\textnormal{)} be defined as in \textnormal{(\ref{d1}) (}resp. as in \textnormal{(\ref{d3}))}. Then we have
\begin{equation} \label{co17.2}
\frac{\#\mathcal{S}_{g_{c(t)}}(m)(\mathbb{F}_{p})}{\#\mathcal{M}_{m}(\mathbb{F}_{p})}=c_{g_{c(t)}} + O_{g_{c(t)},p \ }\biggl(\frac{1}{m}\biggl), \textnormal{ as } m\to \infty
\end{equation}where
\begin{equation*}
    c_{g_{c(t)}} = \prod_{\textnormal{primes }\pi\in \mathbb{F}_{p}[t]}\biggl(1-\frac{\rho_{g_{c(t)}(\pi^2)}}{|\pi|^2}\biggl) 
\end{equation*}and for any polynomial $D\in \mathbb{F}_{p}[t]$, the number $\rho_{g_{c(t)}}(D) = \# \{C \textnormal{ mod } D\ : \ g_{c(t)}(C)\ = \ 0 \textnormal{ mod } D \}$.
\end{cor}

\begin{proof}
Applying a similar argument as in Proof of Corollary \ref{co15.1}, we then obtain the equality (\ref{co17.2}), as needed.
\end{proof}

\begin{rem}
We note (from [\cite{Zeev}, Page 62]) that in Corollary \ref{co15.1} (resp. Corollary \ref{co15.2}) the density $c_{f_{c(t)}}$ (resp. $c_{g_{c(t)}}$) is positive if and only if there is $a\in \mathbb{F}_{p}[t]$ such that $f_{c(t)}(a)$ (resp. $g_{c(t)}(a)$) is square-free in $\mathbb{F}_{p}[t]$.
\end{rem}

Recall that for any $h(x)=\sum_{i=0}^{n}a_{i}(t)x^i\in \mathbb{F}_{p}[t][x]$, the content of $h$ (denoted as $c(h)$) is the greatest common divisor of the coefficients of $h$, i.e., $c(h)$ = gcd$(a_{0}(t), a_{1}(t), \cdots, a_{n}(t))\in \mathbb{F}_{p}[t]$. Now motivated by work of Rudnick \cite{Zeev} on square-free values of fixed degree separable polynomials in $\mathbb{F}_{p}[t][x]$ with square-free content, we then have the following corollary showing that for all fixed degree-$\kappa$ separable polynomials $f_{c(t)}\in \mathbb{F}_{p}[t][x]$ arising from a polynomial discrete dynamical system in Section \ref{sec4} and of fixed height Ht$(f_{c(t)})$, the polynomials $f_{c(t)}(a)$ are square-free in $\mathbb{F}_{p}[t]$ for almost all fixed degree-$m$ polynomials $a\in \mathbb{F}_{p}[t]$, in the large finite field limit:
\begin{cor}\label{co17.4}
Assume second part of Theorem \ref{4.3}, and suppose $f_{c(t)}(x)\in \mathbb{F}_{p}[t][x]$ is a separable polynomial of fixed degree $\kappa=p^{n\ell}$ in $x$. Let $\mathcal{M}_{m}(\mathbb{F}_{p})$ \textnormal{(}resp. $\mathcal{S}_{f_{c(t)}}(m)(\mathbb{F}_{p})$\textnormal{)} be defined as in \textnormal{(\ref{d1}) (}resp. as in \textnormal{(\ref{d2}))}. Then 
\begin{equation}\label{eq17.4}
\frac{\#\mathcal{S}_{f_{c(t)}}(m)(\mathbb{F}_{p})}{\#\mathcal{M}_{m}(\mathbb{F}_{p})}= 1 + O\biggl(\frac{m\kappa  + \textnormal{Ht}(f_{c(t)})\kappa}{p}\biggl), \textnormal{ as } p\to \infty
\end{equation}where the implied constant absolute.
\end{cor}

\begin{proof}
Since we know from earlier discussion in this section that the second part of Theorem \ref{4.3} induces irreducible monic degree-$p^{n\ell}$ polynomials $f_{c(t)}(x)=\varphi_{p^{\ell},c}^n(x)-x\in \mathbb{F}_{p}[t][x]$, we then note that the set of monic polynomials $f_{c(t)}(x)\in \mathbb{F}_{p}[t][x]$ of fixed degree $\kappa=p^{n\ell}$ in $x$ is not empty. Now since fixed degree-$\kappa$ polynomial $f_{c(t)}(x)$ is monic in $x$, it then follows that the content of any polynomial $f_{c(t)}(x)$ is equal to $1$ and so the content of $f_{c(t)}$ is square-free. But now since we know by assumption that every monic $f_{c(t)}(x)\in \mathbb{F}_{p}[t][x]$ is separable, then applying [\cite{Zeev}, Theorem 1.1] on the underlying non-empty set of fixed degree-$\kappa$ monic polynomials $f_{c(t)}\in \mathbb{F}_{p}[t][x]$ of height Ht$(f_{c(t)})$ (as in [\cite{Zeev}, Page 62]) and  Ht$(f_{c(t)})$ bounded, we then obtain the equality (\ref{eq17.4}), as needed. 
\end{proof}

Similarly, by again taking great advantage of a theorem Rudnick [\cite{Zeev}, Theorem 1.1], we then also immediately have the following corollary showing that for all fixed degree-$r$ separable polynomials $g_{c(t)}\in \mathbb{F}_{p}[t][x]$ arising from a polynomial discrete dynamical system in Section \ref{sec5} and of fixed height Ht$(g_{c(t)})$, the polynomials $g_{c(t)}(a)$ are square-free in $\mathbb{F}_{p}[t]$ for almost all fixed degree-$m$ polynomials $a\in \mathbb{F}_{p}[t]$, in the large finite field limit: 

\begin{cor}\label{co17.5}
Assume second part of Theorem \ref{5.3}, and suppose $g_{c(t)}(x)\in \mathbb{F}_{p}[t][x]$ is a separable polynomial of fixed degree $r=(p-1)^{n\ell}$ in $x$. Let $\mathcal{M}_{m}(\mathbb{F}_{p})$ \textnormal{(}resp. $\mathcal{S}_{g_{c(t)}}(m)(\mathbb{F}_{p})$\textnormal{)} be as in \textnormal{(\ref{d1}) (}resp. as in \textnormal{(\ref{d3}))}. Then
\begin{equation}\label{eq17.5}
\frac{\#\mathcal{S}_{g_{c(t)}}(m)(\mathbb{F}_{p})}{\#\mathcal{M}_{m}(\mathbb{F}_{p})}= 1 + O\biggl(\frac{mr + \textnormal{Ht}(g_{c(t)})r}{p}\biggl), \textnormal{ as } p\to \infty
\end{equation}where the implied constant absolute.
\end{cor}
\begin{proof}
Applying a similar argument as in Proof of Corollary \ref{co17.4}, we then obtain the equality (\ref{eq17.5}), as needed.
\end{proof}
\begin{rem}
Since it is well-known that primes $\pi \in \mathbb{F}_{p}[t]$ have positive density among all monic polynomials of any given degree in $\mathbb{F}_{p}[t]$, we then also note that Corollary \ref{co17.4} shows in particular that the polynomial $f_{c(t)}(\pi)$ is square-free in $\mathbb{F}_{p}[t]$ for almost all primes $\pi \in \mathbb{F}_{p}[t]$ as the size $\#\mathbb{F}_{p}\to \infty$. Similarly, we note that Corollary \ref{co17.5} shows the polynomial $g_{c(t)}(\pi)$ is square-free in $\mathbb{F}_{p}[t]$ for almost all primes $\pi \in \mathbb{F}_{p}[t]$ as  $\#\mathbb{F}_{p}\to \infty$. 
\end{rem}

\section{On Probability of $f_{c(t)}\in \mathbb{F}_{p}[t][x]$ with $G_{f_{c(t)}}\neq S_{\textnormal{deg}(f)}$ \& $g_{c(t)}\in \mathbb{F}_{p}[t][x]$ with $G_{g_{c(t)}}\neq S_{\textnormal{deg} (g)}$}\label{sec18}

As in Section \ref{sec17}, we again recall that the second part of Theorem \ref{4.3} (i.e., the part in which $N_{c(t)}^{(n)}(\pi, p) = 0$ for every coefficient $c\not \equiv 0\ (\text{mod} \ \pi)$) implies that the monic degree-$p^{n\ell}$ polynomial $f_{c(t)}(x) = \varphi_{p^{\ell},c}^n(x)-x\in \mathbb{F}_{p}[t][x]$ is irreducible modulo prime $\pi$. Hence, we then note (from standard theory about irreducibility in $\mathbb{F}_{p}[t][x]$) that the monic polynomial $f_{c(t)}(x)$ is irreducible in the polynomial ring $\mathbb{F}_{p}[t][x]$. Moreover, to every such $f_{c(t)}\in \mathbb{F}_{p}[t][x]\subset \mathbb{F}_{p}(t)[x]$, we let $G_{f_{c(t)}}$ be the Galois group of $f_{c(t)}$ over any fixed rational function field $\mathbb{F}_{p}(t)$. 

So now, motivated by recent work of Kaltofen \cite{Kal} on van der Waerden conjecture in function field setting and again (as in Section \ref{sec13}) by Bhargava \cite{gav1}, we then in this section also wish to determine the probability of choosing randomly a monic polynomial $f_{c(t)}$ arising from a polynomial discrete dynamical system in Section \ref{sec4}, such that the associated Galois group $G_{f_{c(t)}}$ is not equal to the symmetric group $S_{\textnormal{deg}(f)}$. To do so, we (assuming second part of Theorem \ref{4.3}) wish to first adhere to the setup in \cite{Kal}, however, for the sake of being neat in our discussion, we also invoke a different notation for the set in consideration and also for the probability in question. That is, for any fixed prime $p\geq 3$ and fixed integer $\kappa=p^{n\ell}$, we let $\mathcal{E}_{(\kappa, \leq 1)}$ (resp. $\mu (\mathcal{E}_{(\kappa, \leq 1)}$) be a set consisting of monic polynomials $f_{c(t)}(x)\in \mathbb{F}_{p}[t][x]$ of degree $\kappa$ in $x$ and degree $\leq 1$ in $t$, such that $G_{f_{c(t)}}\neq S_{\kappa}$ (resp. be the probability of $\mathcal{E}_{(\kappa, \leq 1)}$). So now, taking great advantage of a result of Kaltofen \cite{Kal}, we then obtain the following corollary showing that the probability  $\mu(\mathcal{E}_{(\kappa, \leq 1)})$ is positive and moreover bounded below by $p^{-1}$:  

\begin{cor}\label{co18.1}
Let $p\geq 3$ be any fixed prime, and suppose that the second part of Theorem \ref{4.3} holds. Let $\mathcal{E}_{(\kappa, \leq 1)}$ (resp. $\mu_{\kappa} = \mu (\mathcal{E}_{(\kappa, \leq 1)})$ be the set (resp. the probability) defined as before. Then the probability $\mu_{\kappa} \geq p^{-1}$. 
\end{cor}

\begin{proof}
Since we know from earlier discussion in this section that the second part of Theorem \ref{4.3} induces irreducible monic degree-$p^{n\ell}$ polynomials $f_{c(t)}\in \mathbb{F}_{p}[t][x]$, we then note that the set of monic polynomials $f_{c(t)}(x)\in \mathbb{F}_{p}[t][x]$ of degree $p^{n\ell}$ in $x$ is not empty. Now for every such $f_{c(t)}(x)\in \mathbb{F}_{p}[t][x]\subset \mathbb{F}_{p}(t)[x]$ of fixed degree $\kappa=p^{n\ell}$ in $x$ and of degree $\leq 1$ in $t$, let $G_{f_{c(t)}}$ be the Galois group of $f_{c(t)}$ over $\mathbb{F}_{p}(t)$. But then since degree $\kappa>3$, applying a result in [\cite{Kal}, Page 1] on the set $\mathcal{E}_{(\kappa, \leq 1)}$, we then obtain the conclusion, as required.  
\end{proof}

Note that for any fixed prime $p\geq 3$ and fixed integer $\kappa = p^{n\ell}$, assuming the second part of Theorem \ref{4.3} and  letting $\mathcal{E}_{(\kappa, \leq 1)}^{\perp}$ (resp. $\mu (\mathcal{E}_{(\kappa, \leq 1)}^{\perp}$) be a set consisting of monic polynomials $f_{c(t)}(x)\in \mathbb{F}_{p}[t][x]$ of degree $\kappa$ in $x$ and degree $\leq 1$ in $t$, such that the Galois group $G_{f_{c(t)}}= S_{\kappa}$ (resp. be the probability of the set $\mathcal{E}_{(\kappa, \leq 1)}^{\perp}$), we then (by applying work of Kaltofen \cite{Kal}) also obtain the following corollary on the probability of the set $\mathcal{E}_{(\kappa, \leq 1)}^{\perp}$: 

\begin{cor}\label{co18.2}
Let $p\geq 3$ be any fixed prime, and suppose that the second part of Theorem \ref{4.3} holds. Let $\mathcal{E}_{(\kappa, \leq 1)}^{\perp}$ (resp. $\mu_{\kappa}^{\perp} = \mu (\mathcal{E}_{(\kappa, \leq 1)}^{\perp})$ be the set (resp. the probability) defined as before. Then the probability $\mu_{\kappa}^{\perp} = 1-O( p^{-1})$. 
\end{cor}

\begin{proof}
Applying a similar argument as in Proof of Corollary \ref{co18.1} but with $\mathcal{E}_{(\kappa, \leq 1)}$ replaced with $\mathcal{E}_{(\kappa, \leq 1)}^{\perp}$ and then applying a result in [\cite{Kal}, Page 1] on the non-empty set $\mathcal{E}_{(\kappa, \leq 1)}^{\perp}$, we then obtain the conclusion, as required.
\end{proof}

Similarly, recall that the second part of Theorem \ref{5.3} (i.e., the part in which $M_{c(t)}^{(n)}(\pi, p) = 0$ for every $c\not \equiv \pm, 0\ (\text{mod} \ \pi)$) implies that the monic degree-$(p-1)^{n\ell}$ polynomial $g_{c(t)}(x) = \varphi_{(p-1)^{\ell},c}^n(x)-x\in \mathbb{F}_{p}[t][x]$ is irreducible modulo prime $\pi$. Hence, it then also follows (from standard theory about irreducibility in $\mathbb{F}_{p}[t][x]$) that the monic polynomial $g_{c(t)}(x)$ is irreducible in $\mathbb{F}_{p}[t][x]$. Moreover, to every such $g_{c(t)}\in \mathbb{F}_{p}[t][x]\subset \mathbb{F}_{p}(t)[x]$, we also let $G_{g_{c(t)}}$ be the Galois group of $g_{c(t)}$ over any fixed rational function field $\mathbb{F}_{p}(t)$. So now, motivated again by work of Kaltofen \cite{Kal} and (as in Section \ref{sec13}) by Bhargava \cite{gav1}, we then also wish to determine the probability of choosing randomly a monic polynomial $g_{c(t)}$ arising from a polynomial discrete dynamical system in Section \ref{sec5}, such that the associated Galois group $G_{g_{c(t)}}$ is not equal to the symmetric group $S_{\textnormal{deg}(g)}$. To that end, we (also assuming second part of Theorem \ref{5.3}) first import the setup in \cite{Kal}, however, again for the sake of being neat in our discussion, we also invoke a different notation for the set in consideration and also for the probability in question. That is, for any fixed prime $p\geq 5$ and fixed integer $r=(p-1)^{n\ell}$, we let $\mathcal{E}_{(r, \leq 1)}$ (resp. $\mu (\mathcal{E}_{(r, \leq 1)}$) be a set consisting of monic polynomials $g_{c(t)}(x)\in \mathbb{F}_{p}[t][x]$ of degree $r$ in $x$ and degree $\leq 1$ in $t$, such that the associated Galois group $G_{g_{c(t)}}\neq S_{r}$ (resp. be the probability of the set $\mathcal{E}_{(r, \leq 1)}$). But now taking again great advantage of that same probability result of Kaltofen [\cite{Kal}, Page 1], we then also obtain the following corollary showing that the probability of the set $\mu(\mathcal{E}_{(r, \leq 1)})$ is also positive and moreover bounded below by $p^{-1}$:

\begin{cor}
Let $p\geq 5$ be any fixed prime, and suppose that the second part of Theorem \ref{5.3} holds. Let $\mathcal{E}_{(r, \leq 1)}$ (resp. $\mu_{r} = \mu (\mathcal{E}_{(r, \leq 1)})$ be the set (resp. the probability) defined as before. Then the probability $\mu_{r} \geq p^{-1}$. 
\end{cor}

\begin{proof}
Applying a similar argument as in Proof of Corollary \ref{co18.1}, we then obtain the conclusion, as required.
\end{proof}

As before, we also note that for any fixed prime $p\geq 5$ and fixed $r = (p-1)^{n\ell}$, assuming the second part of Theorem \ref{5.3} and letting $\mathcal{E}_{(r, \leq 1)}^{\perp}$ (resp. $\mu (\mathcal{E}_{(r, \leq 1)}^{\perp}$) be a set consisting of monic polynomials $g_{c(t)}(x)\in \mathbb{F}_{p}[t][x]$ of degree $r$ in $x$ and degree $\leq 1$ in $t$, such that the Galois group $G_{g_{c(t)}}= S_{r}$ (resp. be the probability of $\mathcal{E}_{(r, \leq 1)}^{\perp}$), we then (from applying work of Kaltofen \cite{Kal}) also obtain the following corollary on the probability of $\mathcal{E}_{(r, \leq 1)}^{\perp}$: 

\begin{cor}
Let $p\geq 5$ be any fixed prime, and suppose that the second part of Theorem \ref{5.3} holds. Let $\mathcal{E}_{(r, \leq 1)}^{\perp}$ (resp. $\mu_{r}^{\perp} = \mu (\mathcal{E}_{(r, \leq 1)}^{\perp})$ be the set (resp. the probability) defined as before. Then the probability $\mu_{r}^{\perp} = 1-O(p^{-1})$. 
\end{cor}

\begin{proof}
Applying a similar argument as in Proof of Corollary \ref{co18.2}, we then obtain the conclusion, as required.
\end{proof}

\section{On Number of Intermediate fields $L$ of an Extension $\mathcal{K}_{f_{c(t)}}\slash \mathbb{F}_{p}(t)$ and  $\Tilde{L}$ of $\mathcal{L}_{g_{c(t)}}\slash \mathbb{F}_{p}(t)$}\label{sec19}

As in Section \ref{sec18}, again recall that the second part of Theorem \ref{4.3} (i.e., the part in which $N_{c(t)}^{(n)}(\pi, p) = 0$ for every $c\not \equiv 0\ (\text{mod} \ \pi)$) implies that the monic polynomial $f_{c(t)}(x) = \varphi_{p^{\ell},c}^n(x)-x\in \mathbb{F}_{p}[t][x]$ is irreducible modulo prime $\pi$; and so it then follows (from standard theory about irreducibility in $\mathbb{F}_{p}[t][x]$) that $f_{c(t)}(x)$ is irreducible in $\mathbb{F}_{p}[t][x]$. Similarly, we again note that the second part of Theorem \ref{5.3} (i.e., the part in which $M_{c(t)}^{(n)}(\pi, p) = 0$ for every $c\not \equiv \pm1, 0\ (\text{mod} \ \pi)$) also implies that the monic polynomial $g_{c(t)}(x) = \varphi_{(p-1)^{\ell},c}^n(x)-x\in \mathbb{F}_{p}[t][x]$ is irreducible modulo prime $\pi$; and so it then also follows that $g_{c(t)}(x)$ is irreducible in $\mathbb{F}_{p}[t][x]$. Now since $\mathbb{F}_{p}[t]\hookrightarrow \mathbb{F}_{p}(t)$ is an inclusion of rings and so viewing each coefficient $c(t)$ of $f_{c(t)}(x)$ as an element in $\mathbb{F}_{p}(t)$, we may then to each such irreducible monic polynomial $f_{c(t)}$ associate a field $\mathcal{K}_{f_{c(t)}}= \mathbb{F}_{p}(t)[x]\slash (f_{c(t)}(x))$. Similarly, viewing each coefficient $c(t)$ of $g_{c(t)}(x)$ as an element in $\mathbb{F}_{p}(t)$, we may to each such irreducible monic polynomial $g_{c(t)}$ also associate a field $\mathcal{L}_{g_{c(t)}}= \mathbb{F}_{p}(t)[x]\slash (g_{c(t)}(x))$. It then follows from standard theory of algebraic extensions of function fields that each of the naturally constructed $\mathcal{K}_{f_{c(t)}}\slash \mathbb{F}_{p}(t)$ and $\mathcal{L}_{g_{c(t)}}\slash \mathbb{F}_{p}(t)$ is an algebraic function field. Moreover, we note that the degree $[\mathcal{K}_{f_{c(t)}}: \mathbb{F}_{p}(t)]=$ deg $f_{c(t)} =p^{n\ell}$, and $[\mathcal{L}_{g_{c(t)}}: \mathbb{F}_{p}(t)]=$ deg $g_{c(t)} =(p-1)^{n\ell}$.  

So now as in [\cite{BK33}, Section 13], we also wish to count the number of subfields $L$ of $\mathcal{K}_{f_{c(t)}}$ containing $\mathbb{F}_{p}(t)$; and count also the number of subfields $\Tilde{L}$ of $\mathcal{L}_{g_{c(t)}}$ containing  $\mathbb{F}_{p}(t)$. To do so, we (as in \cite{BK3}) take a great advantage of [\cite{Jef}, Lem. 6] and obtain the following corollaries on number of subfields $\Tilde{L}$ and $\Tilde{L}$ of function fields $\mathcal{K}_{f_{c(t)}}$ and $\mathcal{L}_{g_{c(t)}}$ induced by $f_{c(t)}$ and $g_{c(t)}$ arising from a polynomial discrete dynamical system in Section \ref{sec4} and \ref{sec5}. To do so, we (as in \cite{BK3,BK33}) define and then determine the behavior of the following counting functions

\begin{equation}\label{N_{d}}
N(\kappa) := \# \Bigl\{L\slash \mathbb{F}_{p}(t) : L\subset \mathcal{K}_{f_{c(t)}} \text{ is a subfield and } [\mathcal{K}_{f_{c(t)}} : \mathbb{F}_{p}(t)] = \kappa  \Bigr\}
\end{equation} 

\begin{equation}\label{M_{l}}
M(r) := \# \Bigl\{\tilde{L}\slash \mathbb{F}_{p}(t) : \tilde{L}\subset \mathcal{L}_{g_{c(t)}} \text{ is a subfield and } [\mathcal{L}_{g_{c(t)}} : \mathbb{F}_{p}(t)] = r  \Bigr\}.
\end{equation}

\begin{cor}\label{16.1}
Fix $\mathbb{F}_{p}(t)$, and assume second part of Theorem \ref{4.3}. Let $N(\kappa)$ be defined as in \textnormal{(\ref{N_{d}})}. Then    
\begin{equation}\label{N(k)}
N(\kappa)\leq \kappa2^{\kappa!}, \text{ where } \kappa!\sim \frac{\kappa^\kappa}{e^\kappa}\sqrt{2\pi \kappa} \text{ as } \kappa\to \infty.
\end{equation}
\end{cor}
\begin{proof}
Since we know from earlier discussion in this section that the second part of Theorem \ref{4.3} induces irreducible polynomials $f_{c(t)}(x) = \varphi_{p^{\ell},c}^n(x)-x\in \mathbb{F}_{p}[t][x]$ of degree $\kappa=p^{n\ell}$, and so induces a function field $\mathcal{K}_{f_{c(t)}}=\mathbb{F}_{p}(t)[x]\slash (f_{c(t)}(x))$ for every such $f_{c(t)}$, it then follows that the set of function fields $\mathcal{K}_{f_{c(t)}}$ of degree $\kappa=p^{n\ell}$ is not empty. But now setting $K = \mathcal{K}_{f_{c(t)}}$ and $k = \mathbb{F}_{p}(t)$ and so $[K: k] = \kappa$, and then applying [\cite{Jef}, Lemma 6] on the extension $K \supset k$ of function fields, we then immediately obtain inequality (\ref{N(k)}), as needed.
\end{proof}
Similarly, we also have the following corollary on the number of subfields $\Tilde{L}$ of $\mathcal{L}_{g_{c(t)}}$ such that $\Tilde{L}\supset\mathbb{F}_{p}(t)$:
\begin{cor}
Fix $\mathbb{F}_{p}(t)$, and assume second part of Theorem \ref{5.3}. Let $M(r)$ be defined as in \textnormal{(\ref{M_{l}})}. Then 
\begin{equation}\label{M(r)}
M(r)\leq r2^{r!}, \text{ where } r!\sim \frac{r^r}{e^r}\sqrt{2\pi r} \text{ as } r\to \infty.
\end{equation}
\end{cor}
\begin{proof}
Applying a similar argument as in Proof of Corollary \ref{16.1}, we then obtain inequality (\ref{M(r)}), as needed.
\end{proof} 

\section{On the Vanishing of Zeta Functions $\zeta_{\mathcal{K}_{f_{c(t)}}}$ and $\zeta_{\mathcal{L}_{g_{c(t)}}}$ induced by Fields $\mathcal{K}_{f_{c(t)}}$ \& $\mathcal{L}_{g_{c(t)}}$}\label{sec20}

As in Section \ref{sec16}, recall from function field number theory that for every global function field $\mathcal{K}$, we have a zeta function $\zeta_{\mathcal{K}}$ associated with $\mathcal{K}$; and which as with the zeta function $\zeta_{K}(s)$ (in Equation (\ref{Eqn9})), this zeta function $\zeta_{\mathcal{K}}(s)$ can also be defined by an Euler product. (See [\cite{Rose}, pp. 51-52 ] for the Euler product definition of $\zeta_{\mathcal{K}}(s)$).

So now, for every algebraic function field $\mathcal{K}_{f_{c(t)}}= \mathbb{F}_{p}(t)[x]\slash (f_{c(t)}(x))$ induced by an irreducible monic degree-$\kappa$ polynomial $f_{c(t)}\in\mathbb{F}_{p}[t][x]$ arising from a polynomial discrete dynamical system in Section \ref{sec4}, we then also wish to associate to this global function field $\mathcal{K}_{f_{c(t)}}$ a zeta function $\zeta_{\mathcal{K}_{f_{c(t)}}}$. Now motivated by established function field number-theoretic work on the non-trivial vanishing of zeta functions of arbitrary global function fields, we in this section also wish to understand the non-trivial vanishing of the zeta function $\zeta_{\mathcal{K}_{f_{c(t)}}}$. By taking great advantage of a theorem of Andr\'{e} Weil (1948) (restated in [\cite{Rose}, Theorem 5.10]) on the Riemann Hypothesis for global function fields, we then obtain the following corollary on non-trivial vanishing of zeta function $\zeta_{\mathcal{K}_{f_{c(t)}}}$: 

\begin{cor}\label{RH 17.1}
Let $p\geq 3$ be any fixed prime, and assume second part of Theorem \ref{4.3}. Let $\mathcal{K}_{f_{c(t)}}$ be a global function field as before. Then all the non-trivial zeros of the zeta function $\zeta_{\mathcal{K}_{f_{c(t)}}}(s)$ lie on the line $\mathfrak{R}(s)=1\slash2$. 
\end{cor}

\begin{proof}
Since we know from earlier discussion in this section that the second part of Theorem \ref{4.3} induces irreducible monic polynomials $f_{c(t)}(x) = \varphi_{p^{\ell},c}^n(x)-x\in \mathbb{F}_{p}[t][x]\subset \mathbb{F}_{p}(t)[x]$ of odd degree $\kappa=p^{n\ell}$, and hence induces a global function field $\mathcal{K}_{f_{c(t)}}=\mathbb{F}_{p}(t)[x]\slash (f_{c(t)}(x))$ for every such irreducible degree-$\kappa$ polynomial $f_{c(t)}$. But now applying a theorem of Weil (also restated in [\cite{Rose}, Theorem 5.10]) on the zeta function $\zeta_{\mathcal{K}_{f_{c(t)}}}(s)$ corresponding to $\mathcal{K}_{f_{c(t)}}$, we then obtain the conclusion; which then completes the whole proof, as required. 
\end{proof}

Motivated by a well-known prime number theorem for global function fields, we then obtain immediately the following corollary as a consequence of the Riemann Hypothesis for a function field $\mathcal{K}_{f_{c(t)}}$ in Corollary \ref{RH 17.1}:
\begin{cor}\label{PN18.3}
Let $p\geq 3$ be any fixed prime integer, and $m\geq 1$ be any fixed integer. Suppose the second part of Theorem \ref{4.3} holds, and let $\mathcal{K}_{f_{c(t)}}\slash \mathbb{F}_{p}(t)$ be a global function field defined as before. Then the number 
\begin{equation}\label{pnt1}
    \#\Bigl\{\pi\in  \mathcal{K}_{f_{c(t)}} :  \pi \textnormal{ is a prime of deg} (\pi) = m  \Bigr\} = \frac{p^m}{m}+O\biggl(\frac{p^{\frac{m}{2}}}{m}\biggl).
\end{equation}
\end{cor}

\begin{proof}
By applying the prime number theorem for global function fields (restated also as [\cite{Rose}, Theorem 5.12]) on a global function function $\mathcal{K}_{f_{c(t)}}$, we then obtain the count in (\ref{pnt1}) and thus completing the proof, as desired. 
\end{proof}

Similarly, for every algebraic function field $\mathcal{L}_{g_{c(t)}}= \mathbb{F}_{p}(t)[x]\slash (g_{c(t)}(x))$ induced by an irreducible monic degree-$(p-1)^{n\ell}$ polynomial $g_{c(t)}\in\mathbb{F}_{p}[t][x]$ arising from a polynomial discrete dynamical system in Section \ref{sec5}, we also have a zeta function $\zeta_{\mathcal{L}_{g_{c(t)}}}$ corresponding to $\mathcal{L}_{g_{c(t)}}$. So now, by again taking great advantage of Weil's theorem [\cite{Rose}, Theorem 5.10], we then also obtain the following corollary on non-trivial vanishing of $\zeta_{\mathcal{L}_{g_{c(t)}}}(s)$:

\begin{cor}\label{RH 17.2}
Let $p\geq 5$ be any fixed prime, and assume second part of Theorem \ref{5.3}. Let $\mathcal{L}_{g_{c(t)}}$ be a global function field as before. Then all the non-trivial zeros of the zeta function $\zeta_{\mathcal{L}_{g_{c(t)}}}(s)$ lie on the line $\mathfrak{R}(s)=1\slash 2$. 
\end{cor}

\begin{proof}
Because of the hypothesis that the second part of Theorem \ref{5.3} holds, which as discussed before induces irreducible polynomials $g_{c(t)}(x) = \varphi_{(p-1)^{\ell},c}^n(x)-x\in \mathbb{F}_{p}[t][x]\subset \mathbb{F}_{p}(t)[x]$ of even degree $r=(p-1)^{n\ell}$; and hence induces a global function field $\mathcal{L}_{g_{c(t)}}=\mathbb{F}_{p}(t)[x]\slash (g_{c(t)}(x))$ for every such irreducible degree-$r$ polynomial $g_{c(t)}$. But then applying a theorem of Weil (also restated in [\cite{Rose}, Theorem 5.10]) on the zeta function $\zeta_{\mathcal{L}_{g_{c(t)}}}(s)$ attached to $\mathcal{L}_{g_{c(t)}}$, we then immediately obtain the conclusion; which completes the whole proof, as required.
\end{proof}

As a consequence of the Riemann Hypothesis for a global function field $\mathcal{L}_{g_{c(t)}}$ concluded in Corollary \ref{RH 17.2}, we as before also have here the following corollary on the number of primes $\pi \in \mathcal{L}_{g_{c(t)}}$ of fixed degree $m$:
\begin{cor}
Let $p\geq 5$ be any fixed prime integer, and $m\geq 1$ be any fixed integer. Suppose the second part of Theorem \ref{5.3} holds, and let $\mathcal{L}_{g_{c(t)}}\slash \mathbb{F}_{p}(t)$ be a global function field defined as before. Then the number 
\begin{equation}\label{pnt2}
    \#\Bigl\{\pi\in \mathcal{L}_{g_{c(t)}} :  \pi \textnormal{ is a prime of deg}(\pi) = m  \Bigr\} = \frac{p^m}{m}+O\biggl(\frac{p^{\frac{m}{2}}}{m}\biggl).
\end{equation}
\end{cor}

\begin{proof}
By applying the prime number theorem for global function fields (restated also as [\cite{Rose}, Theorem 5.12]) on a global function function $\mathcal{L}_{g_{c(t)}}$, we then obtain the count in (\ref{pnt2}) and so completing the proof, as desired.
\end{proof}

\section{On Non-vanishing of Zeta Functions $\zeta_{\mathcal{K}_{f_{c(t)}}}$ and $\zeta_{\mathcal{L}_{g_{c(t)}}}$ induced by Fields $\mathcal{K}_{f_{c(t)}}$ \& $\mathcal{L}_{g_{c(t)}}$}

As in Section \ref{sec20}, motivated by established function field number-theoretic work on non-vanishing of zeta functions of arbitrary global function fields, we in this section wish to study the non-vanishing of $\zeta_{\mathcal{K}_{f_{c(t)}}}$ (resp. $\zeta_{\mathcal{L}_{g_{c(t)}}}$) of a global function field $\mathcal{K}_{f_{c(t)}}$ (resp. $\mathcal{L}_{g_{c(t)}}$) arising from a polynomial discrete dynamical system in Section \ref{sec4} (resp. Section \ref{sec5}). In doing so, we then obtain the following corollary on the non-vanishing of $\zeta_{\mathcal{K}_{f_{c(t)}}}$:

\begin{cor}\label{nv19.1}
Fix any rational function field $\mathbb{F}_{p}(t)$, and assume second part of Theorem \ref{4.3}. Let $\mathcal{K}_{f_{c(t)}}$ be a global function field defined as before. Then the zeta function $\zeta_{\mathcal{K}_{f_{c(t)}}}(s)$ does not vanish on the line $\mathfrak{R}(s)=1$. 
\end{cor}

\begin{proof}
Since we know from earlier discussion in this section that the second part of Theorem \ref{4.3} induces irreducible polynomials $f_{c(t)}(x) = \varphi_{p^{\ell},c}^n(x)-x\in \mathbb{F}_{p}[t][x]\subset \mathbb{F}_{p}(t)[x]$ of degree $\kappa=p^{n\ell}$, and so induces a global function field $\mathcal{K}_{f_{c(t)}}=\mathbb{F}_{p}(t)[x]\slash (f_{c(t)}(x))$ for every such irreducible polynomial $f_{c(t)}$. But now applying [\cite{Rose}, Prop. 5.13] on a function field $K=\mathcal{K}_{f_{c(t)}}$ with $\zeta_{K}(s)=\zeta_{\mathcal{K}_{f_{c(t)}}}(s)$, we then obtain the conclusion, as required. 
\end{proof}

Similarly, we then also have the following consequence on the non-vanishing of the zeta function $\zeta_{\mathcal{L}_{g_{c(t)}}}(s)$:

\begin{cor}
Fix any rational function field $\mathbb{F}_{p}(t)$, and assume second part of Theorem \ref{5.3}. Let $\mathcal{L}_{g_{c(t)}}$ be a global function field defined as before. Then the zeta function $\zeta_{\mathcal{L}_{f_{c(t)}}}(s)$ does not vanish on the line $\mathfrak{R}(s)=1$. 
\end{cor}

\begin{proof}
Applying a similar argument as in Proof of Corollary \ref{nv19.1}, we then obtain the conclusion, as required. 
\end{proof}

\section{On Non-vanishing of Artin $L$-Functions induced by Fields $\mathcal{K}_{f_{c(t)}}\slash \mathbb{F}_{p}(t)$ and $\mathcal{L}_{g_{c(t)}}\slash \mathbb{F}_{p}(t)$}\label{sec22}

As in Section \ref{sec15}, recall from function field number theory (e.g., see [\cite{Rose}, Ch. 9] that for any given Galois extension $\mathcal{L}\slash \mathcal{K}$ of global function fields and for any given representation $\rho : \textnormal{Gal}(\mathcal{L}\slash \mathcal{K}) \to \textnormal{Aut}_{\mathbb{C}}(V)$ of the Galois group $\textnormal{Gal}(\mathcal{L}\slash \mathcal{K})$, we have a corresponding Artin $L$-function $L(s, \rho)$ defined and analytic in $\{ s \in \mathbb{C}: \mathfrak{R}(s)>1 \}$. Moreover, we note $L(s, \rho_{o})=\zeta_{\mathcal{K}}(s)$ if $\rho=\rho_{o}$ is the trivial representation; and also note $L(s, \rho_{\textnormal{reg}})=\zeta_{\mathcal{L}}(s)$ if $\rho=\rho_{\textnormal{reg}}$ is the regular representation; and as before here $\zeta_{\mathcal{K}}(s)$ (resp. $\zeta_{\mathcal{L}}(s)$) is the zeta function of $\mathcal{K}$ (resp. $\mathcal{L}$).

So now, as before recall that the second part of Theorem \ref{4.3} (i.e., the part in which $N_{c(t)}^{(n)}(\pi, p) = 0$ for every coefficient $c\not \equiv 0\ (\text{mod} \ \pi)$) induces irreducible monic polynomials $f_{c(t)}(x) = \varphi_{p^{\ell},c}^n(x)-x\in \mathbb{F}_{p}[t][x]$; and from which we then obtain a field extension $\mathcal{K}_{f_{c(t)}}= \mathbb{F}_{p}(t)[x]\slash (f_{c(t)}(x))$ of odd degree $\kappa=p^{n\ell}$ of global function fields. Now assuming that $\mathcal{K}_{f_{c(t)}}\slash \mathbb{F}_{p}(t)$ is a Galois extension, we then in this section wish to study the Artin $L$-function $L(s, \rho_{\mathcal{K}_{f_{c(t)}}})$ associated to the representation $\rho_{\mathcal{K}_{f_{c(t)}}} : \textnormal{Gal}(\mathcal{K}_{f_{c(t)}}\slash \mathbb{F}_{p}(t)) \to \textnormal{Aut}_{\mathbb{C}}(V)$ induced by irreducible monic polynomials $f_{c(t)}\in \mathbb{F}_{p}[t][x]$ arising from a polynomial discrete dynamical system in Section \ref{sec4}. In doing so, we then obtain here the following corollary on the non-vanishing of the Artin $L$-function $L(s, \rho_{\mathcal{K}_{f_{c(t)}}})$:

\begin{cor}\label{co22.1}
Assume second part of Theorem \ref{4.3}, and suppose $\mathcal{K}_{f_{c(t)}}\slash \mathbb{F}_{p}(t)$ is a Galois extension of function fields inducing Artin $L$-function $L(s, \rho_{\mathcal{K}_{f_{c(t)}}})$. Then $L(s, \rho_{\mathcal{K}_{f_{c(t)}}})$ does not vanish for every $s\in \mathbb{C}$ with $\mathfrak{R}(s)>1$. 
\end{cor}

\begin{proof}
Since we know from earlier discussion in this section that the second part of Theorem \ref{4.3} induces irreducible monic polynomials $f_{c(t)}(x) = \varphi_{p^{\ell},c}^n(x)-x\in \mathbb{F}_{p}[t][x]\subset \mathbb{F}_{p}(t)[x]$ of odd degree $\kappa=p^{n\ell}$; and moreover for every such irreducible monic degree-$\kappa$ polynomial, the quotient $\mathcal{K}_{f_{c(t)}}= \mathbb{F}_{p}(t)[x]\slash (f_{c(t)}(x))$ is a degree-$\kappa$ field extension of $\mathbb{F}_{p}(t)$. This then also means that the set of algebraic extensions $\mathcal{K}_{f_{c(t)}}\slash \mathbb{F}_{p}(t)$ of function fields is not empty. But now since we know by assumption that $\mathcal{K}_{f_{c(t)}}\slash \mathbb{F}_{p}(t)$ is Galois and hence induces Artin $L$-function $L(s, \rho_{\mathcal{K}_{f_{c(t)}}})$, then applying   [\cite{Rose}, Prop. 9.15] on $L(s, \rho_{\mathcal{K}_{f_{c(t)}}})$, we then obtain the conclusion, as required.
\end{proof}

Similarly, recall that the second part of Theorem \ref{5.3} (i.e., the part in which $M_{c(t)}^{(n)}(\pi, p) = 0$ for every coefficient $c\not \equiv \pm, 0\ (\text{mod} \ \pi)$) induces irreducible monic polynomials $g_{c(t)}(x) = \varphi_{(p-1)^{\ell},c}^n(x)-x\in \mathbb{F}_{p}[t][x]$; and from which we then obtain a field extension $\mathcal{L}_{g_{c(t)}}= \mathbb{F}_{p}(t)[x]\slash (g_{c(t)}(x))$ of even degree $r=(p-1)^{n\ell}$ of global function fields. So now, assuming that $\mathcal{L}_{g_{c(t)}}\slash \mathbb{F}_{p}(t)$ is a Galois extension, we may then also wish to study the Artin $L$-function $L(s, \rho_{\mathcal{L}_{g_{c(t)}}})$ associated to the representation $\rho_{\mathcal{L}_{g_{c(t)}}} : \textnormal{Gal}(\mathcal{L}_{g_{c(t)}}\slash \mathbb{F}_{p}(t)) \to \textnormal{Aut}_{\mathbb{C}}(W)$ induced by irreducible monic polynomials $g_{c(t)}\in \mathbb{F}_{p}[t][x]$ arising from a polynomial discrete dynamical system in Section \ref{sec5}. In doing so, we then obtain here the following corollary on the non-vanishing of the Artin $L$-function $L(s, \rho_{\mathcal{L}_{g_{c(t)}}})$:

\begin{cor}
Assume second part of Theorem \ref{5.3}, and suppose $\mathcal{L}_{g_{c(t)}}\slash \mathbb{F}_{p}(t)$ is a Galois extension of function fields inducing Artin $L$-function $L(s, \rho_{\mathcal{L}_{g_{c(t)}}})$. Then $L(s, \rho_{\mathcal{L}_{g_{c(t)}}})$ does not vanish for every $s\in \mathbb{C}$ with $\mathfrak{R}(s)>1$. 
\end{cor}

\begin{proof}
Applying a similar argument as in Proof of Corollary \ref{co22.1}, we then obtain the conclusion, as required.
\end{proof}

\begin{rem}
Observe that if $\rho_{\mathcal{K}_{f_{c(t)}}}$ is the trivial representation of $\textnormal{Gal}(\mathcal{K}_{f_{c(t)}}\slash \mathbb{F}_{p}(t))$, then in this case as we noted earlier that the Artin $L$-function $L(s, \rho_{\mathcal{K}_{f_{c(t)}}})= \zeta_{\mathbb{F}_{p}(t)}(s)$. Moreover, since we know by Weil's theorem [\cite{Rose}, Theorem 5.10] that the zeta function $\zeta_{\mathbb{F}_{p}(t)}(s)$ vanishes on the line $\mathfrak{R}(s)=\frac{1}{2}$, it then also follows that the $L$-function $L(s, \rho_{\mathcal{K}_{f_{c(t)}}})$ vanishes on the line $\mathfrak{R}(s)=\frac{1}{2}$. Similarly, if $\rho_{\mathcal{L}_{g_{c(t)}}}$ is the trivial representation of $\textnormal{Gal}(\mathcal{L}_{g_{c(t)}}\slash \mathbb{F}_{p}(t))$, then $L(s, \rho_{\mathcal{L}_{g_{c(t)}}})= \zeta_{\mathbb{F}_{p}(t)}(s)$; and so as before $L(s, \rho_{\mathcal{L}_{g_{c(t)}}})$ vanishes on the line $\mathfrak{R}(s)=\frac{1}{2}$. Furthermore, if $\rho_{\mathcal{K}_{f_{c(t)}}}$ is the regular representation of $\textnormal{Gal}(\mathcal{K}_{f_{c(t)}}\slash \mathbb{F}_{p}(t))$, then in this case as we noted earlier that the $L$-function $L(s, \rho_{\mathcal{K}_{f_{c(t)}}})= \zeta_{\mathcal{K}_{f_{c(t)}}}(s)$. But then by Corollary \ref{RH 17.1}, it then also follows $L(s, \rho_{\mathcal{K}_{f_{c(t)}}})$ vanishes on the line $\mathfrak{R}(s)=\frac{1}{2}$. Similarly, if $\rho_{\mathcal{L}_{g_{c(t)}}}$ is the regular representation of $\textnormal{Gal}(\mathcal{L}_{g_{c(t)}}\slash \mathbb{F}_{p}(t))$, then $L(s, \rho_{\mathcal{L}_{g_{c(t)}}})= \zeta_{\mathcal{L}_{g_{c(t)}}}(s)$. But then by Corollary \ref{RH 17.2}, it then follows $L(s, \rho_{\mathcal{L}_{g_{c(t)}}})$ vanishes on the line $\mathfrak{R}(s)=\frac{1}{2}$.
\end{rem}

\section{On the Distribution of Dirichlet $L$-Functions induced by Polynomials $f_{c(t)}(1)$ \& $g_{c(t)}(1)$}

As in Section \ref{sec22}, we again recall that the second part of Theorem \ref{4.3} induces irreducible monic polynomials $f_{c(t)}(x) = \varphi_{p^{\ell},c}^n(x)-x\in \mathbb{F}_{p}[t][x]$ of odd degree $\kappa = p^{n\ell}$ in $x$. But now we may observe that setting the variable $x=1$ in $f_{c(t)}(x)$, we then obtain then a monic polynomial $f_{c(t)}(1) = \varphi_{p^{\ell},c}^n(1)-1\in \mathbb{F}_{p}[t]$ of odd degree $\geq 3$ in $t$. 

So now, inspired by pioneering work of Katz-Sarnak \cite{KatSar} on the statistics of families of $L$-functions, and in particular inspired by work of Katz \cite{Katz} on equidistribution of $L$-functions attached to Dirichlet characters induced by monic squarefree polynomials over finite fields, we (as in Section \ref{sec16} however) in the same spirit as in \cite{Katz} wish to study the distribution of Dirichlet $L$-functions induced by squarefree monic polynomials $h(t)\mid f_{c(t)}(1)\in \mathbb{F}_{p}[t]$ arising from irreducible monic polynomials $f_{c(t)}(x)\in \mathbb{F}_{p}[t][x]$ obtained from a polynomial discrete dynamical system in Section \ref{sec4}. With that in mind, we (assuming second part of Theorem \ref{4.3}) wish to first adhere to the setup in \cite{Katz}. That is, for any prime $p\geq 3$ and for any squarefree monic polynomial $h(t)$ dividing $f_{c(t)}(1)\in \mathbb{F}_{p}[t]$ of fixed degree $n\geq 2$, we define a finite \'{e}tale algebra $B_{h}:=\mathbb{F}_{p}[t]\slash (h(t))$ over $\mathbb{F}_{p}$ of degree $n$; and also define (as in \cite{Katz}) $L(\chi_{h}, T)$ to be the Dirichlet $L$-function attached to a character $\chi_{h}:B_{h}^{\times}\to \mathbb{C}^{\times}$, where $B_{h}^{\times}$ is the multiplicative group of $B_{h}$. Note from \cite{Katz} we can extend $\chi_{h}$ to all of $B_{h}$; and thus we do so by defining $\chi_{h} (b):= 0$ if $b\in B_{h}$ is not invertible. (Recall that a character $\chi_{h}$ modulo $h$ is \say{\textit{primitive}} if there is no proper divisor $h^{'}\mid h$ such that $\chi_{h}(F)=1$ whenever $F$ is coprime to $h$ and $F=1$ modulo $h^{'}$; and in this case the modulus $h$ is referred to as the \say{\textit{conductor}} of $\chi_{h}$.) So now, let $\mathfrak{W}_{h}^{\textnormal{(odd)}}$ be a family consisting of Dirichlet $L$-functions $L(\chi_{h}, T)$ attached to primitive \say{odd} characters $\chi_{h}:B_{h}\to \mathbb{C}$. (Note that from \cite{Katz} a character $\chi_{h}$ is  \say{odd} if $\chi_{h}$ is nontrivial on $\mathbb{F}_{p}^{\times}\subset B_{h}^{\times}$). We note (by \cite{Katz}) that the $L$-functions $L(\chi_{h}, T)$ in $\mathfrak{W}_{h}^{\textnormal{(odd)}}$ can be written in terms of unitary matrices; and moreover as $p\to \infty$, Katz \cite{Katz} proved that these corresponding matrices become uniformly distributed in the unitary group. So now, by taking great advantage of a theorem of Katz [\cite{Katz}, Theorem 5.10], we then obtain the following corollary about the equidistribution of a family $\mathfrak{W}_{h}^{\textnormal{(odd)}}$:

\begin{cor}\label{co23.1}
Assume second part of Theorem \ref{4.3}, and let $h(t)$ dividing $f_{c(t)}(1)\in \mathbb{F}_{p}[t]$ be a squarefree monic polynomial of fixed degree $n\geq 2$. Let $\mathfrak{W}_{h}^{\textnormal{(odd)}}$ be a family of $L$-functions as before. The family $\mathfrak{W}_{h}^{\textnormal{(odd)}}$ of $L$-functions corresponding to Dirichlet characters with a fixed squarefree conductor is equidistributed as $p\to \infty$. 
\end{cor}

\begin{proof}
Because of the hypothesis, we then note that for every integer $i\in \mathbb{Z}_{\geq 1}$, let $p_{i}\geq 3$ be the $i$-th prime, and $h_{i}=h_{i}(t)$ dividing $f_{c(t)}(1)\in \mathbb{F}_{p_{i}}[t]$ be a squarefree monic polynomial of fixed degree $n\geq 2$. So now, let's form a sequence $(\mathbb{F}_{p_{i}}, h_{i}$). But then we note that the desired conclusion follows from the equidistribution obtained from applying [\cite{Katz}, Theorem 5.10] on the sequence $(\mathbb{F}_{p_{i}}, h_{i}$). This then completes the whole proof, as needed.
\end{proof}

Similarly, we also recall that the second part of Theorem \ref{5.3} induces irreducible monic polynomials $g_{c(t)}(x) = \varphi_{(p-1)^{\ell},c}^n(x)-x\in \mathbb{F}_{p}[t][x]$ of even degree $r = (p-1)^{n\ell}$ in $x$. But now also observe that setting the variable $x=1$ in $g_{c(t)}(x)$, we then obtain a monic polynomial $g_{c(t)}(1) = \varphi_{(p-1)^{\ell},c}^n(1)-1\in \mathbb{F}_{p}[t]$ of degree $\geq 4$.

So now, we (as in Section \ref{sec16} however) in again the same spirit as in \cite{Katz} also wish to study the distribution of $L$-functions associated to Dirichlet characters induced by squarefree monic polynomials $q(t)\mid g_{c(t)}(1)\in \mathbb{F}_{p}[t]$ arising from irreducible polynomials $g_{c(t)}(x)\in \mathbb{F}_{p}[t][x]$ obtained from a polynomial discrete dynamical system in Section \ref{sec5}. To do so, we (assuming second part of Theorem \ref{5.3}) also wish to first adapt to the setup in \cite{Katz}. That is, for any prime $p\geq 5$ and for any squarefree monic polynomial $q(t)$ dividing $g_{c(t)}(1)\in \mathbb{F}_{p}[t]$ of fixed degree $n\geq 2$, we define a finite \'{e}tale algebra $B_{q(t)}:=\mathbb{F}_{p}[t]\slash (q(t))$ over $\mathbb{F}_{p}$ of degree $n\geq 2$; and also define (as in \cite{Katz}) $L(\chi_{q(t)}, T)$ to be the Dirichlet $L$-function attached to the character $\chi_{q(t)}:B_{q(t)}^{\times}\to \mathbb{C}^{\times}$, where $B_{q(t)}^{\times}$ is the multiplicative group of $B_{q(t)}$. As before, we also note  $\chi_{q(t)}$ can be extended to all of $B_{q(t)}$ and so we also do so by defining $\chi_{q(t)} (b):= 0$ if $b\in B_{q(t)}$ is not invertible. Now let $\mathfrak{W}_{q(t)}^{\textnormal{(odd)}}$ be a family consisting of Dirichlet $L$-functions $L(\chi_{q(t)}, T)$ attached to primitive odd characters $\chi_{q(t)}:B_{q(t)}^{\times}\to \mathbb{C}^{\times}$. By again taking great advantage of [\cite{Katz}, Theorem 5.10], we then also obtain the following corollary about the equidistribution of $L$-functions in $\mathfrak{W}_{q(t)}^{\textnormal{(odd)}}$:

\begin{cor}
Assume second part of Theorem \ref{5.3}, and let $q(t)$ dividing $g_{c(t)}(1)\in \mathbb{F}_{p}[t]$ be a squarefree monic polynomial of fixed degree $n\geq 2$. Let $\mathfrak{W}_{q(t)}^{\textnormal{(odd)}}$ be a family of $L$-functions as before. Then the family $\mathfrak{W}_{q(t)}^{\textnormal{(odd)}}$ of $L$-functions  associated to Dirichlet characters having a fixed squarefree conductor is equidistributed as $p\to \infty$. 
\end{cor}

\begin{proof}
Applying a similar argument as in Proof of Corollary \ref{co23.1}, we then obtain the conclusion, as required.
\end{proof}

Inspired further by that same work of Katz \cite{Katz} on equidistribution of $L$-functions attached to primitive \say{even} Dirichlet characters induced by monic squarefree polynomials over finite fields, we in again the same spirit as in \cite{Katz} also wish to study the distribution of Dirichlet $L$-functions (induced by squarefree monic polynomials $h(t)\mid f_{c(t)}(1)\in \mathbb{F}_{p}[t]$ arising from irreducible polynomials $f_{c(t)}(x)\in \mathbb{F}_{p}[t][x]$ obtained from a polynomial discrete dynamical system in Section \ref{sec4}) attached to primitive even characters. With that in mind, we again follow the setup in \cite{Katz}, by letting $B_{h}$ (resp. $L(\chi_{h}, T)$) be the $\mathbb{F}_{p}$-algebra (resp. the Dirichlet $L$-function attached to a character $\chi_{h}:B_{h}\to \mathbb{C}$) defined as before. So now, let $\mathfrak{W}_{h}^{\textnormal{(even)}}$ be a family consisting of Dirichlet $L$-functions $L(\chi_{h}, T)$ attached to primitive \say{even} characters $\chi_{h}:B_{h}\to \mathbb{C}$. (Note that from \cite{Katz} a character $\chi_{h}$ is  \say{even} if $\chi_{h}$ is trivial on $\mathbb{F}_{p}^{\times}\subset B_{h}^{\times}$). As before, we also note (by \cite{Katz}) that the $L$-functions $L(\chi_{h}, T)$ in $\mathfrak{W}_{h}^{\textnormal{(even)}}$ can also be written in terms of unitary matrices; and moreover in the limit $p\to \infty$, Katz \cite{Katz} proved that these corresponding matrices also become uniformly distributed in the unitary group. Now by taking great advantage of a theorem of Katz [\cite{Katz}, Theorem 6.4], we then obtain the following corollary on equidistribution of $\mathfrak{W}_{h}^{\textnormal{(even)}}$:

\begin{cor}\label{co23.3}
Assume second part of Theorem \ref{4.3}, and suppose $h(t)$ dividing $f_{c(t)}(1)\in \mathbb{F}_{p}[t]$ is a squarefree monic polynomial of fixed degree $n\geq 3$ with a root in $\mathbb{F}_{p}$. Let $\mathfrak{W}_{h}^{\textnormal{(even)}}$ be a family of $L$-functions as before. Then $\mathfrak{W}_{h}^{\textnormal{(even)}}$ of $L$-functions associated to characters having a fixed squarefree conductor is equidistributed as $p\to \infty$. 
\end{cor}

\begin{proof}
Because of the hypothesis, we then note that for every integer $i\in \mathbb{Z}_{\geq 1}$, let $p_{i}\geq 3$ be the $i$-th prime, and $h_{i}$ dividing $f_{c(t)}(1)\in \mathbb{F}_{p_{i}}[t]$ be a squarefree monic polynomial of fixed degree $n\geq 3$ with a root in $\mathbb{F}_{p_{i}}$. So now, we form a sequence ($\mathbb{F}_{p_{i}}, h_{i}$). But then we note that the desired conclusion follows from the equidistribution obtained from applying [\cite{Katz}, Theorem 6.4] on $(\mathbb{F}_{p_{i}}, h_{i}$). This then completes the whole proof, as needed. 
\end{proof}

Similarly, we in again that same spirit as in \cite{Katz} wish to also study the distribution of Dirichlet $L$-functions (induced by squarefree monic polynomials $g_{c(t)}(1)\in \mathbb{F}_{p}[t]$ arising from irreducible monic polynomials $g_{c(t)}(x)\in \mathbb{F}_{p}[t][x]$ obtained from a polynomial discrete dynamical system in Section \ref{sec5}) attached to primitive even characters. With that in mind, we again also follow the setup in \cite{Katz}, by letting $B_{q(t)}$ (resp. $L(\chi_{q(t)}, T)$) be the $\mathbb{F}_{p}$-algebra (resp. the Dirichlet $L$-function attached to a character $\chi_{q(t)}:B_{q(t)}\to \mathbb{C}$ defined as before. So now, let $\mathfrak{W}_{q(t)}^{\textnormal{(even)}}$ be a family consisting of Dirichlet $L$-functions $L(\chi_{q(t)}, T)$ attached to primitive even characters $\chi_{q(t)}:B_{q(t)}\to \mathbb{C}$. Now by taking great advantage of a theorem of Katz [\cite{Katz}, Theorem 6.4], we then immediately obtain the following corollary on the equidistribution of Dirichlet $L$-functions $L(\chi_{q(t)}, T)$ in the family $\mathfrak{W}_{q(t)}^{\textnormal{(even)}}$:

\begin{cor}
Assume second part of Theorem \ref{5.3}, and suppose $q(t)$ dividing $g_{c(t)}(1)\in \mathbb{F}_{p}[t]$ is a squarefree monic polynomial of fixed degree $n\geq 3$ with a root in $\mathbb{F}_{p}$. Let $\mathfrak{W}_{q(t)}^{\textnormal{(even)}}$ be a family of $L$-functions as before. Then $\mathfrak{W}_{q(t)}^{\textnormal{(even)}}$ of $L$-functions associated to characters having a fixed squarefree conductor is equidistributed as $p\to \infty$. 
\end{cor}

\begin{proof}
Applying a similar argument as in Proof of Corollary \ref{co23.3}, we then obtain the conclusion, as required.
\end{proof}

\section*{\textbf{Acknowledgments}}
I'm very grateful to Prof. Ilia Binder and Prof. Arul Shankar, along with Prof. Jacob Tsimerman for everything. I'm also very grateful to Prof. Daniel Litt for his very engaging class on Representation theory of finite groups in the Winter 2026. Any opinions expressed in this article belong solely to the author, Brian Kintu; and should never be taken at all as a reflection of the views of anyone that's been very happily acknowledged by the author.

\bibliography{References}

@article {Poonen,
    AUTHOR = {Poonen, B.},
     TITLE = {The classification of rational preperiodic points of quadratic
              polynomials over {${\bf Q}$}: a refined conjecture},
   JOURNAL = {Math. Z.},
  FJOURNAL = {Mathematische Zeitschrift},
    VOLUME = {228},
      YEAR = {1998},
    NUMBER = {1},
     PAGES = {11--29},}

@article {Russo,
    AUTHOR = {Walde, R. and Russo, P.},
     TITLE = {Rational periodic points of the quadratic function
              {$Q_c(x)=x^2+c$}},
   JOURNAL = {Amer. Math. Monthly},
  FJOURNAL = {American Mathematical Monthly},
    VOLUME = {101},
      YEAR = {1994},
    NUMBER = {4},
     PAGES = {318--331},}

@article {Shalom1,
    AUTHOR = {Eliahou, S. and Fares, Y.},
     TITLE = {Poonen's conjecture and {R}amsey numbers},
   JOURNAL = {Discrete Appl. Math.},
  FJOURNAL = {Discrete Applied Mathematics. The Journal of Combinatorial
              Algorithms, Informatics and Computational Sciences},
    VOLUME = {209},
      YEAR = {2016},
     PAGES = {102--106},}

@article{Shalom2,
    AUTHOR = {S. Eliahou and Y. Fares},
     TITLE = {Some results on the {F}lynn-{P}oonen-{S}chaefer conjecture},
   JOURNAL = {Canadian Mathematical Bullentin},
   VOLUME  = {65(3):598-611},
     YEAR  = {2022}}

@article {Stoll,
    AUTHOR = {Stoll, M.},
     TITLE = {Rational 6-cycles under iteration of quadratic polynomials},
   JOURNAL = {LMS J. Comput. Math.},
  FJOURNAL = {LMS Journal of Computation and Mathematics},
    VOLUME = {11},
      YEAR = {2008},
     PAGES = {367--380},}

@article {Flynn,
    AUTHOR = {Flynn, E. V. and Poonen, Bjorn and Schaefer, Edward F.},
     TITLE = {Cycles of quadratic polynomials and rational points on a
              genus-{$2$} curve},
   JOURNAL = {Duke Math. J.},
  FJOURNAL = {Duke Mathematical Journal},
    VOLUME = {90},
      YEAR = {1997},
    NUMBER = {3},
     PAGES = {435--463},}

@article {Ingram,
    AUTHOR = {Hutz, B. and Ingram, P.},
     TITLE = {On {P}oonen's conjecture concerning rational preperiodic
              points of quadratic maps},
   JOURNAL = {Rocky Mountain J. Math.},
  FJOURNAL = {The Rocky Mountain Journal of Mathematics},
    VOLUME = {43},
      YEAR = {2013},
    NUMBER = {1},
     PAGES = {193--204},}

@article {par2,
    AUTHOR = {Panraksa, C.},
     TITLE = {Rational periodic points of $x^d + c$ and Fermat-Catalan equations},
   JOURNAL = {International Journal of Number Theory.},
  FJOURNAL = {World Scientific},
    VOLUME = {18},
      YEAR = {2022},
    NUMBER = {05},
     PAGES = {1111-1129},}

@article {Narkie,
    AUTHOR = {Narkiewicz, W.},
     TITLE = {On a class of monic binomials},
   JOURNAL = {Proc. Steklov Inst. Math.},
  FJOURNAL = {Proceedings of the Steklov Institute of Mathematics},
    VOLUME = {280},
      YEAR = {2013},
    NUMBER = {suppl. 2},
     PAGES = {S65--S70},}

@article {Doyle,
    AUTHOR = {Doyle, J R. and Faber, X. and Krumm, D.},
     TITLE = {Preperiodic points for quadratic polynomials over quadratic
              fields},
   JOURNAL = {New York J. Math.},
  FJOURNAL = {New York Journal of Mathematics},
    VOLUME = {20},
      YEAR = {2014},
     PAGES = {507--605},}

@article {North,
    AUTHOR = {Northcott, D. G.},
     TITLE = {Periodic points on an algebraic variety},
   JOURNAL = {Ann. of Math. (2)},
  FJOURNAL = {Annals of Mathematics. Second Series},
    VOLUME = {51},
      YEAR = {1950},
     PAGES = {167--177},}

@article {Erkama,
    AUTHOR = {Erkama, T.},
     TITLE = {Periodic orbits of quadratic polynomials},
   JOURNAL = {Bull. London Math. Soc.},
  FJOURNAL = {The Bulletin of the London Mathematical Society},
    VOLUME = {38},
      YEAR = {2006},
    NUMBER = {5},
     PAGES = {804--814},}

@book{Netto,
    AUTHOR = {E. Netto},
     TITLE = {Vorlesungen uber Algebra II},
 PUBLISHER = {Teubner},
      YEAR = {(1900), pp. 222-227},}

@article {Morton,
    AUTHOR = {Morton, P. and Silverman, J H.},
     TITLE = {Rational periodic points of rational functions},
   JOURNAL = {Internat. Math. Res. Notices},
  FJOURNAL = {International Mathematics Research Notices},
      YEAR = {1994},
    NUMBER = {2},
     PAGES = {97--110},}

@article {Hutz,
    AUTHOR = {Hutz, B.},
     TITLE = {Determination of all rational preperiodic points for morphisms
              of {PN}},
   JOURNAL = {Math. Comp.},
  FJOURNAL = {Mathematics of Computation},
    VOLUME = {84},
      YEAR = {2015},
    NUMBER = {291},
     PAGES = {289--308},}

@article {Call,
    AUTHOR = {Call, G S. and Goldstine, S W.},
     TITLE = {Canonical heights on projective space},
   JOURNAL = {J. Number Theory},
  FJOURNAL = {Journal of Number Theory},
    VOLUME = {63},
      YEAR = {1997},
    NUMBER = {2},
     PAGES = {211--243},}

@article {mor,
    AUTHOR = {Morton, P.},
     TITLE = {Arithmetic properties of periodic points of quadratic maps.
              {II}},
   JOURNAL = {Acta Arith.},
  FJOURNAL = {Acta Arithmetica},
    VOLUME = {87},
      YEAR = {1998},
    NUMBER = {2},
     PAGES = {89--102},}

@book{par1,
    AUTHOR = {C. Panraska },
     TITLE = {Arithmetic dynamics of quadratic polynomials and dynamical units, Phd dissertation},
 PUBLISHER = { University of Maryland, College Park},
      YEAR = {(2011), pp. 1-42},}

@article {detto,
    AUTHOR = {Benedetto, R L.},
     TITLE = {Preperiodic points of polynomials over global fields},
   JOURNAL = {J. Reine Angew. Math.},
  FJOURNAL = {Journal f\"{u}r die Reine und Angewandte Mathematik. [Crelle's
              Journal]},
    VOLUME = {608},
      YEAR = {2007},
     PAGES = {123--153},}

@article{Kat,
    AUTHOR = {A. Katok and B. Hasselblatt},
     TITLE = {Introduction to the Modern Theory of Dynamical Systems},
   JOURNAL = {Cambridge University Press},
   VOLUME  = {Vol. 54},
      YEAR = {1995},}

@article {Doy,
    AUTHOR = {Doyle, John R.},
     TITLE = {Preperiodic points for quadratic polynomials with small cycles over quadratic fields},
   JOURNAL = {Math. Z},
  FJOURNAL = {Springer Link},
    VOLUME = {289},
      YEAR = {2018},
    NUMBER = {1-2},
     PAGES = {729--786},}

@article{sch1,
    AUTHOR = {Bhargava, M. and Shankar, A. and Wang, X.},
     TITLE = {Squarefree values of polynomial discriminants {I}},
   JOURNAL = {Invent. Math.},
    VOLUME = {Vol. 228},
      YEAR = {(2022), pp. 1-37},}

@article{Ada,
    AUTHOR = {D. Adam and Y. Fares},
     TITLE = {On two affine-like dynamical systems in a local field},
 PUBLISHER = {J. Number Theory. 132},
    VOLUME = {132},
      YEAR = {(2012), 2892-2906},}

@article{gav,
    AUTHOR = {M. Bhargava},
     TITLE = {The density of discriminants of quartic rings and fields},
 PUBLISHER = {Ann. of Math.},
    VOLUME = {Ann. of Math. 162},
      YEAR = {(2005), 1031–1063},}

@article{as,
    AUTHOR = {J. Brakenhoff, A. Ash and T. Zarrabi},
     TITLE = {Equality of Polynomial and Field Discriminants},
 PUBLISHER = {Invent. math.},
    VOLUME = {Experiment. Math. 16},
      YEAR = {(2007), 367–374},}

@book{BK3,
    AUTHOR = {B. Kintu},
     TITLE = {Counting the number of $\mathbb{Z}_{p}$- and $\mathbb{F}_{p}[t]$-fixed points of a discrete dynamical system with applications from arithmetic statistics, III},
 PUBLISHER = {https://arxiv.org/pdf/2505.24565},
      YEAR = {pp. 1-25},}

@book{BK33,
    AUTHOR = {B. Kintu},
     TITLE = {Counting the number of 2-periodic $\mathbb{Z}_{p}$- and $\mathbb{F}_{p}[t]$-points of a discrete dynamical system with applications from arithmetic statistics, VI},
 PUBLISHER = {https://arxiv.org/pdf/2511.00322},
      YEAR = {},}

@book {Dev,
    AUTHOR = {Devaney, Robert L.},
     TITLE = {An introduction to chaotic dynamical systems},
    SERIES = {Addison-Wesley Studies in Nonlinearity},
   EDITION = {Second},
 PUBLISHER = {Addison-Wesley Publishing Company, Advanced Book Program,
              Redwood City, CA},
      YEAR = {1989},
     PAGES = {xviii+336},
      ISBN = {0-201-13046-7},}

@article{Jef,
    AUTHOR = {Martin Widmer, Jeffrey Lin Thunder and},
     TITLE = {Counting points of fixed degree and given height over function fields},
   JOURNAL = {Bull. London Math. Soc.},
    VOLUME = {45(2013) 283-300},}

@book{BK111,
    AUTHOR = {B. Kintu},
     TITLE = {Counting the number of $n$-periodic integral points of a discrete dynamical system with applications from arithmetic statistics, I},
 PUBLISHER = {https://arxiv.org/pdf/2507.08601},
      YEAR = {pp. 1-18},}

@book{BK222,
    AUTHOR = {B. Kintu},
     TITLE = {Counting the number of $m$-periodic $\mathcal{O}_{K}$-points of a discrete dynamical system with applications from arithmetic statistics, V},
 PUBLISHER = {https://arxiv.org/abs/2508.16393},
      YEAR = {pp. 1-27},}

@article{lem,
    AUTHOR = {F. Thorne, R. J. Lemke Oliver and},
     TITLE = {Upper bounds on number fields of given degree and bounded discriminant},
 PUBLISHER = {Duke Mathematical Journal},
    VOLUME = {Duke Mathematical Journal, Vol. 171, No. 15},
      YEAR = {(2022), pp. 1-11},}

@article{ho,
    AUTHOR = {Ho, W. and Shankar, A. and Varma, I.},
     TITLE = {Odd degree number fields with odd class number},
   JOURNAL = {Duke Math. Journal.},
    VOLUME = {Vol. 167(5)},
      YEAR = {(2018), pp. 1-53},}

@book{Sia,
    AUTHOR = {A. Siad},
     TITLE = {Monogenic fields with odd class number Part {II}: even degree},
 PUBLISHER = {https://arxiv.org/pdf/2011.08842},
      YEAR = {pp. 1-49},}

@article{Nico,
    AUTHOR = {Shankar, A. and S\"{o}dergren, A. and Templier, N.},
     TITLE = {Sato-$\text{T}$ate equidistribution of certain families of $\text{A}$rtin $\textit{L}$-functions},
   JOURNAL = {Forum of Mathematics, Sigma (2019)},
    VOLUME = {Vol.7, e23, 62 pages},
}

@article{Sar,
    AUTHOR = {Sarnak, P. and Shin, S.W. and Templier, N.},
     TITLE = {Families of $\textit{L}$-functions and their symmetry},
   JOURNAL = {Proceedings of Simons Symposia, Families of Automorphic Forms and the Trace Formula},
    VOLUME = {(Springer Verlag, 2016), 531-578},
}

@article{Bjo,
    AUTHOR = {Poonen, B.},
     TITLE = {Squarefree values of multivariable polynomials},
   JOURNAL = {Duke Math. J.},
    VOLUME = {118(2)},
      YEAR = {(2003), 353-373},}

@article{Ram,
    AUTHOR = {K. Ramsay},
     TITLE = {Square-free values of polynomials in one variable over function fields},
   JOURNAL = {Int. Math. Not. IMRN},
   VOLUME  = {1992 (4)},
     YEAR  = {(1992) 97-102},}

@article{Zeev,
    AUTHOR = {Z. Rudnick},
     TITLE = {Square-free values of polynomials over the rational function field},
 PUBLISHER = {J. Number Theory. 135},
    VOLUME = {135},
      YEAR = {(2014), 60-66},}

@article{gav1,
    AUTHOR = {M. Bhargava},
     TITLE = {Galois groups of random integer polynomials and van der Waerden's Conjecture},
   JOURNAL = {Ann. of Math.},
    VOLUME = {201},
      YEAR = {(2025), 339–377},
}

@book {Rose,
    AUTHOR = {Rosen, M.},
     TITLE = {Number theory in function fields},
    SERIES = {Graduate texts in mathematics},
    VOLUME = {210},
 PUBLISHER = {Springer, New York},
      YEAR = {2002},
     PAGES = {x+508},
      ISBN = {978-1-4419-2954-9},
}

@article{AM,
    AUTHOR = {Artin, M. and B. Mazur},
     TITLE = {On periodic points},
   JOURNAL = {Ann. of Math.},
    VOLUME = {81 (1)},
      YEAR = {(1965), pp. 82–99},
}

@book{Kal,
    AUTHOR = {E L. Kaltofen},
     TITLE = {A note on the van der Waerden conjecture on random polynomials with symmetric Galois group for function fields},
 PUBLISHER = {https://arxiv.org/abs/2204.02836v2},
      YEAR = {},}

@article {Katz,
    AUTHOR = {Katz, N M.},
     TITLE = {On a question of Keating and Rudnick about primitive Dirichlet characters with squarefree conductor},
   JOURNAL = {Internat. Math. Res. Notices},
  FJOURNAL = {International Mathematics Research Notices},
      YEAR = {2013},
    NUMBER = {14},
     PAGES = {3221-3249},}

@article{KatSar,
    AUTHOR = {Katz, N M. and P. Sarnak},
     TITLE = {Zeros of zeta functions and symmetry},
   JOURNAL = {Bull. Amer. Math. Soc. (N.S.)},
    VOLUME = {36 (1999)},
      YEAR = {(1999), no. 1, 1–26.},
}

@book{BK3rm,
    AUTHOR = {B. Kintu},
     TITLE = {Counting the number of 1$_{n}$-, 2$_{n}$-, m$_{n}$-preperiodic $\mathbb{Z}$-, $\mathcal{O}_{K}$-, $\mathbb{Z}_{p}$- and $\mathbb{F}_{p}[t]$-points of a discrete dynamical system with applications from arithmetic statistics},
 PUBLISHER = {In preparation},
      YEAR = {},}

@article{DoyPo,
    AUTHOR = {Doyle, J. and B. Poonen},
     TITLE = {Gonality of dynatomic curves and strong uniform boundedness of preperiodic points},
   JOURNAL = {Compos. Math.},
    VOLUME = {156},
      YEAR = {(2020), 733-743},
}
\bibliographystyle{plain}

\noindent Dept. of Math. and Comp. Sciences (MCS), University of Toronto, Mississauga, Canada \newline
\textit{E-mail address:} \textbf{brian.kintu@mail.utoronto.ca}\newline 
\date{\small{\textit{April 4, 2026}}}

\end{document}